\documentclass[11pt,reqno]{amsart}

\newcommand\version{July 27, 2022}

\usepackage[T1]{fontenc}
\usepackage[utf8]{inputenc}
\usepackage{amsmath,amsfonts,amsthm,amssymb,amsxtra}
\usepackage{bbm} 

\usepackage[ngerman, german, english]{babel} 

\usepackage[parfill]{parskip}    
\usepackage{graphicx}
\usepackage{amssymb}
\usepackage{amsmath}
\usepackage{amsthm}
\usepackage{mathabx}
\usepackage{commath}
\usepackage{epstopdf}

\usepackage{amsmath,amsfonts,amsthm,amssymb,amsxtra,dsfont,mathrsfs}
\usepackage{hyperref} 
\usepackage{colordvi}
\usepackage[usenames,dvipsnames]{color}

\newtheorem{theorem}{Theorem}[section]
\newtheorem{proposition}[theorem]{Proposition}
\newtheorem{lemma}[theorem]{Lemma}
\newtheorem{corollary}[theorem]{Corollary}

\theoremstyle{definition}

\newtheorem{assumption}[theorem]{Assumption}

\theoremstyle{remark}

\newtheorem{remark}[theorem]{Remark}

\numberwithin{equation}{section}

\renewcommand{\epsilon}{\varepsilon}

\newcommand{\R}{\mathbb{R}}

\newcommand{\eps}{\varepsilon}
\newcommand{\xl}{{x, \lambda}}
\newcommand{\pxi}{\partial_{x_i}}
\newcommand{\pxj}{\partial_{x_j}}
\newcommand{\pyi}{\partial_{y_i}}
\newcommand{\pyj}{\partial_{y_j}}

\newcommand{\pl}{\partial_{\lambda}}

\DeclareMathOperator{\Hess}{Hess}

\title[Blow-up of solutions of critical elliptic equations]{Blow-up of solutions of critical elliptic equations\\ in three dimensions}

\author{Rupert L. Frank}
\address[Rupert L. Frank]{Mathematisches Institut, Ludwig-Maximilians-Universit\"at M\"unchen, Theresienstr. 39, 80333 M\"unchen, Germany, and Mathematics 253-37, Caltech, Pasa\-de\-na, CA 91125, USA}
\email{r.frank@lmu.de, rlfrank@caltech.edu}

\author{Tobias König}
\address[Tobias König]{Institut de Mathématiques de Jussieu -- Paris Rive Gauche	Université de Paris - Campus des Grands Moulins, Bâtiment Sophie Germain, Boite Courrier 7012,	8 Place Aurélie Nemours, 75205 Paris Cedex 13}
\email{koenig@imj-prg.fr}

\author {Hynek Kova\v{r}\'{\i}k}
\address [Hynek Kova\v{r}\'{\i}k]{DICATAM, Sezione di Matematica, Universit\`a degli studi di Brescia, Italy}
\email {hynek.kovarik@unibs.it}

\date{\version}
\thanks{\copyright\, 2021 by the authors. This paper may be reproduced, in its entirety, for non-commercial purposes.\\
	Partial support through US National Science Foundation grants DMS-1363432 and DMS-1954995 (R.L.F.) and through ANR BLADE-JC ANR-18-CE40-002 (T.K.) is acknowledged. T.K. thanks Paul Laurain for several useful discussions. The authors are grateful to Haim Br\'ezis for helpful remarks on a first draft of this manuscript.}

\begin{document}

\begin{abstract}
	We describe the asymptotic behavior of positive solutions $u_\eps$ of the equation $-\Delta u + au = 3\,u^{5-\eps}$ in $\Omega\subset\R^3$ with a homogeneous Dirichlet boundary condition. The function $a$ is assumed to be critical in the sense of Hebey and Vaugon and the functions $u_\eps$ are assumed to be an optimizing sequence for the Sobolev inequality. Under a natural nondegeneracy assumption we derive the exact rate of the blow-up and the location of the concentration point, thereby proving a conjecture of Br\'ezis and Peletier (1989). Similar results are also obtained for solutions of the equation $-\Delta u + (a+\eps V) u = 3\,u^5$ in $\Omega$. 
\end{abstract}

\maketitle

\section{Introduction and main results}

We are interested in the behavior of solutions to certain semilinear elliptic equations that are perturbations of the critical equation
$$
-\Delta U = 3\, U^5
\qquad\text{in}\ \R^3 \,.
$$
It is well-known that all positive solutions to the latter equation are given by
\begin{equation}
	\label{eq:uxl}
	U_{x, \lambda}(y) := \frac{\lambda^{1/2}}{(1 + \lambda^2|y-x|^2)^{1/2}}
\end{equation}
with parameters $x \in \R^3$ and $\lambda > 0$. This equation arises as the Euler--Lagrange equation of the optimization problem related to the Sobolev inequality
$$
\int_{\R^3} |\nabla z|^2 \geq S \left( \int_{\R^3} z^6 \right)^{1/3}
$$
with sharp constant \cite{Rod,Ro,Au,Ta}
$$
S := 3 \left( \frac\pi 2 \right)^{4/3}.
$$

The perturbed equations that we are interested in are posed in a bounded open set $\Omega\subset\R^3$ and involve a function $a$ on $\Omega$ such that the operator $-\Delta +a$ with Dirichlet boundary conditions is coercive. (Later, we will be more precise concerning regularity assumptions on $\Omega$ and $a$.) One of the two families of equations also involves another, rather arbitrary function $V$ on $\Omega$. The case where $a$ and $V$ are constants is also of interest.

We consider solutions $u=u_\epsilon$, parametrized by $\epsilon>0$, to the following two families of equations,
\begin{equation}
	\label{BP-problem}
	\begin{cases}
		- \Delta u + a u = 3\, u^{5-\eps} & \text{ in } \Omega,  \\
		u > 0  & \text{ in } \Omega, \\
		u = 0 & \text{ on } \partial  \Omega
	\end{cases}
\end{equation}
and
\begin{equation}
	\label{equation u}
	\begin{cases}
		- \Delta u + (a + \epsilon V) u = 3\, u^5 & \text{ in } \Omega,  \\
		u > 0  & \text{ in } \Omega, \\
		u = 0 & \text{ on } \partial  \Omega.
	\end{cases}
\end{equation}
While there are certain differences between the problems \eqref{BP-problem} and \eqref{equation u}, the methods used to study them are similar, and we will treat both in this paper. We are interested in the behavior of the solutions $u_\eps$ as $\eps\to 0$, and we assume that in this limit the solutions form a minimizing sequence for the Sobolev inequality. More precisely, for \eqref{equation u} we assume
\begin{equation}
	\label{eq:sobmin}
	\lim_{\epsilon \to 0} \frac{\int_\Omega |\nabla u_\epsilon|^2}{\left( \int_\Omega u_\epsilon^6 \right)^{1/3}} = S \,
\end{equation}
and for \eqref{BP-problem} we assume 
\begin{equation}
\label{eq:sobmineps}
\lim_{\epsilon \to 0} \frac{\int_\Omega |\nabla u_\epsilon|^2}{\left( \int_\Omega u_\epsilon^{6-\eps} \right)^{\frac{2}{6-\eps}}} = S \, .
\end{equation}

For example, when $\Omega=B$ is the unit ball, $a=-\pi^2/4$, and $V=-1$,  then \eqref{equation u} has a solution if and only if  $0<\epsilon<\frac{3\pi^2}{4}$, see \cite[Sec.~1.2]{BrNi}. Note that in this case $\pi^2$ is the first eigenvalue of the operator  $-\Delta$ with Dirichlet boundary conditions  on $\Omega$.

Returning to the general situation, the existence of solutions to \eqref{BP-problem} and \eqref{equation u} satisfying \eqref{eq:sobmin} and \eqref{eq:sobmineps} can be proved via minimization under certain assumptions on $a$ and $V$; see, for instance, \cite{FrKoKo1} for \eqref{equation u}. Moreover, it is not hard to prove, based on the characterization of optimizers in Sobolev's inequality, that these functions converge weakly to zero in $H^1_0(\Omega)$ and that $u_\eps^6$ converges weakly in the sense of measures to a multiple of a delta function; see Proposition \ref{lemma PU + w}. In this sense, the functions $u_\eps$ blow up.

The problem of interest is to describe this blow-up behavior more precisely. This question was advertised in an influential paper by Br\'ezis and Peletier \cite{BrPe}, who presented a detailed study of the case where $\Omega$ is a ball and $a$ and $V$ are constants. For earlier results on \eqref{BP-problem} with $a\equiv 0$, see \cite{AtPe,Bu}. Concerning the case of general open sets $\Omega\subset\R^3$, the Br\'ezis--Peletier paper contains three conjectures, the first two of which concern the blow-up behavior of solutions to the analogues of  \eqref{BP-problem} and \eqref{equation u} in dimensions $N\geq 3$ ($N\geq 4$ for \eqref{equation u}) with $a\equiv 0$. These conjectures were proved independently in seminal works of Han \cite{Ha} and Rey \cite{Re1,Re2}.

In the present paper, under a natural nondegeneracy condition, we prove the third Br\'ezis--Peletier conjecture, which has remained open so far. It concerns the blow-up behavior of solutions of \eqref{BP-problem} for certain nonzero $a$ in the three-dimensional case. We also prove the corresponding result for \eqref{equation u}. This latter result is not stated explicitly  as a conjecture in \cite{BrPe}, but it is contained there in spirit and could have been formulated using the same heuristics. Indeed, it is the version with $a \not \equiv 0$ of the second Brézis--Peletier conjecture, in the same way as, concerning \eqref{BP-problem}, the third conjecture is the $a \not \equiv 0$ version of the first one.

A characteristic feature of the three-dimensional case is the notion of criticality for the function $a$. To motivate this concept, let
$$
S(a) := \inf_{0\not\equiv z\in H^1_0(\Omega)} \frac{\int_\Omega (|\nabla z|^2+ a z^2)}{(\int_\Omega z^6)^{1/3}} \,.
$$
One of the findings of Br\'ezis and Nirenberg \cite{BrNi} is that if $a$ is small (for instance, in $L^\infty(\Omega)$), but possibly nonzero, then $S(a)=S$. This is in stark contrast to the case of dimensions $N\geq 4$ where the corresponding analogue of $S(a)$ (with the exponent $6$ replaced by $2N/(N-2)$) is always strictly below the corresponding Sobolev constant, whenever $a$ is negative somewhere.

This phenomenon leads naturally to the following definition due to Hebey and Vaugon \cite{HeVa}. A continuous function $a$ on $\overline\Omega$ is said to be \emph{critical} in $\Omega$ if $S(a)=S$ and if for any continuous function $\tilde a$ on $\overline\Omega$ with $\tilde a\leq a$ and $\tilde a\not\equiv a$ one has $S(\tilde a)<S(a)$. Throughout this paper we assume that $a$ is critical in $\Omega$.

A key role in our analysis is played by the regular part of the Green's function and its zero set. To introduce these, we follow the sign and normalization convention of \cite{Re2}. Since the operator $-\Delta+a$ in $\Omega$ with Dirichlet boundary conditions is assumed to be coercive, it has a Green's function $G_a$ satisfying, for each fixed $y\in\Omega$,
\begin{equation} \label{Ga-pde}
\left\{
\begin{array}{l@{\quad}l}
-\Delta_x\, G_a(x,y) + a(x)\, G_a(x,y) =  4\pi \, \delta_y & \quad \text{in} \ \ \Omega\,, \\
G_a(\cdot,y) = 0  & \quad \text{on} \ \ \partial\Omega \,.
\end{array}
\right.  
\end{equation}
The regular part $H_a$ of $G_a$ is defined by 
\begin{equation} \label{ha-def}
H_a(x,y) := \frac{1}{|x-y|} - G_a(x,y)\, .
\end{equation}
It is well-known that for each $y\in\Omega$ the function $H_a(\cdot,y)$, which is originally defined in $\Omega\setminus\{y\}$, extends to a continuous function in $\Omega$ and we abbreviate
$$
\phi_a(y) := H_a(y,y) \,.
$$
It was proved by Br\'ezis \cite{Br} that $\inf_{y\in\Omega} \phi_a(y) <0$ implies $S(a) <S$. The reverse implication, which was stated in \cite{Br} as an open problem, was proved by Druet \cite{Dr}. Hence, as a consequence of criticality we have
\begin{equation}
	\label{eq:druet}
	\inf_{y\in\Omega} \phi_a(y) = 0 \,;
\end{equation}
see also \cite{Es} and \cite[Proposition 5.1]{FrKoKo1} for alternative proofs. Note that \eqref{eq:druet} implies, in particular, that each point $x$ with $\phi_a(x)=0$ is a critical point of $\phi_a$.

Let us summarize the setting in this paper. In the sequel we set 
$$
\mathcal N_a := \left\{ x \in \Omega \, : \, \phi_a(x) = 0 \right\}.
$$

\begin{assumption}
	\begin{enumerate}
		\item[(a)] $\Omega\subset\R^3$ is a bounded, open set with $C^2$ boundary
		\item[(b)] $a \in C^{0,1}(\overline{\Omega})\cap C^{2,\sigma}_{\rm loc}(\Omega)$ for some $\sigma>0$ 
		\item[(c)] $a$ is critical in $\Omega$
		\item[(d)]  Any point in $\mathcal N_a$ is a nondegenerate critical point of $\phi_a$, that is, for any $x_0\in\mathcal N_a$, the Hessian $D^2 \phi_a(x_0)$ does not have a zero eigenvalue
	\end{enumerate}
\end{assumption}

Let us briefly comment on these items. Assumptions (a) and (b) are modest regularity assumptions, which can probably be further relaxed with more effort. Concerning assumption (d) we first note that $\phi_a \in C^2(\Omega)$ by Lemma \ref{lemma C2}, and therefore any point in $\mathcal N_a$ is a critical point of $\phi_a$, see  \eqref{eq:druet}. We believe that assumption (d) is `generically' true. (For results in this spirit, but in the noncritical case $a\equiv 0$, see \cite{MiPi}.) The corresponding assumption for $a\equiv 0$ appears frequently in the literature, for instance, in \cite{Re2,dPDoMu}. Assumption (d) holds, in particular, if $\Omega$ a ball and $a$ is a constant, as can be verified by explicit computation.

To leading order, the blow-up behavior of solutions of \eqref{equation u} will be given by the projection of a solution \eqref{eq:uxl} of the unperturbed whole space equation to $H^1_0(\Omega)$. For parameters $x\in\R^3$, $\lambda>0$ we introduce $PU_{x, \lambda} \in H^1_0(\Omega)$ as the unique function satisfying 
\begin{equation} \label{eq-pu}
\Delta PU_{x,\lambda} = \Delta U_{x,\lambda}\ \ \  \text{ in } \Omega, \qquad PU_{x,\lambda} = 0 \ \ \ \text{ on } \partial \Omega \,.
\end{equation}
Moreover, let
$$
T_{x, \lambda} := \text{ span}\, \big\{ PU_{x, \lambda},\  \partial_\lambda PU_{x, \lambda},\  \partial_{x_1} PU_{x, \lambda}\,\ \partial_{x_2} PU_{x, \lambda}\,\ \partial_{x_3} PU_{x, \lambda} \big\}
$$
and let $T_{x, \lambda}^\perp$ be the orthogonal complement of $T_{x,\lambda}$ in $H^1_0(\Omega)$ with respect to the inner product $\int_\Omega \nabla u \cdot \nabla v$. By $\Pi_{x,\lambda}$ and $\Pi_{x,\lambda}^\bot$ we denote the orthogonal projections in $H^1_0(\Omega)$ onto $T_{x,\lambda}$ and $T_{x,\lambda}^\bot$, respectively.

Here are our main results. We begin with those pertaining to equation \eqref{BP-problem} and we first provide an asymptotic expansion of $u_\eps$ with a remainder in $H^1_0(\Omega)$.

\begin{theorem}[Asymptotic expansion of $u_\eps$]
	\label{thm expansionconj}
	Let $(u_\epsilon)$ be a family of solutions to \eqref{BP-problem} satisfying \eqref{eq:sobmineps}. Then there are sequences $(x_\epsilon)\subset\Omega$, $(\lambda_\epsilon)\subset(0,\infty)$, $(\alpha_\epsilon)\subset\R$ and $(r_\eps) \subset T_{x_\eps, \lambda_\eps}^\bot$ such that
	\begin{equation} \label{u-eps-finalconj}
		u_\epsilon =  \alpha_\epsilon \left( PU_{x_\epsilon, \lambda_\epsilon} - \lambda_\epsilon^{-1/2}\, \Pi_{x_\epsilon,\lambda_\epsilon}^\bot (H_a(x_\epsilon, \cdot)- H_0(x_\epsilon, \cdot)) + r_\epsilon \right)
	\end{equation}
	and a point $x_0 \in \Omega$ with $\nabla\phi_a(x_0)=0$ such that, along a subsequence,
	\begin{align}
		|x_\epsilon - x_0| &= o(1) \,, \label{x-xconj} \\
		\lim_{\epsilon \to 0}\,  \epsilon \,  \lambda_\epsilon & = \frac{32}\pi\,\phi_a(x_0) \,, \label{lim eps lambdaconj}\\[3pt]
		\alpha_\epsilon^{4-\eps} & = 1 + \frac\eps 2 \log\lambda_\eps +
		\begin{cases}
			\mathcal O(\lambda_\eps^{-1}) & \text{if}\ \phi_a(x_0)\neq 0 \,, \\
			\frac{64}{3\pi} \, \phi_0(x_0) \,\lambda_\eps^{-1} + o(\lambda_\eps^{-1}) & \text{if}\ \phi_a(x_0) = 0 \,,
		\end{cases}
		\label{alpha-asympconj}\\
		\|\nabla r_\epsilon\|_2&= \begin{cases} \mathcal O(\lambda_\eps^{-1}) & \text{if}\ \phi_a(x_0)\neq 0 \,, \\
			\mathcal O(\lambda_\eps^{-3/2}) & \text{if}\ \phi_a(x_0) = 0 \,.
			\end{cases}
	\end{align}
	Moreover, if $\phi_a(x_0)=0$, then 
\begin{equation}
\label{lim-eps-lambda^2}
	\lim_{\epsilon \to 0}\,  \epsilon \,  \lambda_\epsilon^2 = -32\,a(x_0) \,.
\end{equation}
\end{theorem}

Our second main result concerns the pointwise blow-up behavior, both at the blow-up point and away from it, and, in the special case of constant $a$, verifies the conjecture from \cite{BrPe} under the natural nondegeneracy assumption (d).

\begin{theorem}[Br\'ezis-Peletier conjecture]
	\label{thm BPconj}
	Let $(u_\epsilon)$ be a family of solutions to \eqref{BP-problem} satisfying \eqref{eq:sobmineps}.
	\begin{enumerate}
		\item[(a)] The asymptotics close to the concentration point $x_0$ are given by 
		\[ \lim_{\epsilon \to 0}\,  \epsilon \,  \|u_\eps\|^2_\infty = \lim_{\epsilon \to 0}\,  \epsilon \, |u_\eps(x_\eps)|^2  = \frac{32}\pi \, \phi_a(x_0). \]
		If $\phi_a(x_0)=0$, then 
		\begin{equation}
		\label{lim-eps-u^4}
		\lim_{\epsilon \to 0}\,  \epsilon \,  \|u_\eps\|^4_\infty =  \lim_{\epsilon \to 0}\,  \epsilon \, |u_\eps(x_\eps)|^4 = -32\, a(x_0) \, .
		\end{equation}
		\item[(b)] The asymptotics away from the concentration point $x_0$ are given by
		\[ u_\eps (x) = \lambda_\eps^{-1/2} G_a(x, x_0) + o(\lambda_\eps^{-1/2}) \]
		for every fixed $x \in \Omega \setminus \{x_0\}$. The convergence is uniform for $x$ away from $x_0$.
	\end{enumerate}
\end{theorem}

Strictly speaking, the Br\'ezis--Peletier conjecture in \cite{BrPe} is stated without the criticality assumption (c) on $a$, but rather under the assumption $\phi_a\geq 0$ on $\Omega$. (Note that \cite{BrPe} uses the opposite sign convention for the regular part of the Green's function. Also, their Green's function is normalized to be $\frac{1}{4\pi}$ times ours.) The remaining case, however, is much simpler and can be proved with existing methods. Indeed, by Druet's theorem \cite{Dr}, the inequality $\phi_a\geq 0$ on $\Omega$ is equivalent to $S(a)=S$, and the assumption that $a$ is critical is equivalent to $\min\phi_a=0$. Thus, the case of the Br\'ezis--Peletier conjecture that is not covered by our Theorem \ref{thm BPconj} is that where $\min\phi_a>0$. This case can be treated in the same way as the case $a\equiv 0$ in \cite{Ha,Re1} (or as we treat the case $\phi_a(x_0)>0$). Note that in this case the nondegeneracy assumption (d) is not needed. Whether this assumption can be removed in the case where $\phi_a(x_0)=0$ is an open problem. 

We note that Theorems \ref{thm expansionconj} and \ref{thm BPconj} and, in particular, the asymptotics \eqref{lim-eps-lambda^2} and \eqref{lim-eps-u^4}, hold independently of whether $a(x_0)=0$ or not. We note that $a(x_0)\leq 0$ if $\phi_a(x_0)$ as shown in \cite[Corollary 2.2]{FrKoKo1}.
We are grateful to H.~Br\'ezis (personal communication) for raising the question of whether $a(x_0)=0$ can happen and what the asymptotics of $\lambda_\eps$ resp.~$\|u_\eps\|_\infty$ would be in this case, or whether one can show that $\phi_a(x_0)=0$ implies $a(x_0)<0$. Deciding which alternative holds does not appear to be easy, in particular due to the non-local nature of $\phi_a(x_0)$. Here is a simple observation that may illustrate the expected level of difficulty: In the spirit of \cite[Theorem 2.1 and Corollary 2.2]{FrKoKo1}, $a(x_0)< 0$ would follow if one could exhibit a family of very refined test functions $\eta_{x_0, \lambda}$ such that when $\inf_\Omega \phi_a = \phi_a(x_0) = 0$, the Sobolev quotient defining $S(a)$ satisfies $\mathcal S_a[\eta_{x_0, \lambda}] = S - c_1 a(x_0) \lambda^{-2} - c_2 \lambda^{-\tau} + o(\lambda^{-\tau})$ for some $c_1, c_2 >0$ and $\tau > 2$, say. However, extracting such an explicit term $c_2 \lambda^{-\tau}$ is beyond the precision of both \cite{FrKoKo1} and the present paper.    

We also point out that the conjecture in \cite{BrPe} is formulated with assumption \eqref{eq:sobmin} rather that \eqref{eq:sobmineps}. However, the latter assumption is typically used in the posterior literature dealing with problem \eqref{BP-problem}, see e.g. \cite{Ha, GrPa}, and we follow this convention.

We now turn our attention to the results for the second family of equations, namely \eqref{equation u}. Whenever we deal with that problem, we impose the following additional assumptions; \begin{assumption}
	\begin{enumerate}
		\item[(e)] $a<0$ in $\mathcal N_a$
		\item[(f)] $V\in C^{0,1}(\overline\Omega)$
	\end{enumerate}
\end{assumption}

Again, assumption (f) is a modest regularity assumption, which can probably be further relaxed with more effort. Assumption (e) is not severe, as we know from \cite[Corollary 2.2]{FrKoKo1} that any critical $a$ satisfies $a \leq 0$ on $\mathcal N_a$, see also the above discussion of the question by 
Brezis of whether or not this assumption is automatically satisfied. In particular, it is fulfilled if $a$ is a negative constant.

Let
\begin{align} \label{eq-Q}
	Q_V(x) & := \int_\Omega V(y) \, G_a(x,y)^2  , \qquad x\in\Omega \,.
\end{align}
Again, we first provide an asymptotic expansion of $u_\eps$ with a remainder in $H^1_0(\Omega)$.

\begin{theorem}[Asymptotic expansion of $u_\eps$]
\label{thm expansion}
Let $(u_\epsilon)$ be a family of solutions to \eqref{equation u} satisfying \eqref{eq:sobmin}. Then there are sequences $(x_\epsilon)\subset\Omega$, $(\lambda_\epsilon)\subset(0,\infty)$, $(\alpha_\epsilon)\subset\R$ and $(r_\eps) \subset T_{x_\eps, \lambda_\eps}^\bot$ such that
\begin{equation} \label{u-eps-final}
u_\epsilon =  \alpha_\epsilon \left( PU_{x_\epsilon, \lambda_\epsilon} - \lambda_\epsilon^{-1/2}\, \Pi_{x_\epsilon,\lambda_\epsilon}^\bot (H_a(x_\epsilon, \cdot)- H_0(x_\epsilon, \cdot)) + r_\epsilon \right)
\end{equation}
and a point $x_0 \in \mathcal N_a$ with $Q_V(x_0) \leq 0$ such that, along a subsequence,
\begin{align}
|x_\epsilon - x_0| &= o(\eps^{1/2}) \,, \label{x-x} \\
\phi_a(x_\epsilon) & = o(\epsilon) \,, \label{phi-asymp} \\
\lim_{\epsilon \to 0}\,  \epsilon \,  \lambda_\epsilon & = 4\pi^2\, \frac{|a(x_0)|}{|Q_V(x_0)|} \,, \label{lim eps lambda}\\[3pt]
\alpha_\epsilon & = 1 +  \frac{4}{3\pi^3}\, \frac{\phi_0(x_0) \, |Q_V(x_0)|}{|a(x_0)|}\, \epsilon + o(\eps) \,,  \label{alpha-asymp}\\
\|\nabla r_\epsilon\|_2&= \mathcal O(\epsilon^{3/2})\,.   
\end{align}
If $Q_V(x_0) = 0$, the right side of \eqref{lim eps lambda} is to be interpreted as $\infty$. 
\end{theorem}

The following result concerns the pointwise blow-up behavior.

\begin{theorem}\label{thm BP}
Let $(u_\epsilon)$ be a family of solutions to \eqref{equation u} satisfying \eqref{eq:sobmin}.
\begin{enumerate}
\item[(a)] The asymptotics close to the concentration point $x_0$ are given by 
\[ \lim_{\epsilon \to 0}\,  \epsilon \,  \|u_\eps\|^2_\infty = \lim_{\epsilon \to 0}\,  \epsilon \, |u_\eps(x_\eps)|^2  = 4\pi^2\, \frac{|a(x_0)|}{|Q_V(x_0)|}. \]
If $Q_V(x_0) = 0$, the right side is to be interpreted as $\infty$. 
\item[(b)] The asymptotics away from the concentration point $x_0$ are given by
\[ u_\eps (x) = \lambda_\eps^{-1/2} G_a(x, x_0) + o(\lambda_\eps^{-1/2}) \]
for every fixed $x \in \Omega \setminus \{x_0\}$. The convergence is uniform for $x$ away from $x_0$.
\end{enumerate}
\end{theorem}

Theorems \ref{thm expansionconj} and \ref{thm expansion} state that to leading order the solution is given by a projected bubble $PU_{x_\epsilon,\lambda_\epsilon}$. One of the main points of these theorems, which enters crucially in the proof of Theorems \ref{thm BPconj} and \ref{thm BP}, is the identification of the localization length $\lambda_\eps^{-1}$ of the projected bubble as an explicit constant times $\eps$ (for \eqref{BP-problem} if $\phi_a(x_0)\neq 0$ and for \eqref{equation u} if $Q_V(x_0)<0$) or $\eps^{1/2}$ (for \eqref{BP-problem} if $\phi_a(x_0)=0$ and $a(x_0)\neq 0$).

The fact that the solutions are given to leading order by a projected bubble is a rather general phenomenon, which is shared, for instance, also by the higher dimensional generalizations of \eqref{BP-problem} and \eqref{equation u}. In contrast to the higher dimensional case, however, in order to compute the asymptotics of the localization length $\lambda_\eps^{-1}$, we need to extract the leading order correction to the bubble. Remarkably, this correction is for both problems \eqref{BP-problem} and \eqref{equation u} given by $\lambda_\epsilon^{-1/2}\, \Pi_{x_\epsilon,\lambda_\epsilon}^\bot (H_a(x_\epsilon, \cdot)- H_0(x_\epsilon, \cdot))$. 

In this relation it is natural to wonder whether the above projected bubble $PU_{x,\epsilon}$ can be replaced by a different projected bubble $\widetilde{PU}_{x,\lambda}$, namely where the projection is defined with respect to the scalar product coming from the operator $-\Delta+a$, leading to $(-\Delta +a) \widetilde{PU}_\xl = (-\Delta + a) U_\xl$, $\widetilde{PU}_\xl|_{\partial \Omega} = 0$. Such a choice is probably possible and would even simplify some computations, but it would lead to additional difficulties elsewhere (for instance, in the proofs of Propositions \ref{lemma PU + w} and \ref{prop first expansion conj} our choice allows us to apply the classical results by Bahri and Coron).

Moreover, for both problems the concentration point $x_0$ is shown to satisfy $\nabla\phi_a(x_0)=0$. Here, however, we see an interesting difference between the two problems. Namely, for \eqref{equation u} we also know that $\phi_a(x_0)=0$, whereas we know from \cite[Theorem 2(b)]{dPDoMu} that there are solutions of \eqref{BP-problem} concentrating at any critical point of $\phi_a$, not necessarily in $\mathcal N_a$. (These solutions also satisfy \eqref{eq:sobmin}.)

An asymptotic expansion very similar to that in Theorem \ref{thm expansion} is proved in \cite{FrKoKo1} for energy-minimizing solutions of \eqref{equation u}; see also \cite{FrKoKo2} for the simpler higher-dimensional case. There, we did not assume the nondegeneracy of $D^2\phi_a(x_0)$, but we did assume that $Q_V<0$ in $\mathcal N_a$. Moreover, in the energy minimizing setting we showed that $x_0$ satisfies
$$
Q_V(x_0)^2/|a(x_0)| = \sup_{x\in\mathcal N_a, Q_V(x)<0}Q_V(x)^2/|a(x)| \,,
$$
but this cannot be expected in the more general setting of the present paper.

Before describing the technical challenges that we overcome in our proofs, let us put our work into perspective. In the past three decades there has been an enormous literature on blow-up phenomena of solutions to semilinear equations with critical exponent, which is impossible to summarize. We mention here only a few recent works from which, we hope, a more complete bibliography can be reconstructed. In some sense, the situation in the present paper is the simplest blow-up situation, as it concerns single bubble blow-up of positive solutions in the interior. Much more refined blow-up scenarios have been studied, including, for instance, multi-bubbling, sign-changing solutions or concentration on the boundary under Neumann boundary conditions. For an introduction and references we refer to the books \cite{DrHeRo,He}. In this paper we are interested in the description of the behavior of a given family of solutions. For the converse problem of constructing blow-up solutions in our setting, see \cite{dPDoMu} and also \cite{MuSa}, and for a survey of related results, see \cite{Pi} and the references therein. Obstructions to the existence of solutions in three dimensions were studied in \cite{DrLa}. The spectrum near zero of the linearization of solutions was studied in \cite{GrPa,ChKiLe}. There are also connections to the question of compactness of solutions, see \cite{BrMa,KhMaSc} and references therein.

What makes the critical case in three dimensions significantly harder than the higher-dimen\-sional analogues solved by Han \cite{Ha} and Rey \cite{Re1,Re2} is a certain cancellation, which is related to the fact that $\inf \phi_a =0$. Thus, the term that in higher dimensions completely determines the blow-up vanishes in our case. Our way around this impasse is to iteratively improve our knowledge about the functions $u_\epsilon$. The mechanism behind this iteration is a certain coercivity inequality, due to Esposito \cite{Es}, which we state in Lemma \ref{lemma coercivity}, and a crucial feature of our proof is to apply this inequality repeatedly, at different orders of precision. To arrive at the level of precision stated in Theorems \ref{thm expansionconj} and \ref{thm expansion} two iterations are necessary (plus a zeroth one, hidden in the proof of Proposition \ref{lemma PU + w}).

The first iteration, contained in Sections \ref{sec:firstexpansion} and \ref{sec-Bp-prelim}, is relatively standard and follows Rey's ideas in \cite{Re2} with some adaptions due to Esposito \cite{Es} to the critical case in three dimensions. The main outcome of this first iteration is the fact that concentration occurs in the interior and an order-sharp remainder bound in $H^1_0$ on the remainder $\alpha_\eps^{-1} u_\eps - PU_{x_\epsilon, \lambda_\epsilon}$.

The second iteration, contained in Sections \ref{section refining} and \ref{section refining conj}, is more specific to the problem at hand. Its main outcome is the extraction of the subleading correction $\lambda_\epsilon^{-1/2}\, \Pi_{x_\epsilon,\lambda_\epsilon}^\bot (H_a(x_\epsilon, \cdot)- H_0(x_\epsilon, \cdot))$. Using the nondegeneracy of $D^2\phi_a(x_0)$ we will be able to show in the proof of Theorems \ref{thm expansionconj} and \ref{thm expansion} that $\lambda_\eps$ is proportional to $\eps^{-1}$ (for \eqref{BP-problem} if $\phi_a(x_0)\neq 0$ and for \eqref{equation u} if $Q_V(x_0)<0$) or $\eps^{-1/2}$ (for \eqref{BP-problem} if $\phi_a(x_0)=0$ and $a(x_0)\neq 0$).

The arguments described so far are, for the most part, carried out in $H^1_0$ norm. Once one has completed the two iterations, we apply in Subsections \ref{subsection infty bound w} and \ref{subsection infty bound wconj} a Moser iteration argument in order to show that the remainder $\alpha_\eps^{-1} u_\eps - PU_{x_\epsilon, \lambda_\epsilon}$ is negligible also in $L^\infty$ norm. This will then allow us to deduce Theorems \ref{thm BPconj} and \ref{thm BP}.

As we mentioned before, Theorem \ref{thm expansion} is the generalization of the corresponding theorem in \cite{FrKoKo1} for energy-minimizing solutions. In that previous paper, we also used a similar iteration technique. Within each iteration step, however, minimality played an important role in \cite{FrKoKo1} and we used the iterative knowledge to further expand the energy functional evaluated at a minimizer. There is no analogue of this procedure in the current paper. Instead, as in most other works in this area, starting with \cite{BrPe}, Pohozaev identities now play an important role. These identities were not used in \cite{FrKoKo1}. In fact, in \cite{FrKoKo1} we did not use equation \eqref{equation u} at all and our results there are valid as well for a certain class of `almost minimizers'.

There are five types of Pohozaev-type identities corresponding, in some sense, to the five linearly independent functions in the kernel of the Hessian at an optimizer of the Sobolev inequality on $\R^3$ (resulting from its invariance under multiplication by constants, by dilations and by translations). All five identities will be used to control the five parameters $\alpha_\eps$, $\lambda_\eps$ and $x_\eps$ in \eqref{u-eps-finalconj} and \eqref{u-eps-final}, which precisely correspond to the five asymptotic invariances. In fact, all five of these identities are used in the first iteration and then again in the second iteration. (To be more precise, in the first iteration in the proof of Theorem \ref{thm expansion} it is more economical to only use four identities, since the information from the fifth identity is not particularly useful at this stage, due to the above mentioned cancellation $\phi_a(x_0)=0$.)

Thinking of the five Pohozaev-type identities as coming from the asymptotic invariances is useful, but an oversimplification. Indeed, there are several possible choices for the multipliers in each category, for instance, $u$, $PU_\xl$, $\psi_\xl$ corresponding to multiplication by constants, $y\cdot\nabla u$, $\partial_\lambda PU_\xl$, $\partial_\lambda \psi_\xl$ corresponding to dilations and $\partial_{x_j} u$, $\nabla_{x_j} PU_\xl$, $\nabla_{x_j} PU_\xl$ corresponding to translations. (Here $\psi_\xl$ is a modified bubble defined below in \eqref{definition psi}.) The choice of the multiplier is subtle and depends on the available knowledge at the moment of applying the identity and the desired precision of the outcome. In any case, the upshot is that these identities can be brought together in such a way that they give the final result of Theorems \ref{thm expansionconj} and \ref{thm expansion} concerning the expansion in $H^1_0(\Omega)$. As mentioned before, the desired pointwise bounds in Theorems \ref{thm BPconj} and \ref{thm BP} then follow in a relatively straightforward way using a Moser iteration.

We believe that our techniques are robust enough to derive blow-up asymptotics for \eqref{BP-problem} and \eqref{equation u} in more general situations containing a non-zero weak limit and/or multiple concentration points. Since our main motivation was to solve the Brézis--Peletier conjecture stated for single blow-up in \cite{BrPe}, and to limit the amount of calculations needed, we do not attempt to pursue this further here.

Let us also mention that a problem similar to, but different from \eqref{BP-problem} has been studied  in the recent article \cite{MM} by  similar approach. While the analysis in  \cite{MM}, carried out on a Riemannian manifold $M$ of dimension $n \geq 5$, is rather comprehensive and also treats the case of multiple blow-up points, it does not seem to contain an analogue of the vanishing phenomenon for $\phi_a(x_0)$ nor, as a consequence, of our refined iteration step we described above.

The structure of this paper is as follows. The first part of the paper, consisting of Sections \ref{sec:firstexpansion}, \ref{section refining} and \ref{sec:proofsadd}, is devoted to problem \eqref{equation u}, while the second part, consisting of Sections \ref{sec-Bp-prelim}, \ref{section refining conj} and \ref{sec:proofssubcrit}, is devoted to \eqref{BP-problem}. The two parts are presented in a parallel manner, but the emphasis in the second part is on the necessary changes compared to the first part. The preliminary Sections \ref{sec:firstexpansion} and \ref{sec-Bp-prelim} contain an initial expansion, the subsequent Sections \ref{section refining} and \ref{section refining conj} contain its refinement and, finally, in Sections \ref{sec:proofsadd} and \ref{sec:proofssubcrit} the main theorems presented in this introduction are proved. Some technical results are deferred to two appendices.


\section{Additive case: A first expansion}\label{sec:firstexpansion}

In this and the following section we will prepare for the proof of Theorems \ref{thm expansion} and \ref{thm BP}. 

The main result from this section is the following preliminary asymptotic expansion of the family of solutions $(u_\eps)$.

\begin{proposition}
\label{prop first expansion}
Let $(u_\epsilon)$ be a family of solutions to \eqref{equation u} satisfying \eqref{eq:sobmin}.

Then, up to extraction of a subsequence, there are sequences $(x_\epsilon)\subset\Omega$, $(\lambda_\epsilon)\subset(0,\infty)$, $(\alpha_\epsilon)\subset\R$ and $(w_\eps) \subset T_{x_\eps, \lambda_\eps}^\bot$ such that
\begin{equation}
\label{expansion PU + w}
u_\epsilon = \alpha_\epsilon(PU_{x_\eps, \lambda_\eps} + w_\eps),
\end{equation}
and a point $x_0 \in \Omega$ such that
\begin{equation}
\label{parameters PU + w}
|x_\eps - x_0| = o(1), \quad \alpha_\eps = 1 + o(1), \quad \lambda_\eps \to \infty, \quad \|\nabla w_\eps \|_2= \mathcal O(\lambda^{-1/2}).
\end{equation}
\end{proposition}

This proposition follows to a large extent by an adaptation of existing results in the literature. We include the proof since we have not found the precise statement and since related arguments will appear in the following section in a more complicated setting.

An initial qualitative expansion follows from works of Struwe \cite{St} and Bahri-Coron \cite{BaCo}. In order to obtain the statement of Proposition \ref{prop first expansion}, we then need to show two things, namely, the bound on $\|\nabla w\|$ and the fact that $x_0\in\Omega$. The proof of the bound on $\|\nabla w\|$ that we give is rather close to that of Esposito \cite{Es}. The setting in \cite{Es} is slightly different (there, $V$ is equal to a negative constant and, more importantly, the solutions are assumed to be energy minimizing), but this part of the proof extends to our setting. On the other hand, the proof in \cite{Es} of the fact that $x_0\in\Omega$ relies on the energy minimizing property and does not work for us. Instead, we adapt some ideas from Rey in \cite{Re2}. The proof in \cite{Re2} is only carried out in dimensions $\geq 4$ and without the background $a$, but, as we will see, it extends with some effort to our situation.

We subdivide the proof of Proposition \ref{prop first expansion} into a sequence of subsections. The main result of each subsection is stated as a proposition at the beginning and summarizes the content of the corresponding subsection.


\subsection{A qualitative initial expansion}

As a first important step, we derive the following expansion, which is already of the form of that in Proposition \ref{prop first expansion}, except that all remainder bounds are nonquantitative and the limit point $x_0$ may a priori be on the boundary $\partial \Omega$. 

\begin{proposition}
\label{lemma PU + w}
Let $(u_\epsilon)$ be a family of solutions to \eqref{equation u} satisfying \eqref{eq:sobmin}.

Then, up to extraction of a subsequence, there are sequences $(x_\epsilon)\subset\Omega$, $(\lambda_\epsilon)\subset(0,\infty)$, $(\alpha_\epsilon)\subset\R$ and $(w_\eps) \subset T_{x_\eps, \lambda_\eps}^\bot$ such that \eqref{expansion PU + w} holds and a point $x_0 \in \overline{\Omega}$ such that
\begin{equation}
\label{parameters PU + w prelim}
|x_\eps - x_0| = o(1), \quad \alpha_\eps = 1 + o(1), \quad d_\eps \lambda_\eps \to \infty, \quad \|\nabla w_\eps \|_2= o(1),
\end{equation}
where we denote $d_\eps := d(x_\eps, \partial \Omega)$.
\end{proposition}

\begin{proof}
	We shall only prove that $u_\eps\rightharpoonup 0$ in $H^1_0(\Omega)$. Once this is shown, we can use standard arguments, due to Lions \cite{Lio}, Struwe \cite{St} and Bahri--Coron \cite{BaCo}, to complete the proof of the proposition; see, for instance,~\cite[Proof of Proposition 2]{Re2}.
	
	\emph{Step 1.} We begin by showing that $(u_\eps)$ is bounded in $H^1_0(\Omega)$ and that $\|u_\eps\|_6 \gtrsim 1$. Integrating the equation for $u_\eps$ against $u_\eps$, we obtain
	\begin{equation}
		\label{eq:energyident}
		\int_\Omega \left( |\nabla u_\eps|^2 + (a-\eps V)u_\eps^2\right) = 3 \int_\Omega u_\eps^6
	\end{equation}
	and therefore	
	$$
	3 \left( \int_\Omega u_\eps^6 \right)^{2/3} = \frac{\int_\Omega |\nabla u_\eps|^2}{\left( \int_\Omega u_\eps^6 \right)^{1/3}} + \frac{\int_\Omega (a+\epsilon V) u_\eps^2}{\left( \int_\Omega u_\eps^6 \right)^{1/3}} \,.
	$$
	On the right side, the first quotient converges by \eqref{eq:sobmin} and the second quotient is bounded by H\"older's inequality. Thus, $(u_\eps)$ is bounded in $L^6(\Omega)$. By \eqref{eq:sobmin} we obtain boundedness in $H^1_0(\Omega)$. By coercivity of $-\Delta +a$ in $H^1_0(\Omega)$ and Sobolev's inequality, for all sufficiently small $\eps>0$, the left side in \eqref{eq:energyident} is bounded from below by a constant times $\|u_\eps\|_6^2$. This yields the lower bound on $\|u_\eps\|_6 \gtrsim 1$.
	
	\emph{Step 2.} According to Step 1, $(u_\eps)$ has a weak limit point in $H^1_0(\Omega)$ and we denote by $u_0$ one of those. Our goal is to show that $u_0\equiv 0$. Throughout this step, we restrict ourselves to a subsequence of $\eps$'s along which $u_\eps\rightharpoonup u_0$ in $H^1_0(\Omega)$. By Rellich's lemma, after passing to a subsequence, we may also assume that $u_\eps\to u_0$ almost everywhere. Moreover, passing to a further subsequence, we may also assume that $\|\nabla u_\eps\|$ has a limit. Then, by \eqref{eq:sobmin}, $\|u_\eps\|_6$ has a limit as well and, by Step 1, none of these limits is zero.
	
	We now argue as in the proof of \cite[Proposition 3.1]{FrKoKo1} and note that, by weak convergence,
	$$
	\mathcal T = \lim_{\eps\to 0} \int_\Omega |\nabla (u_\eps-u_0)|^2
	\quad \text{exists and satisfies}\quad \lim_{\eps\to 0} \int_\Omega |\nabla u_\eps|^2 = \int_\Omega |\nabla u_0|^2 + \mathcal T
	$$
	and, by the Br\'ezis--Lieb lemma \cite{BrLi},
	$$
	\mathcal M = \lim_{\eps\to 0} \int_\Omega  (u_\eps-u_0)^6
	\quad \text{exists and satisfies}\quad \lim_{\eps\to 0} \int_\Omega u_\eps^6 = \int_\Omega u_0^6 + \mathcal M \,. 
	$$
	Thus, \eqref{eq:sobmin} gives
	$$
	S\left( \int_\Omega u_0^6 + \mathcal M \right)^{1/3} = \int_\Omega |\nabla u_0|^2 + \mathcal T \,.
	$$
	We bound the left side from above with the help of the elementary inequality
	\begin{equation*}
		\left( \int_\Omega u_0^6 + \mathcal M \right)^{1/3} \leq \left( \int_\Omega u_0^6\right)^{1/3} + \mathcal M^{1/3}
	\end{equation*}
	and, by the Sobolev inequality for $u_\eps-u_0$, we bound the right side from below using
	\begin{equation*}
		\mathcal T \geq S \mathcal M^{1/3} \,.
	\end{equation*}
	Thus,
	$$
	S \left( \int_\Omega u_0^6 \right)^{1/3} \geq \int_\Omega |\nabla u_0|^2 \,.
	$$
	Thus, either $u_0\equiv 0$ or $u_0$ is an optimizer for the Sobolev inequality. Since $u_0$ has support in $\Omega\subsetneq\R^3$, the latter is impossible and we conclude that $u_0\equiv 0$, as claimed.	
\end{proof}

\textbf{Convention.} Throughout the rest of the paper, we assume that the sequence $(u_\epsilon)$ satisfies the assumptions and conclusions from Proposition \ref{lemma PU + w}. We will make no explicit mention of subsequences. Moreover, we typically drop the index $\eps$ from $u_\eps$, $\alpha_\eps$, $x_\eps$, $\lambda_\eps$, $d_\eps$ and $w_\eps$. 


\subsection{Coercivity}
The following coercivity inequality from \cite[Lemma 2.2]{Es} is a crucial tool for us in subsequently refining the expansion of $u_\eps$. It states, roughly speaking, that the subleading error terms coming from the expansion of $u_\eps$ can be absorbed into the leading term, at least under some orthogonality condition.

\begin{lemma}\label{lemma coercivity}
There are constants  $T_*<\infty$ and $\rho>0$ such that for all $x\in\Omega$, all $\lambda>0$ with $d\lambda\geq T_*$ and all $v\in T_{x,\lambda}^\bot$,
\begin{equation} \label{coercivity}
\int_\Omega \left( |\nabla v|^2 + av^2 - 15\, U_{x,\lambda}^4 v^2\right) \geq \rho \int_\Omega |\nabla v|^2 \,.
\end{equation}
\end{lemma}

The proof proceeds by compactness, using the inequality \cite[(D.1)]{Re2}
\begin{equation*}
\int_\Omega \left( |\nabla v|^2 - 15\, U_{x,\lambda}^4 v^2 \right) \geq \frac47 \int_\Omega |\nabla v|^2
\qquad\text{for all}\ v\in T_{x,\lambda}^\bot \,.
\end{equation*}
For details of the proof, we refer to \cite{Es}.

In the following subsection, we use Lemma \ref{lemma coercivity} to deduce a refined bound on $\|\nabla w\|_2$. We will use it again in Section \ref{subsection nabla r} below to obtain improved bounds on the refined error term $\|\nabla r\|_2$, with $r \in T_\xl^\bot$ defined in \eqref{definition r}. 


\subsection{The bound on $\|\nabla w\|_2$}
\label{subsection bound w}

The goal of this subsection is to prove

\begin{proposition}\label{boundw}
	As $\eps\to 0$,
	\begin{equation}
		\label{bound w subsec}
		\|\nabla w\|_2= \mathcal O(\lambda^{-1/2}) + \mathcal O((\lambda d)^{-1}). 
	\end{equation}
\end{proposition}

Using this bound, we will prove in Subsection \ref{subsec bdry conc} that $d^{-1} = \mathcal O(1)$ and therefore the bound in Proposition \ref{boundw} becomes $\| \nabla w\|_2=  \mathcal O(\lambda^{-1/2})$, as claimed in Proposition \ref{prop first expansion}. 

\begin{proof}
	The starting point is the equation satisfied by $w$. Since $-\Delta PU_\xl = -\Delta U_\xl = 3 U_\xl^5$, from \eqref{expansion PU + w} and \eqref{equation u} we obtain
\begin{equation}
\label{equation w} (-\Delta +a) w = - 3 U_\xl^5 + 3 \alpha^4 (PU_\xl + w)^5 - (a + \eps V) PU_\xl - \eps V w. 
\end{equation}
Integrating this equation against $w$ and using $\int_\Omega U_\xl^5 w = (1/3) \int_\Omega \nabla PU_\xl\cdot \nabla w = 0$, we get 
\begin{equation}
\label{esposito eq1} 
\int_\Omega (|\nabla w|^2 + a w^2)  = 3 \alpha^4 \int_\Omega (PU_\xl +w)^5 w  - \int_\Omega (a + \eps V) PU_\xl w  - \int_\Omega \eps V w^2 . 
\end{equation} 
We estimate the three terms on the right hand side separately. 

 The second and third ones are easy: we have by Lemma \ref{lemma Lq norm of U}
\[ \left|\int_\Omega (a + \eps V) PU_\xl w  \right| \lesssim \| w\|_6 \|U_\xl\|_{6/5} \lesssim \lambda^{-1/2} \|\nabla w\|_2\,.   \]
Moreover, 
\[ \left|\int_\Omega \eps V w^2  \right| \lesssim \eps \| w\|_6^2 = o(\|\nabla w\|^2_2) \,. \]
The first term on the right side of \eqref{esposito eq1} needs a bit more care. We write $PU_\xl = U_\xl - \varphi_\xl$ as in Lemma \ref{lemma PU} and expand 
\begin{align*}
& \int_\Omega (PU_\xl +w)^5 w  \\
&= \int_\Omega U_\xl^5 w  + 5 \int_\Omega U_\xl^4 w^2  + \mathcal O\left( \int_\Omega \left( U_\xl^4\, \varphi_\xl |w| + U_\xl^3 (|w|^3 + |w| \varphi_\xl^2) + \varphi_\xl^5 |w| + w^6 \right) \right) \\
& =  5 \int_\Omega U_\xl^4 w^2  + \mathcal O \left( 
\int_\Omega U_\xl^4 \, \varphi_\xl |w|  + \|\nabla w\|_2\|\varphi_\xl\|_6^2 + \|\nabla w\|_2^3 \right). 
\end{align*}
where we again used $\int_\Omega U_\xl^5 w  = 0$. By Lemmas \ref{lemma Lq norm of U} and \ref{lemma PU}, we have $\|\varphi_\xl\|_6^2 \lesssim (d\lambda)^{-1}$ and
\begin{align*}
\int_\Omega U_\xl^4 \varphi_\xl |w|  &\lesssim \| w\|_6 \|\varphi_\xl\|_\infty \|U_\xl\|_{24/5}^4 \lesssim  \|\nabla w\|_2  (d \lambda)^{-1}.
\end{align*}
Putting all the estimates together, we deduce from \eqref{esposito eq1} that 
\begin{align*}
 \int_\Omega (|\nabla w|^2 + a w^2 - 15 \alpha^4 U^4 w^2)  & = \mathcal O( (d\lambda)^{-1} \|\nabla w\|_2+ \lambda^{-1/2} \|\nabla w\|_2) + o (\|\nabla w\|^2_2)\, .
\end{align*}
Due to the coercivity inequality from Lemma \ref{lemma coercivity}, the left side is bounded from below by a positive constant times $\| \nabla w\|^2_2$.  Thus, \eqref{bound w subsec} follows. 
\end{proof}


\subsection{Excluding boundary concentration}
\label{subsec bdry conc}

The goal of this subsection is to prove
\begin{proposition}
\label{prop bdry concentration}
 $d^{-1} = \mathcal O(1)$. 
\end{proposition}

By integrating the equation for $u$ against $\nabla u$, one obtains the  Pohozaev-type identity
\begin{equation}
\label{pohozaev type u}
- \int_\Omega (\nabla (a+\eps V)) u^2  = \int_{\partial \Omega} n \left( \frac{\partial u}{\partial n} \right)^2  \,.
\end{equation}
Inserting the decomposition $u = \alpha ( PU + w)$, we get 
\begin{align}
\label{bdry conc identity}
\int_{\partial \Omega} n \left( \frac{\partial PU_\xl}{\partial n} \right)^2  & = -\int_{\partial \Omega} n \left( 2 \frac{\partial PU_\xl}{\partial n}   \frac{\partial w}{\partial n} + \left( \frac{\partial w}{\partial n} \right)^2 \right)  \notag \\
& \quad  - \int_\Omega (\nabla (a+ \eps V)) (PU_\xl+w)^2 .
\end{align}

Since $a, V \in C^1(\overline{\Omega})$, the volume integral is bounded by
\begin{equation}
	\label{eq:poho1vol}
	\left| \int_\Omega (\nabla (a+ \eps V)) (PU_\xl+w)^2  \right| \lesssim \|PU_\xl\|_2^2 + \|w\|_2^2 \lesssim \lambda^{-1} + (\lambda d)^{-2},
\end{equation}
where we used \eqref{bound w subsec} and Lemmas \ref{lemma Lq norm of U} and \ref{lemma PU}.

The function $\partial PU_\xl/\partial n$ on the boundary is discussed in Lemma \ref{lemma PU bdry integral}. We now control the function $\partial w/\partial n$ on the boundary.

\begin{lemma}
\label{lemma w bdry integral}
$\int_{\partial \Omega} \left( \frac{\partial w}{\partial n} \right)^2  = \mathcal O(\lambda^{-1} d^{-1}) + o(\lambda^{-1} d^{-2})$. 
\end{lemma}

\begin{proof}
The following proof is analogous to \cite[Appendix C]{Re2}. It relies on the inequality 
\begin{equation}
\label{trace estimate w}
\left\| \frac{\partial z }{\partial n} \right\|^2_{L^2(\partial \Omega)} \lesssim \left\| \Delta z \right\|^2_{L^{3/2}(\Omega)} \qquad \text{ for all } z \in H^2(\Omega) \cap H^1_0(\Omega) \,.
\end{equation}
This inequality is well-known and contained in \cite[Appendix C]{Re2}. A proof can be found, for instance, in \cite{HaWaYa}.

We write equation \eqref{equation w} for $w$ as $-\Delta w = F$ with
\begin{equation}
	\label{eq:eqwrhs}
	F:=3 \alpha^4 (PU_\xl  + w)^5 - 3 U_\xl^5 - (a+ \eps V) (PU_\xl +w) \,.
\end{equation}
We fix a smooth $0 \leq \chi \leq 1$ with $\chi \equiv 0$ on $\{|y| \leq 1/2\}$ and $\chi \equiv 1$ on $\{ |y| \geq 1 \}$ and define the cut-off function
\begin{equation} \label{zeta-def} 
\zeta(y) := \chi \left(\frac{y-x}{d}\right). 
\end{equation}
Then $\zeta w \in H^2(\Omega) \cap H^1_0(\Omega)$ and
\[ -\Delta (\zeta w) = \zeta F - 2 \nabla \zeta \cdot \nabla w - (\Delta \zeta)w \,. \]
The function $F$ satisfies the simple pointwise bound 
\begin{equation}
\label{pointwise bound f}
|F| \lesssim U_\xl^5 + |w|^5 + U_\xl + |w| \,,
\end{equation}
which, when combined with inequality \eqref{trace estimate w}, yields
\begin{align*}
	\left\| \frac{\partial w}{\partial n} \right\|_{L^2(\partial \Omega)}^2 & = \left\| \frac{\partial (\zeta w)}{\partial n} \right\|_{L^2(\partial \Omega)}^2 \lesssim \|\zeta F - 2 \nabla \zeta \cdot \nabla w - (\Delta \zeta)w \|_{3/2}^2 \\
	& \lesssim \|\zeta(U_\xl ^5 + |w|^5 + U_\xl  + |w|)\|_{3/2}^2 + \| |\nabla \zeta| |\nabla w|\|_{3/2}^2 + \|(\Delta \zeta)w \|_{3/2}^2 \,.
\end{align*}

It remains to bound the norms on the right side. The term most difficult to estimate is $\|\zeta w^5 \|_{3/2}$, because $5\cdot 3/2 = 15/2 > 6$, and we shall come back to it later. The other terms can all be estimated using bounds on $\|U\|_{L^p(\Omega \setminus B_{d/2}(x))}$ from Lemma \ref{lemma Lq norm of U}, as well as the bound $\|w\|_6 \lesssim \lambda^{-1/2} + \lambda^{-1} d^{-1}$ from Proposition \ref{boundw}. Indeed, we have
\begin{align*}
	\| \zeta U_\xl ^5 \|_{3/2}^2 & \lesssim \|U_\xl \|_{L^{15/2}(\Omega \setminus B_{d/2}(x))}^{10} \lesssim \lambda^{-5} d^{-6} = o(\lambda^{-1} d^{-2}), \\
	\| \zeta U_\xl  \|_{3/2}^2 & \lesssim \|U_\xl \|_{L^{3/2}(\Omega \setminus B_d)}^2 \lesssim \lambda^{-1} = \mathcal O(\lambda^{-1} d^{-1}), \\
	\| \zeta w\|_{3/2}^2 & \lesssim \|w\|_6^2 \lesssim \lambda^{-1} + \lambda^{-2} d^{-2} = \mathcal O(\lambda^{-1} d^{-1}) + o(\lambda^{-1} d^{-2}), \\
	\| |\nabla \zeta| |\nabla w| \|_{3/2}^2 & \lesssim \|\nabla w\|_2^2 \|\nabla \zeta\|_6^2 \lesssim (\lambda^{-1} + \lambda^{-2} d^{-2}) d^{-1} = \mathcal O(\lambda^{-1} d^{-1}) + o(\lambda^{-1} d^{-2})
\end{align*}
and 
\[ \| (\Delta \zeta) w\|_{3/2}^2 \lesssim \|w\|_6^2 \|\Delta \zeta\|_2^2  \lesssim (\lambda^{-1} + \lambda^{-2} d^{-2}) d^{-1} = \mathcal O(\lambda^{-1} d^{-1}) + o(\lambda^{-1} d^{-2}). \]

In order to estimate the difficult term $\|\zeta w^5 \|_{3/2}$, we multiply the equation $-\Delta w =\! F$ by $\zeta^{1/2} |w|^{1/2} w$ and integrate over $\Omega$ to obtain
\begin{equation}
	\label{eq:moserinput}
	\int_\Omega \nabla (\zeta^{1/2} |w|^{1/2} w) \cdot \nabla w  \leq \int_\Omega |F| \,  \zeta^{1/2} |w|^{3/2}  \,.
\end{equation}
We now note that there are universal constants $c>0$ and $C<\infty$ such that pointwise a.e.
\begin{equation}
	\label{eq:ptwineq}
		\nabla (\zeta^{1/2} |w|^{1/2} w) \cdot\nabla w \geq c  |\nabla (\zeta^{1/4} |w|^{1/4} w)|^2 - C  |w|^{5/2} |\nabla (\zeta^{1/4})|^2.
\end{equation}
Indeed, by repeated use of the product rule and chain rule for Sobolev functions, one finds
	\begin{align*}
		\nabla (\zeta^{1/2} |w|^{1/2} w) \cdot\nabla w & = \frac 32 \left( \frac{4}{5} \right)^2 |\nabla (\zeta^{1/4} |w|^{1/4} w)|^2 + \left( \frac 32 \left( \frac{4}{5} \right)^2 -  \frac 45 \cdot 2 \right) |w|^{5/2} |\nabla (\zeta^{1/4})|^2 \\
		& \quad - \left(  \frac 32 \left( \frac{4}{5} \right)^2 \cdot 2 - \frac 45 \cdot 2 \right) |w|^{1/4} w \nabla (\zeta^{1/4}) \cdot \nabla (\zeta^{1/4} |w|^{1/4} w)  \,.
	\end{align*}
	The claimed inequality \eqref{eq:ptwineq} follows by applying Schwarz's inequality $v_1 \cdot v_2 \geq - \eps |v_1|^2  - \frac{1}{4 \eps}|v_2|^2$ to the cross term on the right side with $\eps > 0$ small enough.
	
As a consequence of \eqref{eq:ptwineq}, we can bound the left side in \eqref{eq:moserinput} from below by
\begin{align*}
	\int_\Omega \nabla (\zeta^{1/2} |w|^{1/2} w) \cdot \nabla w   
	\geq c \int_\Omega |\nabla (\zeta^{1/4} |w|^{1/4} w)|^2  - C \int_\Omega |w|^{5/2} |\nabla (\zeta^{1/4})|^2  \,.
\end{align*}
Thus, by the Sobolev inequality for the function $\zeta^{1/4} |w|^{1/4} w$ and \eqref{eq:moserinput}, we get
\begin{align} \| \zeta w^5\|_{3/2}^2 &= \left(\int_\Omega |\zeta^{1/4} |w|^{1/4} w|^6  \right)^{4/3} \lesssim \left(\int_\Omega |\nabla (\zeta^{1/4} |w|^{1/4} w)|^2  \right)^4 \label{zeta w5 estimate} \nonumber \\
&\lesssim \left( \int_\Omega |w|^{5/2} |\nabla (\zeta^{1/4})|^2  \right) ^4 + \left( \int_\Omega |F| \,  \zeta^{1/2} |w|^{3/2}  \right)^4. 
\end{align}
For the first term on the right side, we have
\begin{align*}
\left( \int_\Omega |w|^{5/2} |\nabla (\zeta^{1/4})|^2  \right) ^4 &\leq \|w\|_6^{10} \left( \int_\Omega  |\nabla (\zeta^{1/4})|^{24/7} \right)^{7/3} \lesssim (\lambda^{-5} + \lambda^{-10} d^{-10}) d^{-1} \\
&= \mathcal O(\lambda^{-1} d^{-1}) + o(\lambda^{-1} d^{-2}). 
\end{align*}
To control the second term on the right side of \eqref{zeta w5 estimate}, we use again the pointwise estimate \eqref{pointwise bound f}. The contribution of the $|w|^5$ term to the second term on the right side of \eqref{zeta w5 estimate} is 
\[ \left( \int_\Omega |w|^{5 + \frac 32}  \zeta^{1/2}   \right)^4 = \left( \int_\Omega ( \zeta^{1/2} w^{5/2}) w^4  \right)^4 \leq \|\zeta w^5\|_{3/2}^2 \|w\|_6^{16} = o(\|\zeta w^5\|_{3/2}^2), \]
which can be absorbed into the left side of \eqref{zeta w5 estimate}. 

For the remaining terms, we have
\begin{align*}
 \left( \int_\Omega |w|^{ 3/2} U_\xl^5  \zeta^{1/2}  \right)^4 &\lesssim \|w\|_6^6 \|U_\xl\|_{L^{20/3}(\Omega \setminus B_{d/2}(x))}^{20} = (\lambda^{-3} + (d \lambda)^{-6}) (\lambda^{-10} d^{-11}), \\
 \left( \int_\Omega |w|^{ 3/2} U_\xl  \zeta^{1/2}   \right)^4 &\lesssim \|w\|_6^6 \|U_\xl\|_{L^{4/3}(\Omega)}^{4} = (\lambda^{-3} + (d \lambda)^{-6}) \lambda^{-2}, \\
 \left( \int_\Omega |w|^{ 5/2}  \zeta^{1/2}   \right)^4 &\lesssim \|w\|_6^{10} = \lambda^{-5} + (d\lambda)^{-10}, 
\end{align*}
all of which is $\mathcal O(\lambda^{-1} d^{-1}) + o(\lambda^{-1} d^{-2})$. This concludes the proof of the bound $\|\zeta w^5\|_{3/2}^2 = \mathcal O(\lambda^{-1} d^{-1}) + o(\lambda^{-1} d^{-2})$, and thus of Lemma \ref{lemma w bdry integral}. 
\end{proof}

It is now easy to complete the proof of the main result of this section.

\begin{proof}[Proof of Proposition \ref{prop bdry concentration}]
	The identity \eqref{bdry conc identity}, together with the bound \eqref{eq:poho1vol} and Lemma \ref{lemma PU bdry integral} (a), yields 
	\[ C \lambda^{-1} \nabla \phi_0(x) = \mathcal O(\lambda^{-1}) + o (\lambda^{-1} d^{-2}) + \mathcal O\left( \left\| \frac{\partial PU_\xl}{\partial n}\right\|_{L^2(\partial \Omega)} \left\| \frac{\partial w}{\partial n} \right\|_{L^2(\partial \Omega)} + \left\| \frac{\partial w}{\partial n} \right\|_{L^2(\partial \Omega)}^2 \right) \]
	for some $C >0$. By Lemmas \ref{lemma PU bdry integral} (c) and \ref{lemma w bdry integral} the last term on the right side is bounded by $\lambda^{-1} d^{-3/2} + o(\lambda^{-1}d^{-2})$, so we get
	$$
	\nabla \phi_0(x) = \mathcal O(d^{-3/2}) + o(d^{-2}) \,.
	$$
	On the other hand, according to \cite[Equation (2.9)]{Re2}, we have $|\nabla \phi_0(x) | \gtrsim d^{-2}$. Hence $d^{-2} =\mathcal O(d^{-3/2}) + o(d^{-2})$, which yields $d^{-1} = \mathcal O(1)$, as claimed.
\end{proof}


\subsection{Proof of Proposition \ref{prop first expansion}}

The existence of the expansion follows from Proposition \ref{lemma PU + w}. Proposition \ref{prop bdry concentration} implies that $d^{-1} = \mathcal O(1)$, which implies that $x_0\in\Omega$. Moreover, inserting the bound $d^{-1} = \mathcal O(1)$ into Proposition \ref{boundw}, we obtain $\| \nabla w\|_2=  \mathcal O(\lambda^{-1/2})$, as claimed in Proposition \ref{prop first expansion}. This completes the proof of the proposition. \qed


\section{Additive case: Refining the expansion}
\label{section refining}

Our goal in this section is to improve the decomposition given in Proposition \ref{prop first expansion}. As in \cite{FrKoKo1}, our goal is to discover that a better approximation to $u_\epsilon$ is given by the function 
\begin{equation}
\label{definition psi}
\psi_\xl := PU_{\xl} - \lambda^{-1/2}\left(H_a(x, \cdot) - H_0(x, \cdot)\right).
\end{equation}
Let us set
\begin{equation}
\label{definition q}
q_\epsilon := w_\epsilon + \lambda_\epsilon^{-1/2} \left(H_a(x_\epsilon, \cdot) - H_0(x_\epsilon, \cdot) \right),
\end{equation}
so that
$$
u_\epsilon = \alpha_\eps \left( \psi_{x_\eps, \lambda_\eps} + q_\eps \right).
$$
As in \cite{FrKoKo1}, we further decompose
\begin{equation} 	\label{q-split}
q_\eps = s_\eps + r_\eps 
\end{equation}
with $s_\eps \in T_{x_\eps, \lambda_\eps}$ and $r_\eps \in T_{x_\eps, \lambda_\eps}^\bot$ given by
\begin{equation}
\label{definition r}
r_\eps := \Pi_{x_\eps, \lambda_\eps}^\perp q
\qquad\text{and}\qquad
s_\eps :=  \Pi_{x_\eps, \lambda_\eps} q \,.
\end{equation}
We note that the notation $r_\eps$ is consistent with the one used in Theorem \ref{thm expansion} since, writing $w_\eps= q_\eps+ \lambda_\eps^{-1/2} \left( H_a(x_\eps,\cdot) - H_0(x_\eps,\cdot) \right)$ and using $w_\eps\in T_{x_\eps, \lambda_\eps}^\bot$, we have
\begin{equation}
	\label{eq:sproj}
	s_\eps = \lambda_\eps^{-1/2}\,  \Pi_{x_\eps, \lambda_\eps} \left( H_a(x_\eps,\cdot) - H_0(x_\eps,\cdot) \right).
\end{equation}

The following proposition summarizes the results of this section. 

\begin{proposition}
\label{prop second expansion}
Let $(u_\epsilon)$ be a family of solutions to \eqref{equation u} satisfying \eqref{eq:sobmin}.

Then, up to extraction of a subsequence, there are sequences $(x_\epsilon)\subset\Omega$, $(\lambda_\epsilon)\subset(0,\infty)$, $(\alpha_\epsilon)\subset\R$, $(s_\eps) \subset T_{x_\eps, \lambda_\eps}$ and $(r_\eps) \subset T_{x_\eps, \lambda_\eps}^\bot$ such that
\begin{equation}
\label{expansion psi + q}
u_\epsilon = \alpha_\epsilon(\psi_{x_\eps, \lambda_\eps} + s_\eps + r_\eps)
\end{equation}
and a point $x_0 \in \Omega$ such that, in addition to Proposition \ref{prop first expansion},
\begin{align}
\|\nabla r_\eps\|_2 &= \mathcal O(\eps \lambda_\eps^{-1/2}) \,, \label{r-eps-bound}\\
\phi_a(x_\eps) &=   a(x_\eps) \pi \lambda_\eps^{-1} - \frac{\eps}{4\pi}\, Q_V(x_\eps) + o(\lambda_\eps^{-1}) +o(\eps) \,, \nonumber \\
\nabla \phi_a(x_\eps) &= \mathcal O(\eps^\mu) \qquad\text{for any}\ \mu<1\,,  \nonumber\\
\lambda_\eps^{-1} &= \mathcal O(\eps) \,, \nonumber \\
\alpha_\eps^4 &= 1 + \frac{64}{3 \pi}\, \phi_0(x_\eps)\,  \lambda_\eps^{-1}  + \mathcal O(\eps \lambda_\eps^{-1}) \,. \nonumber
\end{align}
\end{proposition}

The expansion of $\phi_a(x)$ will be of great importance also in the final step of the proof of Theorem \ref{thm expansion}. Indeed, by using the bound on $|\nabla \phi_a(x)|$ we will show that in fact $\phi_a(x) = o(\lambda^{-1}) + o(\epsilon)$. This allows us to determine $\lim_{\epsilon \to 0} \eps \lambda_\eps$. 

We prove Proposition \ref{prop second expansion} in the following subsections. Again the strategy is to expand suitable energy functionals. 


\subsection{Bounds on $s$}
\label{subsection s}

In this section we record bounds on the function $s$ introduced in \eqref{definition r}, and on the coefficients $\beta,\gamma$ and $\delta_j$ defined by the decomposition 
\begin{equation}
	\label{expansion s}
	s =  \Pi_\xl q =: \lambda^{-1} \beta PU_\xl + \gamma \partial_\lambda PU_\xl + \lambda^{-3} \sum_{i = 1}^3 \delta_i \partial_{x_i} PU_\xl \,. 
\end{equation}
Since $PU_\xl$, $\partial_\lambda PU_\xl$ and $\partial_{x_i} PU_\xl$, $i=1,2,3$, are linearly independent for sufficiently small $\epsilon$, the numbers $\beta$, $\gamma$ and $\delta_i$, $i=1,2,3$, (depending on $\eps$, of course) are uniquely determined. The choice of the different powers of $\lambda$ multiplying these coefficients is motivated by the following proposition.

\begin{proposition}
\label{proposition s}
The coefficients appearing in \eqref{expansion s} satisfy 
\begin{equation}
\label{bound beta gamma delta}  \beta, \gamma, \delta_i = \mathcal O(1).
\end{equation}
Moreover, we have the bounds 
\begin{equation}
\label{bounds s}
\|s\|_\infty = \mathcal O(\lambda^{-1/2}), \quad \| \nabla s\|_2= \mathcal O(\lambda^{-1})\quad \text{and } \quad \|s\|_{2} = \mathcal O(\lambda^{-3/2}),
\end{equation}
as well as
\begin{equation}
\label{bound nabla s outside} \|\nabla s\|_{L^2(\Omega \setminus B_{d/2}(x))} =\mathcal O (\lambda^{-3/2}). 
\end{equation}
\end{proposition}

\begin{proof}
	Because of \eqref{eq:sproj}, $s_\eps$ depends on $u_\eps$ only through the parameters $\lambda$ and $x$. Since these parameters satisfy the same properties $\lambda\to\infty$ and $d^{-1} =\mathcal O(1)$ as in \cite{FrKoKo1}, the results on $s_\eps$ there are applicable. In particular, the bound \eqref{bound beta gamma delta} follows from \cite[Lemma 6.1]{FrKoKo1}. 
	
	The bounds stated in \eqref{bounds s} follow readily from \eqref{expansion s} and \eqref{bound beta gamma delta}, together with the corresponding bounds on the basis functions $PU_\xl$, $\pl PU_\xl$ and $\pxi PU_\xl$, $i=1,2,3$, which come from 
	\[ \| U_\xl\|_\infty \lesssim \lambda^{1/2}, \quad \|\nabla U_\xl\|_2\lesssim 1, \quad   \|U_\xl\|_{2} \lesssim \lambda^{-1/2}, \]
	and similar bounds on $\pl U_\xl$ and $\pxi U_\xl$, compare Lemma \ref{lemma Lq norm of U}, as well as
	\[ \|H_0(x, \cdot)\|_2+\|\nabla_x H_0(x, \cdot)\|_2+ \|\nabla_x \nabla_y H_0(x,y)\|_2 \lesssim 1. \]

	It remains to prove \eqref{bound nabla s outside}. Again by \eqref{expansion s} and \eqref{bound beta gamma delta}, it suffices to show that
	\begin{equation}
		\label{eq:propsproof}
		\lambda^{-1} \|\nabla  PU_\xl\|_{L^2(\Omega \setminus B_{d/2}(x))}
		+ \|\nabla  \pl PU_\xl\|_{L^2(\Omega \setminus B_{d/2}(x))} + \lambda^{-3} \|\nabla \pxi PU_\xl\|_{L^2(\Omega \setminus B_{d/2}(x))}
		 \lesssim \lambda^{-3/2}.
	\end{equation}
	(In fact, there is a better bound on $\nabla \pxi PU_\xl$, but we do not need this.) Since the three bounds in \eqref{eq:propsproof} are all proved similarly, we only prove the second one.
	
	By integration by parts, we have
\[ \int_{\Omega \setminus B_{d/2}(x)} |\nabla \pl PU_\xl|^2 = 15 \int_{\Omega \setminus B_{d/2}(x)} U_\xl^4 \pl U_\xl \pl PU_\xl + \int_{\partial B_{d/2}(x)} \frac{\partial (\pl PU_\xl)}{\partial n} \pl PU_\xl \,. \]
By the bounds from Lemmas \ref{lemma Lq norm of U} and \ref{lemma PU}, the volume integral is estimated by
\begin{align*}
 \int_{\Omega \setminus B_{d/2}(x)} U_\xl^4 \pl U_\xl \pl PU_\xl 
&\leq \int_{\R^3 \setminus B_{d/2}(x)} U_\xl^4 (\pl U_\xl)^2 + \|\pl \varphi_\xl\|_\infty  \int_{\R^3 \setminus B_{d/2}(x)} U_\xl^4 |\pl U_\xl| \\
& \lesssim \lambda^{-5}. 
\end{align*}
Since 
\[  \nabla \pl U_\xl(y) = \frac{\lambda^{3/2}}{2}  \frac{(-5+3\lambda^2|y-x|^2)(y-x) }{(1 + \lambda^2|y-x|^2)^{5/2}},  \]
we find $|\nabla \pl U_\xl| \lesssim \lambda^{-3/2}$ on $\partial B_{d/2}(x)$. By the mean value formula for the harmonic function $\pl \varphi_\xl$ and the bound from Lemma \ref{lemma PU},
\[ |\nabla \pl \varphi_\xl(y)| = \|\pl \varphi_\xl\|_\infty \lesssim \lambda^{-3/2}
\qquad\text{for all}\ y\in\partial B_{d/2}(x) . \]
This implies that $|\nabla (\pl PU_\xl)| \lesssim \lambda^{-3/2}$ on $\partial B_{d/2}(x)$. Thus, the boundary integral is estimated by 
\begin{align*}
\int_{\partial B_{d/2}(x)} \frac{\partial (\pl PU_\xl)}{\partial n} \pl PU_\xl &= \|\nabla (\pl PU_\xl)\|_{L^\infty(\partial B_{d/2}(x))} (\| \pl U_\xl\|_{L^\infty(\Omega \setminus B_{d/2}(x))} + \|\pl \varphi_\xl\|_\infty) \\
& \lesssim \lambda^{-3}  \,,
\end{align*} 
since $\| \pl U_\xl\|_{L^\infty(\Omega \setminus B_{d/2}(x))} \lesssim \lambda^{-3/2}$ by Lemma \ref{lemma Lq norm of U}. Collecting these estimates, we find that
$\|\nabla \pl PU_\xl\|_{L^2(\Omega \setminus B_{d/2}(x))}$ $\lesssim \lambda^{-3/2}$, which is the second bound in \eqref{eq:propsproof}. 
\end{proof}

Later we will also need the leading order behavior of the zero mode coefficients $\beta$ and $\gamma$ in \eqref{expansion s}.

\begin{proposition} \label{prop-beta-gamma} As $\eps\to 0$, 
	\begin{equation}
		\label{eq beta gamma subsec}
		\beta= \frac{16}{3\pi}\,  (\phi_a(x) - \phi_0(x)) + \mathcal{O}(\lambda^{-1}) , \qquad 
		\gamma=-\frac 85\, \beta + \mathcal{O}(\lambda^{-1}).  
	\end{equation}
\end{proposition}

\begin{proof}
	According to \eqref{eq:sproj}, we have 
	\begin{align}
		\label{eq s Ha H0 scalarprod1}
		\int_\Omega \nabla s \cdot \nabla PU_\xl & = \lambda^{-1/2} \int_\Omega \nabla (H_a(x,\cdot) - H_0(x,\cdot))\cdot \nabla PU_\xl, \\
		\label{eq s Ha H0 scalarprod2}
		\int_\Omega \nabla s \cdot \nabla \pl PU_\xl & = \lambda^{-1/2} \int_\Omega \nabla (H_a(x,\cdot) - H_0(x,\cdot)) \nabla \pl PU_\xl.
	\end{align}
	By \eqref{expansion s}, the left side of \eqref{eq s Ha H0 scalarprod1} is 
	\begin{align*}
		& \beta \lambda^{-1} \int_\Omega |\nabla PU_\xl|^2 + \gamma \int_\Omega \nabla \pl PU_\xl\cdot \nabla PU_\xl + \lambda^{-3} \sum_{i=1}^3 \delta_i \int_\Omega \nabla \pxi PU_\xl \cdot \nabla PU_\xl \\
		& = 3 \beta \lambda^{-1} \frac{\pi^2}{4} + \mathcal O(\lambda^{-2}),
	\end{align*}
	where we used the facts that, by \cite[Appendix B]{Re2},
	\begin{align}
		\label{rey scalarprods}
		& \int_\Omega |\nabla PU_\xl|^2 = 3 \frac{\pi^2}{4} + \mathcal O(\lambda^{-1}), \qquad \int_\Omega \nabla \pl PU_\xl \cdot \nabla PU_\xl = \mathcal O(\lambda^{-2}) \\
		& \int_\Omega \nabla \pxi PU_\xl \cdot \nabla PU_\xl = \mathcal O(\lambda^{-1}). 
	\end{align}
	On the other hand, the right side of \eqref{eq s Ha H0 scalarprod1} is 
	\begin{align*}
		\lambda^{-1/2} \int_\Omega \nabla (H_a(x,\cdot) - H_0(x,\cdot)) \cdot \nabla PU 
		& = 3 \lambda^{-1/2} \int_\Omega (H_a(x,\cdot) - H_0(x,\cdot)) U_\xl^5 \\
		& = 4 \pi (\phi_a(x) - \phi_0(x)) \lambda^{-1} + \mathcal O(\lambda^{-2})
	\end{align*}
	by Lemma \ref{lemma U Ha}. Comparing both sides yields the expansion of $\beta$ stated in \eqref{eq beta gamma subsec}. 
	
	Similarly, by \eqref{expansion s}, the left side of \eqref{eq s Ha H0 scalarprod2} is 
	\begin{align*}
		& \frac{\beta}{\lambda^2} \int_\Omega \nabla PU_\xl \cdot \nabla \pl PU_\xl + \gamma \int_\Omega |\nabla \pl PU_\xl|^2 +  \lambda^{-3} \sum_{i=1}^3 \delta_i \int_\Omega \nabla \pxi PU_\xl \cdot \nabla \pl PU_\xl \\
		&= \frac{15 \pi^2  \gamma}{64\,  \lambda^2} + \mathcal O(\lambda^{-3}) \,,
	\end{align*}
	where, besides \eqref{rey scalarprods}, we used $\int_\Omega \nabla \pxi PU_\xl \cdot \nabla \pl PU_\xl  = \mathcal O(\lambda^{-2})$ by \cite[Appendix B]{Re2}, and
	\[
	\int_\Omega |\nabla \pl PU_\xl|^2 = \int_\Omega |\nabla \pl U_\xl|^2 + \mathcal O(\lambda^{-3}) = \frac{15 \pi^2}{64} \lambda^{-2} + \mathcal O(\lambda^{-3}) \,.
	\]
	(The numerical value comes from an explicit evaluation of the integral in terms of beta functions, which we omit.) On the other hand, the right side of \eqref{eq s Ha H0 scalarprod2} is 
	\begin{align*}
		\lambda^{-1/2} \int_\Omega \nabla (H_a(x,\cdot) - H_0(x,\cdot))\cdot \nabla \pl PU_\xl &= 15 \lambda^{-1/2} \int_\Omega (H_a(x,\cdot) - H_0(x,\cdot)) U_\xl^4 \pl U_\xl \\
		& = -2 \pi (\phi_a(x) - \phi_0(x)) \lambda^{-2} + \mathcal O(\lambda^{-3})
	\end{align*}
	by Lemma \ref{lemma U Ha}. Comparing both sides yields the expansion of $\gamma$ stated in \eqref{eq beta gamma subsec}.
\end{proof}


\subsection{The bound on $\|\nabla r\|_2$}
\label{subsection nabla r} The goal of this subsection is to prove 

\begin{proposition} \label{prop-r-bound}
As $\eps\to 0$, 
\begin{equation}
\label{nabla r subsection}
\|\nabla r\|_2= \mathcal O(\phi_a(x) \lambda^{-1}) + \mathcal O(\lambda^{-3/2}) + \mathcal O(\eps \lambda^{-1/2}).
\end{equation} 
\end{proposition}


Using $\Delta (H_a(x,\cdot)-H_0(x,\cdot)) = -a G_a(x,\cdot)$ and introducing the function $g_\xl$ from \eqref{eq:defg}, we see that the equation \eqref{equation w} for $w$ implies
\begin{equation}
\label{equation r}
(-\Delta+ a) r = - 3 U_\xl^5 + 3 \alpha^4 (\psi_\xl + s + r)^5 + a (f_\xl+g_\xl) - as   - \eps V (\psi_\xl + s +r) + \Delta s \,.
\end{equation}
Integrating against $r$ and using the orthogonality conditions $\int_\Omega (\Delta s ) r = -\int_\Omega \nabla s \cdot \nabla r = 0$ and $3\int_\Omega U_\xl^5 r = \int_\Omega \nabla PU_\xl\cdot\nabla r =0$, we obtain
\begin{equation}
\label{energy r}
\int_\Omega \left(|\nabla r|^2+ ar^2\right) = 3 \alpha^4 \int_\Omega (\psi_\xl + s + r)^5 r - \int_\Omega a (s- f_\xl - g_\xl) r - \int_\Omega \eps V (\psi_\xl + s +r) r. 
\end{equation}

The terms appearing in \eqref{energy r} satisfy the following bounds. 

\begin{lemma}
\label{lemma expansion r}
As $\eps \to 0$, the following holds. 
\begin{enumerate}
\item[(a)] $\left| 3 \alpha^4 \int_\Omega (\psi_\xl + s + r)^5 r - 15  \alpha^4 \int_\Omega U_\xl^4 r^2\right| \lesssim \left( \lambda^{-3/2} + \lambda^{-1}\phi_a(x) + \|r\|_6^2 \right)\|r\|_6$. 
\item[(b)] $\left | \int_\Omega \left( a (s- f_\xl - g_\xl ) + \epsilon V(\psi_\xl+s+r)\right) r \right| \lesssim \left( \lambda^{-3/2} + \epsilon \lambda^{-1/2} \right) \|r\|_6$.
\end{enumerate}
\end{lemma}

\begin{proof}
(a) We write $\psi_\xl = U_\xl - \lambda^{-1/2} H_a(x,\cdot) - f_\xl$ and bound pointwise
\begin{align}
	\label{eq:expansionrproof}
	(\psi_\xl + s + r)^5 & = U_\xl^5 + 5U_\xl^4(s+r) + \mathcal O\left( U_\xl^4\left( \lambda^{-1/2} |H_a(x,\cdot)| + |f_\xl|\right) + U_\xl^3 \left( r^2 + s^2 \right) \right) \notag \\
	& \quad + \mathcal O\left( \lambda^{-5/2}|H_a(x,\cdot)|^5 + |f_\xl|^5 + |r|^5 + |s|^5\right).
\end{align}
When integrated against $r$, the first term vanishes by orthogonality. Let us bound the contribution coming from the second term, that is, from $5 U_\xl^4 s$. We write
$$
s = \lambda^{-1} \beta U_\xl + \gamma \pl U_\xl + \tilde s \,,
$$
so $\tilde s$ consists of the zero mode contributions involving the $\delta_i$, plus contributions from the difference between $PU_\xl$ and $U_\xl$ in the terms involving $\beta$ and $\gamma$. By orthogonality, we have
$$
\int_\Omega U_\xl^4 sr = \int_\Omega U_\xl^4 \tilde s r = \mathcal O(\|U_\xl\|_6^4 \|\tilde s\|_6 \|r\|_6) \,.
$$
and, by Lemmas \ref{lemma Lq norm of U} and \ref{lemma PU}, as well as Proposition \ref{proposition s},
$$
\|\tilde s\|_6 \leq \left( |\beta|+|\gamma|\right) \left( \lambda^{-1} \|\varphi_\xl \|_6 + \|\pl\varphi_\xl\|_6 \right) + \lambda^{-3} \sum_{i=1}^3 |\delta_i| \|\partial_{x_i}PU_\xl\|_6 \lesssim \lambda^{-3/2} \,.
$$
This proves
\begin{equation} \label{u4-rs}
\int_\Omega U_\xl^4 sr = \mathcal O(\lambda^{-3/2}\| r\|_6) \,.
\end{equation} 

It remains to bound the remainder terms in \eqref{eq:expansionrproof}. We write $H_a(x,y) = \phi_a(x) + \mathcal O(|x-y|)$ and bound
\begin{align*}
\int_{\Omega} U_\xl^{24/5} |H_a(x,\cdot)|^{6/5} & \lesssim \phi_a(x)^{6/5} \int_\Omega U_\xl^{24/5}  + \int_\Omega U_\xl^{24/5} |x-y|^{6/5} \lesssim \lambda^{-3/5}\phi_a(x)^{6/5} + \lambda^{-9/5} \,.
\end{align*}
Hence
\begin{align}
	& \left| \int_\Omega U_\xl^4 \left( \lambda^{-1/2} |H_a(x, \cdot)| + |f_\xl| \right) |r| \right| \leq \left( \lambda^{-1/2} \|U_\xl^4 H_a(x,\cdot)\|_{6/5} + \| U_\xl^4\|_{6/5} \|f_\xl \|_\infty \right) \| r\|_6 \nonumber \\
	& \lesssim  \left(\lambda^{-1} \phi_a(x) + \lambda^{-2} \right) \| r\|_6 \,.  \label{u4r6}
\end{align}
Finally, using Proposition \ref{proposition s},
\begin{align*}
& \int_\Omega U_\xl^3 \left(r^2 + s^2 \right)|r| + \int_\Omega \left(\lambda^{-5/2} |H_a(x,\cdot)|^5 + |f_\xl|^5 + |r|^5 + |s|^5 \right)|r| \\
& \lesssim  \left( \|r\|_6^2 + \|s\|^2_6 + \lambda^{-5/2} + \|f_\xl\|_\infty^5 + \|r\|_6^5 + \|s\|_6^5 \right) \|r\|_6 
\lesssim \left( \|r\|_6^2 + \lambda^{-2} \right) \| r\|_6 . \nonumber
\end{align*} 

\medskip

(b) We have
\begin{align*}
	& \left | \int_\Omega \left( a (s- f_\xl - g_\xl ) + \epsilon V(\psi_\xl+s+r)\right) r \right| \\
	& \lesssim \left( \|s\|_{6/5} + \|f_\xl\|_{6/5} + \|g_\xl\|_{6/5} + \epsilon \|\psi_\xl\|_{6/5} + \epsilon \|r\|_{6/5} \right) \|r\|_6 \,.
\end{align*}
By Proposition \ref{proposition s}, $\|s\|_{6/5}\lesssim \|s\|_2 \lesssim \lambda^{-3/2}$. By Lemma \ref{lemma PU}, $\|f_\xl\|_{6/5}\lesssim \|f_\xl\|_\infty \lesssim \lambda^{-5/2}$. By Lemma \ref{lem-g}, $\|g_\xl\|_{6/5}\lesssim \lambda^{-2}$. By Lemmas \ref{lemma Lq norm of U} and \ref{lemma PU}, $\|\psi_\xl\|_{6/5} \lesssim \lambda^{-1/2}$. Finally, $\|r\|_{6/5}\lesssim \|r\|_6$. This proves the claimed bound.
\end{proof}

\begin{proof}[Proof of Proposition \ref{prop-r-bound}]
We deduce from identity \eqref{energy r} together with Lemma \ref{lemma expansion r} that
\begin{align*}
	\int_\Omega \left( |\nabla r|^2 + ar^2 - 15 \alpha^4 \, U_{x,\lambda}^4 r^2\right) 
	& \lesssim \left( \lambda^{-1} \phi_a(x) + \lambda^{-3/2} + \epsilon \lambda^{-1/2} + \|\nabla r\|^2_2 + \epsilon \|\nabla r\|_2\right) \|\nabla r\|_2\,.
	\end{align*}
Since $\alpha^4 \to 1$ and $r \in T_\xl^\bot$, the coercivity inequality \eqref{coercivity} implies that for all sufficiently small $\epsilon>0$ the left side is bounded from below by $c \|\nabla r\|^2_2$ with a universal constant $c>0$. Thus,
$$
\|\nabla r\|_2\lesssim \lambda^{-1} \phi_a(x) + \lambda^{-3/2} + \epsilon \lambda^{-1/2}  + \|\nabla r\|^2_2 + \epsilon \|\nabla r\|_2\,.
$$
For all sufficiently small $\epsilon>0$, the last two terms on the right side can be absorbed into the left side and we obtain the claimed inequality \eqref{nabla r subsection}. 
\end{proof}

\noindent Proposition \ref{prop-r-bound} is a first step to prove the bound \eqref{r-eps-bound} in Proposition \ref{prop second expansion}. In Section \ref{subsection phi a} we will show that $\phi_a(x)=\mathcal O(\lambda^{-1}+\epsilon)$ and $\lambda^{-1} =\mathcal O(\eps)$. Combining these bounds with Proposition \ref{prop-r-bound} we will obtain \eqref{r-eps-bound}.


\subsection{Expanding $\alpha^4$}
\label{ssec-alpha^4}
In this subsection, we will prove 

\begin{proposition}  \label{prop-alpha4}
As $\eps\to 0$, 
\begin{equation}
\label{alpha4 exp subsec}
\alpha^4 = 1 - 4 \beta \lambda^{-1} + \mathcal O(\phi_a(x) \lambda^{-1} + \lambda^{-2} + \eps \lambda^{-1}),
\end{equation}
where $\beta$ is the zero-mode coefficient from \eqref{expansion s}. 
\end{proposition}

\noindent To prove \eqref{alpha4 exp subsec}, we expand the energy identity obtained by integrating the equation for $u$ against $u$. Writing $u = \alpha(\psi_\xl + q)$, this yields
\begin{equation*}
\int_\Omega |\nabla (\psi_\xl + q)|^2  + \int_\Omega (a + \eps V) (\psi_\xl + q)^2  = 3 \alpha^4 \int_\Omega (\psi_\xl + q)^6,
\end{equation*}
which we write as
\begin{align}
	\label{energy identity alpha^4}
	& \int_\Omega \left( |\nabla\psi_\xl|^2 + (a+\eps V) \psi_\xl^2 - 3 \alpha^4 \psi_\xl^6 \right) + 2 \int_\Omega \left( \nabla q\cdot\nabla\psi_\xl + (a+\eps V) q \psi_\xl - 9 \alpha^4 q\psi_\xl^5 \right) \notag \\
	& = \mathcal R_0
\end{align}
with
$$
\mathcal R_0 := - \int_\Omega \left(|\nabla q|^2 +  (a+\eps V) q^2 \right) + 3 \alpha^4 \sum_{k=2}^6 {6 \choose k} \int_\Omega \psi_\xl^{6-k} q^k \,.
$$

\noindent The following lemma provides the expansions of the terms in \eqref{energy identity alpha^4}. 

\begin{lemma}
\label{lemma alpha^4}
As $\eps \to 0$, the following holds. 
\begin{enumerate}
	\item[(a)] $\int_\Omega \left( |\nabla\psi_\xl|^2 + (a+\eps V) \psi_\xl^2 - 3 \alpha^4 \psi_\xl^6 \right) = (1-\alpha^4) \frac{3\pi^2}{4} + \mathcal O(\phi_a(x) \lambda^{-1} + \lambda^{-2} + \eps \lambda^{-1})$.
	\item[(b)] $\int_\Omega \left( \nabla q\cdot\nabla\psi_\xl + (a+\eps V) q \psi_\xl - 9 \alpha^4 q\psi_\xl^5 \right) = (1-3\alpha^4) \frac{3\pi^2}{4}\beta \lambda^{-1} + \mathcal O(\lambda^{-2}+ \epsilon^2 \lambda^{-1})$.
	\item[(c)] $\mathcal R_0 = \mathcal O (\lambda^{-2}+ \eps^2\lambda^{-1})$.
\end{enumerate}
\end{lemma}

\begin{proof}
(a) In \cite[Theorem 2.1]{FrKoKo1}, we have shown the expansions
\begin{align*}
	& \int_\Omega \left( |\nabla \psi_\xl|^2 + (a+\eps V) \psi_\xl^2 \right)  = 3 \frac{\pi^2}{4} + \mathcal O(\phi_a(x)\lambda^{-1} + \lambda^{-2} + \eps \lambda^{-1}) \,, \\
	& 3 \int_\Omega \psi_\xl^6  = 3 \frac{\pi^2}{4} + \mathcal O(\phi_a(x) \lambda^{-1} + \lambda^{-2}) \,,
\end{align*}
which immediately imply the bound in (a).

\medskip

(b) Since $\Delta (H_a(x,\cdot)-H_0(x,\cdot)) = -a G_a(x,\cdot)$, we have $-\Delta\psi_\xl = 3U_\xl^5 - \lambda^{-1/2}a G_a(x,\cdot)$. Since $\psi_\xl=\lambda^{-1/2}G_a(x,\cdot) - f_\xl - g_\xl$ with $g_\xl$ from \eqref{eq:defg}, we can rewrite this as
\begin{equation}\label{eq:eqpsi}
	-\Delta\psi_\xl + a\psi_\xl = 3U_\xl^5 - a(f_\xl + g_\xl) \,.
\end{equation}
Thus,
\begin{align*}
	& \int_\Omega \left( \nabla q\cdot\nabla\psi_\xl + (a+\eps V) q \psi_\xl - 9 \alpha^4 q\psi_\xl^5 \right) \\
	& = 3(1-3\alpha^4) \int_\Omega q U_\xl^5 - \int_\Omega q \left( 9 \alpha^4 ( \psi_\xl^5 - U_\xl^5 ) + a(f_\xl + g_\xl)+ \eps V\psi_\xl \right).
\end{align*}
By orthogonality and the computations in the proof of Proposition \ref{prop-beta-gamma},
$$
3 \int_\Omega q U_\xl^5 = \int_\Omega \nabla s \cdot \nabla PU_\xl 
= \frac{3\pi^2}{4}\, \beta \lambda^{-1} + \mathcal O(\lambda^{-2}) \,.
$$
Moreover,
\begin{align*}
	& \left| \int_\Omega q \left( 9 \alpha^4 ( \psi_\xl^5 - U_\xl^5 ) + a(f_\xl + g_\xl) + \eps V \psi_\xl \right)  \right| \\
	& \lesssim \|q\|_6 \left( \| \psi_\xl^5 - U_\xl^5 \|_{6/5} + \|f_\xl\|_{6/5} + \|g_\xl\|_{6/5} + \eps \|\psi_\xl\|_{6/5} \right).
\end{align*}
By Propositions \ref{proposition s} and \ref{prop-r-bound}, we have
\begin{equation}
	\label{eq:qbound}
	\|q\|_6 \lesssim \|\nabla q\|_2\lesssim \lambda^{-1} + \epsilon \lambda^{-1/2} \,,
\end{equation}
by Lemma \ref{lemma PU}, $\|f_\xl\|_\infty \lesssim \lambda^{-5/2}$ and, by Lemma \ref{lem-g}, $\|g_\xl\|_{6/5}\lesssim \lambda^{-2}$. Moreover, writing $\psi_\xl = U_\xl - \lambda^{-1/2} H_a(x,\cdot) - f_\xl$ and using Lemmas \ref{lemma Lq norm of U} and \ref{lemma PU} and \eqref{Ha-bound}, we get $\|\psi_\xl\|_{6/5} \lesssim \lambda^{-1/2}$. Also, bounding
$$
\left| \psi_\xl^5 - U_\xl^5 \right| \lesssim \psi_\xl^4 \left( \lambda^{-1/2} |H_a(x,\cdot)| + |f_\xl| \right) + \lambda^{-5/2} |H_a(x,\cdot)|^5 + |f_\xl|^5 \,,
$$
we obtain from Lemmas \ref{lemma Lq norm of U} and \ref{lemma PU} and from \eqref{Ha-bound},
$$
\| \psi_\xl^5 - U_\xl^5 \|_{6/5} \lesssim \lambda^{-1/2} \|\psi_\xl\|_{24/5}^4 + \lambda^{-5/2}\lesssim \lambda^{-1} \,.
$$
Collecting all the terms, obtain the claimed bound.

\medskip

(c) Because of the second inequality in \eqref{eq:qbound}, the first integral in the definition of $\mathcal R_0$ is $\mathcal O(\lambda^{-2} + \eps^2\lambda^{-1})$. The second integral is bounded, in absolute value, by a constant times
$$
\int_\Omega \left( \psi_\xl^4 q^2 + q^6\right) \leq \|\psi_\xl\|_6^4 \|q\|_6^2 + \|q\|_6^6 \lesssim \lambda^{-2} + \eps^2\lambda^{-1} \,.
$$
This completes the proof.
\end{proof}

\begin{proof}[Proof of Proposition \ref{prop-alpha4}]
The claim follows from \eqref{energy identity alpha^4} and Lemma \ref{lemma alpha^4}. 
\end{proof}


 \subsection{Expanding $\phi_a(x)$ }
\label{subsection phi a}

In this subsection we prove the following important expansion.

\begin{proposition}\label{phiaexp}
	As $\epsilon\to 0$,
	\begin{equation}
		\label{phi a exp subsection}
		\phi_a(x) =  \pi\, a(x)\, \lambda^{-1} - \frac{\eps}{4\pi}\, Q_V(x) + o(\lambda^{-1}) +o(\eps)
	\end{equation}
\end{proposition}

Before proving it, let us note the following consequence.

\begin{corollary} \label{cor-lambda-eps}
We have $\phi_a(x_0)=0$, $Q_V(x_0) \leq 0$ and 
\begin{equation} \label{lambda-1}
 \lambda^{-1} = \mathcal O(\eps),
\end{equation}
as $\eps \to 0$. Moreover, $\|\nabla r\|_2= \mathcal{O}(\eps\lambda^{-1/2})$ and $\alpha^4 = 1 + \frac{64}{3 \pi} \phi_0(x)  \lambda^{-1}  + \mathcal O(\eps \lambda^{-1})$.
\end{corollary}

\begin{proof}
	The fact that $\phi_a(x_0)=0$ follows immediately from \eqref{phi a exp subsection}. Since $\phi_a(x)\geq 0$ by criticality and since $a(x_0)<0$ by assumption, we deduce from \eqref{phi a exp subsection} that $Q_V(x_0) \leq 0$ and that
$$
 \lambda^{-1} \leq \frac{|Q_V(x_0)| + o(1)}{4 \pi^2 |a(x_0)| + o(1)}\ \eps = \mathcal O(\eps).
$$
Reinserting this into \eqref{phi a exp subsection} we find $\phi_a(x) = \mathcal O(\eps)$. Inserting this into Proposition \ref{prop-r-bound}, we obtain the claimed bound on $\|\nabla r\|_2$, and inserting it into \eqref{alpha4 exp subsec} and \eqref{eq beta gamma subsec}, we obtain the claimed expansion of $\alpha^4$.
\end{proof}

The proof of \eqref{phi a exp subsection} is based on the Pohozaev identity obtained by integrating the equation for $u$ against $\pl \psi_\xl$. We write the resulting equality in the form
\begin{align}    
\label{pohozaev identity lambda refined II}
& \int_\Omega \left( \nabla\psi_\xl\cdot\nabla\pl\psi_\xl + (a+\eps V)\psi_\xl\pl\psi_\xl -3 \alpha^4 \psi_\xl^5 \pl\psi_\xl \right) \notag \\
&= -  \int_\Omega \left( \nabla q\cdot\nabla\pl\psi_\xl + a q \pl\psi_\xl -15 \alpha^4 q\psi_\xl^4 \pl\psi_\xl \right) + 30\alpha^4 \int_\Omega q^2 \psi_\xl^{3} \pl \psi_\xl + \mathcal R
\end{align}
with
$$
\mathcal R = - \epsilon \int_\Omega V q \pl\psi_\xl + 3 \alpha^4  \sum_{k = 3}^5 {5 \choose k} \int_\Omega \psi_\xl^{5-k} q^{k} \pl \psi_\xl . 
$$

The involved terms can be expanded as follows.

\begin{lemma}
\label{lemma pohozaev bis}
As $\eps\to 0$, the following holds.
\begin{enumerate}
\item[(a)] 
	$ \displaystyle
	\begin{aligned}[t]
	& \int_\Omega \left( \nabla\psi_\xl\cdot\nabla\pl\psi_\xl + (a+\eps V)\psi_\xl\pl\psi_\xl -3 \alpha^4 \psi_\xl^5 \pl\psi_\xl \right) \\
	& = -2\pi \phi_a(x) \lambda^{-2} - \frac12 Q_V(x) \epsilon\lambda^{-2} 	
	+ (1-\alpha^4) 4\pi \phi_a(x) \lambda^{-2} + \left( 2\pi^2 a(x) + 15 \pi^2 \phi_a(x)^2 \right) \lambda^{-3} \\
	& \quad + o(\lambda^{-3}) + o(\eps \lambda^{-2}) \,.
	\end{aligned} $
\item[(b)] 
	$ \displaystyle
	\begin{aligned}[t]
	& \int_\Omega \left( \nabla q\cdot\nabla\pl\psi_\xl + a q \pl\psi_\xl -15 \alpha^4 q\psi_\xl^4 \pl\psi_\xl \right) \\
	& = -(1-\alpha^4)2\pi \left( \phi_a(x) - \phi_0(x) \right) \lambda^{-2} + \mathcal O(\phi_a(x)\lambda^{-3}) + o(\epsilon\lambda^{-2}) + o(\lambda^{-3}) \,.
	\end{aligned} $
\item[(c)] 
	$ \displaystyle
	\begin{aligned}[t]
	30 \alpha^4 \int_\Omega q^2 \psi_\xl^{3} \pl \psi_\xl = \frac{15\pi^2}{16}\, \beta\gamma\, \lambda^{-3} + \mathcal O(\phi_a(x)\lambda^{-3}) + o(\epsilon\lambda^{-2}) + o(\lambda^{-3}) \,.
	\end{aligned} $
\item[(d)]
	$ \displaystyle
	\begin{aligned}[t]
	\mathcal R = \mathcal O(\phi_a(x)\lambda^{-3}) + o(\epsilon\lambda^{-2}) + o(\lambda^{-3}) \,.
	\end{aligned} $
\end{enumerate}
\end{lemma}

\noindent We emphasize that the proof of Lemma \ref{lemma pohozaev bis} is independent of the expansion of $\alpha^4$ in \eqref{alpha4 exp subsec}. We only use the fact that $\alpha=1+o(1)$.

\begin{proof} 
[Proof of Lemma \ref{lemma pohozaev bis}]
(a) Because of \eqref{eq:eqpsi}, the quantity of interest can be written as
\begin{align}\label{eq:expnp1}
	& \int_\Omega \left( \nabla\psi_\xl\cdot\nabla\pl\psi_\xl + (a+\eps V)\psi_\xl\pl\psi_\xl -3 \alpha^4 \psi_\xl^5 \pl\psi_\xl \right) \notag \\
	& = 3 \int_\Omega \left( U_\xl^5 - \alpha^4 \psi_\xl^5 \right)\pl \psi_\xl 
	- \int_\Omega a (f_\xl + g_\xl)\pl\psi_\xl + \eps \int_\Omega V\psi_\xl\pl\psi_\xl \,.
\end{align}
We discuss the three integrals on the right side separately. As a general rule, terms involving $f_\xl$ will be negligible as a consequence of the bounds $\|f_\xl\|_\infty = \mathcal O(\lambda^{-5/2})$ and $\| \pl f_\xl\|_\infty = \mathcal O(\lambda^{-7/2})$ in Lemma \ref{lemma PU}. This will not always be carried out in detail.

We have
\begin{equation}
	\label{eq:expnp11}
	\int_\Omega \left( U_\xl^5 - \alpha^4 \psi_\xl^5 \right)\pl \psi_\xl 
	= (1-\alpha^4) \int_\Omega U_\xl^5 \pl \psi_\xl + \alpha^4 \int_\Omega \left( U_\xl^5 - \psi_\xl^5 \right)\pl \psi_\xl \,.
\end{equation}
The first integral is, since $\psi_\xl = U_\xl - \lambda^{-1/2}H_a(x,\cdot) - f_\xl$,
\begin{equation} \label{U^5-psi} 
\int_\Omega U_\xl^5 \pl \psi_\xl  = \int_\Omega U_\xl^5 \pl U_\xl + \frac12 \lambda^{-3/2} \int_\Omega U_\xl^5 H_a(x,\cdot) + \mathcal O(\lambda^{-4}) \,.
\end{equation}
Since $\int_{\R^3} U_\xl^5 \pl U_\xl = (1/6) \pl \int_{\R^3} U_\xl^6 = 0$, we have 
\begin{align}\label{U^5-psi-2} 
	\left| \int_\Omega U_\xl^5 \pl U_\xl \right| = \left| \int_{\R^3\setminus\Omega} U_\xl^5 \pl U_\xl \right| \lesssim \lambda^{-1} \int_{d \lambda}^\infty \left|\frac{r^2 - r^4}{(1 + r^2)^4} \right| \, dr = \mathcal O(\lambda^{-4}).  
\end{align}
Next, by Lemma \ref{lemma U Ha},
$$
\frac12 \lambda^{-3/2} \int_\Omega U_\xl^5 H_a(x,\cdot) = \frac{2\pi}3 \phi_a(x) \lambda^{-2} +  \mathcal O(\lambda^{-3}) \,.
$$
This completes our discussion of the first term on the right side of \eqref{eq:expnp11}. For the second term we have similarly,
\begin{equation} \label{u5-psi5-1}
\begin{aligned}
	\int_\Omega \left( U_\xl^5 - \psi_\xl^5 \right)\pl \psi_\xl & = \int_\Omega \left( U_\xl^5 - (U_\xl-\lambda^{-1/2}H_a(x,\cdot))^5 \right)\pl (U_\xl-\lambda^{-1/2}H_a(x,\cdot)) \\
	& \quad + o(\lambda^{-3}) \\
	& = 5 \lambda^{-1/2} \int_\Omega U_\xl^4H_a(x,\cdot)\pl U_\xl + \frac52 \lambda^{-2} \int_\Omega U_\xl^4H_a(x,\cdot)^2 \\
	& \quad - 10 \lambda^{-1} \int_\Omega U_\xl^3 H_a(x,\cdot)^2 \pl U_\xl \\
	& \quad + \sum_{k=3}^5 {5 \choose k} (-1)^k \lambda^{-k/2} \int_\Omega U_\xl^{5-k}  H_a(x,\cdot)^k\pl U_\xl \\
	& \quad - \frac12 \sum_{k=2}^5 {5 \choose k} (-1)^k \lambda^{-(k+3)/2} \int_\Omega U_\xl^{5-k}  H_a(x,\cdot)^{k+1} + o(\lambda^{-3}) \,.
\end{aligned}
\end{equation}
Again, by Lemma \ref{lemma U Ha},
\begin{equation}  \label{u5-psi5-2}
\begin{aligned}
	& 5 \lambda^{-1/2} \int_\Omega U_\xl^4H_a(x,\cdot)\pl U_\xl + \frac52 \lambda^{-2} \int_\Omega U_\xl^4H_a(x,\cdot)^2 - 10 \lambda^{-1} \int_\Omega U_\xl^3 H_a(x,\cdot)^2 \pl U_\xl \\
	& = -\frac{2 \pi}3\, \phi_a(x)\, \lambda^{-2} + \left( 2\pi\, a(x) + 5 \pi^2\, \phi_a(x)^2 \right) \lambda^{-3} + o(\lambda^{-3}) \,.
\end{aligned}
\end{equation}
Finally, the two sums are bounded, in absolute value, by
\begin{align*}
	& \int_\Omega (U_\xl^2 \lambda^{-3/2} |H_a(x,\cdot)|^3 + \lambda^{-5/2} |H_a(x,\cdot)|^5)|\pl U_\xl| + \int_\Omega (U_\xl^3 \lambda^{-5/2} |H_a(x,\cdot)|^3 + \lambda^{-4} |H_a(x,\cdot)|^6) \\
	& \lesssim \|\pl U_\xl\|_6 (\|U_\xl\|_{12/5}^2 \lambda^{-3/2} + \lambda^{-5/2}) + \|U_\xl\|_3^3 \lambda^{-5/2} + \lambda^{-4} = o(\lambda^{-3}). 
\end{align*} 
This completes our discussion of the second term on the right side of \eqref{eq:expnp11} and therefore of the first term on the right side of \eqref{eq:expnp1}.

For the second term on the right side of \eqref{eq:expnp1} we get, using $\psi_\xl = U_\xl - \lambda^{-1/2}H_a(x,\cdot) - f_\xl$,
$$
\int_\Omega a (f_\xl + g_\xl)\pl\psi_\xl = \int_\Omega a g_\xl \pl U_\xl +\frac12 \lambda^{-3/2} \int_\Omega a g_\xl H_a(x,\cdot) + o(\lambda^{-3})
$$
The second integral is negligible since, by Lemma \ref{lem-g},
$$
\left| \frac12 \lambda^{-3/2} \int_\Omega a g_\xl H_a(x,\cdot) \right| \lesssim \lambda^{-3/2} \int_{\Omega} g_\xl \lesssim \lambda^{-4}\log\lambda \,.
$$

Since $a$ is differentiable, we can expand the first integral as
$$
\int_\Omega a g_\xl \pl U_\xl = a(x) \int_\Omega g_\xl \pl U_\xl + \mathcal O\left( \int_\Omega |x-y| g_\xl |\pl U_\xl| \right).
$$
We have
$$
\int_\Omega g_\xl \pl U_\xl = \lambda^{-3} \int_{\lambda(\Omega-x)} g_{0,1} \partial_\lambda U_{0,1} = \lambda^{-3} \int_{\R^3} g_{0,1} \partial_\lambda U_{0,1} + o(\lambda^{-3})
$$
and
$$
\int_{\R^3} g_{0,1} \partial_\lambda U_{0,1} = 4\pi \int_0^\infty \left( \frac1r - \frac1{\sqrt{1+r^2}} \right) \frac{1-r^2}{2(1+r^2)^{3/2}} \,r^2\,dr = 2\pi(3-\pi) \,.
$$
Using similar bounds one verifies that
$$
\int_\Omega |x-y| g_\xl |\pl U_\xl| \lesssim \lambda^{-4} \int_{\lambda(\Omega-x)} |z| g_{0,1} |\partial_\lambda U_{0,1} | \lesssim \lambda^{-4} \,.
$$
This completes our discussion of the second term on the right side of \eqref{eq:expnp1}.

For the third term on the right side of \eqref{eq:expnp1}, we write $\psi_\xl= \lambda^{-1/2} G_a(x,\cdot) - f_\xl-g_\xl$ and get
\begin{align*}
	& \int_\Omega V\psi_\xl\pl\psi_\xl = \int_\Omega V \left( \lambda^{-1/2} G_a(x,\cdot) - g_\xl \right) \partial_\lambda \left( \lambda^{-1/2} G_a(x,\cdot) - g_\xl \right) + o(\lambda^2) \\
	& = - \frac12 \lambda^{-2} Q_V(x) + \mathcal O \left( \lambda^{-3/2}  \int_\Omega G_a(x,\cdot) g_\xl + \lambda^{-1/2} \int_\Omega G_a(x,\cdot) |\pl g_\xl| + \int_\Omega g_\xl |\pl g_\xl| \right) \\
	& \quad + o(\lambda^2) \\
	&= - \frac12 \lambda^{-2} Q_V(x) + \mathcal O\left(\lambda^{-3/2} \|G_a(x,\cdot)\|_2 \|g_\xl\|_2 + \lambda^{-1} \|G_a(x,\cdot)\|_2 \|\pl g_\xl\|_2 + \|g_\xl\|_2 \|\pl g_\xl\|_2 \right) \\
	& \quad + o(\lambda^{-2}) \\
	&= - \frac12 \lambda^{-2} Q_V(x) + o( \lambda^{-2}). 
\end{align*}
In the last equality we used the bounds from Lemma \ref{lem-g} and the fact that $G_a(x,\cdot)\in L^2(\Omega)$. This completes our discussion of the third term on the right side of \eqref{eq:expnp1} and concludes the proof of (a).

\medskip

(b) We note that \eqref{eq:eqpsi} yields
$$
-\Delta\pl\psi_\xl + a \pl\psi_\xl = 15 U_\xl^4 \pl U_\xl - a\left( \pl f_\xl + \pl g_\xl \right).
$$
Because of this equation, the quantity of interest can be written as
\begin{align}\label{eq:expnp2}
	& \int_\Omega \left( \nabla q\cdot\nabla\pl\psi_\xl + a q \pl\psi_\xl -15 \alpha^4 q\psi_\xl^4 \pl\psi_\xl \right) \notag \\
	& = 15 \int_\Omega q \left( U_\xl^4 \pl U_\xl - \alpha^4 \psi_\xl^4 \pl\psi_\xl\right)
	- \int_\Omega a q \left( \pl f_\xl + \pl g_\xl \right).
\end{align}
We discuss the two integrals on the right side separately.

We have
\begin{align}
	\label{eq:expnp21}
	\int_\Omega q \left( U_\xl^4 \pl U_\xl - \alpha^4 \psi_\xl^4 \pl\psi_\xl\right)
	& = (1-\alpha^4) \int_\Omega q U_\xl^4 \pl U_\xl \notag \\
	& \quad + \alpha^4 \int_\Omega q \left( U_\xl^4 \pl U_\xl - \psi_\xl^4 \pl\psi_\xl\right).
\end{align}
The first integral is, by the orthogonality condition $0 = \int_\Omega \nabla w\cdot \nabla\pl P U_\xl = 15 \int_\Omega w U_\xl^4\pl U_\xl$,
\begin{align} \label{qU^4-U}
\int_\Omega q U_\xl^4 \pl U_\xl & = \lambda^{-1/2} \int_\Omega \left( H_a(x,\cdot) - H_0(x,\cdot) \right) U_\xl^4 \pl U_\xl \notag \\
& = - \frac2{15} \pi \left( \phi_a(x)-\phi_0(x)\right) \lambda^{-2} + 
\mathcal O(\lambda^{-3}).
\end{align}
For the second integral on the right side of \eqref{eq:expnp21} we have
\begin{align}\label{u^4-psi^4}
	& \int_\Omega q \left( U_\xl^4 \pl U_\xl - \psi_\xl^4 \pl\psi_\xl\right)   \notag \\ 
	& = \int_\Omega q \left( U_\xl^4 \pl U_\xl - (U_\xl - \lambda^{-1/2} H_a(x,\cdot))^4 \pl \left( U_\xl - \lambda^{-1/2} H_a(x,\cdot) \right) \right)	+ o(\lambda^{-3})\nonumber\\
	& = \mathcal O(\phi_a(x)\lambda^{-3}) + o(\epsilon\lambda^{-2}) + o(\lambda^{-3}) \,.
\end{align}
Let us justify the claimed bound here for a typical term. We write $H_a(x,y) = \phi_a(x) + \mathcal O(|x-y|)$ and get 
$$
\int_\Omega q U_\xl^4 \lambda^{-3/2} H_a(x,\cdot) = \lambda^{-3/2} \phi_a(x) \int_\Omega q U_\xl^4 + \mathcal O\left( \lambda^{-3/2} \int_\Omega q U_\xl^4 |x-y| \right).
$$
Using the bound \eqref{eq:qbound} on $q$ and Lemma \ref{lemma Lq norm of U} we get
$$
\left| \int_\Omega q U_\xl^4 \right| \leq \|q\|_6 \|U_\xl\|_{24/5}^4 \lesssim \lambda^{-3/2} + \epsilon \lambda^{-1} \,.
$$
The remainder term is better because of the additional factor of $|x-y|$. We gain a factor of $\lambda^{-1}$ since
$$
\left\| |x-\cdot|^{1/4} U_\xl \right\|_{24/5}^4 \lesssim \lambda^{-3/2} \,.
$$
Another typical term,
$$
\int_\Omega q U_\xl^3 \lambda^{-1/2} H_a(x,\cdot) \pl U_\xl \,,
$$
can be treated in the same way, since the bounds for $\pl U_\xl$ are the same as for $\lambda^{-1} U_\xl$; see Lemma \ref{lemma Lq norm of U}. The remaining terms are easier. This completes our discussion of the first term on the right side of \eqref{eq:expnp2}.

The second term on the right side of \eqref{eq:expnp2} is negligible. Indeed,
\begin{equation} \label{negligible}
\int_\Omega a q \left( \pl f_\xl + \pl g_\xl \right) = \mathcal O( \|q\|_6 \|\pl g_\xl\|_{6/5}) + o(\lambda^{-3}) = o(\lambda^{-3}) \,,
\end{equation}
where we used Lemma \ref{lem-g} and the same bound on $q$ as before. This completes our discussion of the second term on the right side of \eqref{eq:expnp2} and concludes the proof of (b).

\medskip

(c) We use the form \eqref{expansion s} of the zero modes $s$, as well as the bounds on $\|\nabla s\|_2$ and $\|\nabla r\|_2$ from \eqref{bounds s} and \eqref{nabla r subsection}, to find
\begin{align} 
	& \int_\Omega q^2\,  \psi_\xl^3 \, \pl \psi_\xl  = \int_\Omega s^2\,  \psi_\xl^3 \, \pl \psi_\xl  + \mathcal O(\phi_a(x) \lambda^{-3}) + o(\lambda^{-3}) + o(\eps \lambda^{-2}) \nonumber \\
	& = \beta^2\lambda^{-2} \int_\Omega U_\xl^5 \, \pl U_\xl  +2\beta\gamma\, \lambda^{-1} 
	\int_\Omega U_\xl^4 \, (\pl U_\xl )^2 +\gamma^2 \int_\Omega U_\xl^3 \, (\pl U_\xl )^3 \nonumber \\
	& \quad + \mathcal O(\phi_a(x) \lambda^{-3}) + o(\lambda^{-3}) + o(\eps \lambda^{-2}) \, . \label{rk-2}
\end{align}
A direct calculation using \eqref{pl-U} gives 
$$
\lambda^{-2} \int_\Omega U_\xl^5 \, \pl U_\xl = o(\lambda^{-3}), \qquad \int_\Omega U_\xl^3 \, (\pl U_\xl )^3 = o(\lambda^{-3})
$$
and 
\begin{align*}
	\int_\Omega U_\xl^4 \, (\pl U_\xl )^2 & = \frac 14\, \lambda^{-2} \int_\Omega U_\xl^6 -\lambda^3 \int_\Omega   \frac{ |x-y|^2}{(1+\lambda^2\, |x-y|^2)^{4}}\, + \lambda^5 \int_\Omega   \frac{ |x-y|^4}{(1+\lambda^2\, |x-y|^2)^{5}}\ \\
	& = \frac{\pi^2}{16}\, \lambda^{-2} -4\pi \lambda^{-2}\, \int_0^\infty\frac{t^4\, dt}{(1+t^2)^4} +4\pi \lambda^{-2} \int_0^\infty\frac{t^6\, dt}{(1+t^2)^5}+  o(\lambda^{-2})  \\
	& = \frac{\pi^2}{64}\, \lambda^{-2}+  o(\lambda^{-2}) .
\end{align*}
Inserting this into \eqref{rk-2} gives the claimed expansion (c).

The proof of (d) uses similar bounds as in the rest of the proof and is omitted.
\end{proof}

\begin{proof}[Proof of Proposition \ref{phiaexp}]
Combining \eqref{pohozaev identity lambda refined II} with Lemma \ref{lemma pohozaev bis} yields
\begin{align}
\label{phi a intermed}
0 &= -4 \pi \phi_a(x) \lambda^{-2}  - Q_V(x) \eps \lambda^{-2} + 4 \pi^2 a(x) \lambda^{-3} + \lambda^{-3} R \notag \\
& \quad + \mathcal O(\phi_a(x) \lambda^{-3}) + o(\lambda^{-3}) + o(\eps \lambda^{-2})
\end{align}
with 
\begin{equation*}
R = \lambda(1-\alpha^4) 4\pi \left( \phi_a(x)+\phi_0(x) \right) + 30\pi^2\phi_a(x)^2 - \frac{15}{8} \beta \gamma \pi^2 \,.
\end{equation*}
We now make use of the expansion \eqref{alpha4 exp subsec} of $\alpha^4-1$ and obtain
\begin{align*}
	R & = 16 \beta \pi \phi_0(x)  - \frac{15}{8} \beta \gamma \pi^2 + \mathcal O(\phi_a(x) + \lambda^{-1} + \epsilon) \,.
\end{align*}
Inserting the expansions \eqref{eq beta gamma subsec} of $\beta$ and $\gamma$, we find the cancellation
\begin{align}
	\label{R3}
	R & = \mathcal O(\phi_a(x) + \lambda^{-1} + \epsilon) \,. 
\end{align}
In particular, $R=\mathcal O(1)$ and, inserting this into \eqref{phi a intermed}, we obtain 
\[ \phi_a(x) = \mathcal O(\lambda^{-1} + \eps). \]
In particular, for the error term in \eqref{phi a intermed}, we have $\phi_a(x) \lambda^{-3} = o(\lambda^{-3})$ and, moreover, by \eqref{R3}, $R=\mathcal O(\lambda^{-1} + \eps)$. Inserting this bound into \eqref{phi a intermed}, we obtain the claimed expansion \eqref{phi a exp subsection}. 
\end{proof}

\subsection{Bounding $\nabla \phi_a(x)$}

In this subsection we prove the bound on $\nabla \phi_a(x)$ in Proposition~\ref{prop second expansion}.

\begin{proposition} \label{prop-grad-phi}
	For every $\mu<1$, as $\eps\to 0$, 
\begin{equation}  \label{bound-grad-phi}
|\nabla \phi_a(x)| \lesssim \eps^\mu \,.
\end{equation}
\end{proposition}

The proof of this proposition is a refined version of the proof of Proposition \ref{prop bdry concentration}. It is also based on expanding the Pohozaev identity \eqref{pohozaev type u}. Abbreviating, for $v,z \in H^1(\Omega)$,
\begin{equation} \label{I-def}
I[v,z] :=  \int_{\partial \Omega} \frac{\partial v}{\partial n} \frac{\partial z}{\partial n} n  + \int_\Omega (\nabla a) vz ,
\end{equation}
and writing $u = \alpha (\psi_\xl + q)$, we can write identity \eqref{pohozaev type u} as
\begin{equation}
	\label{pohoz-nabla-u}
	0 = I[\psi_\xl] + 2 I[\psi_\xl, q] + I[q] + \eps \int_\Omega (\nabla V) (\psi_\xl + q)^2 \,.
\end{equation}
The following lemma extracts the leading contribution from the main term $I[\psi_\xl]$. 
 
\begin{lemma}
\label{lemma I psi}
$I[\psi_\xl] = 4 \pi \nabla \phi_a(x) \lambda^{-1} +  \mathcal O(\lambda^{-1-\mu})$ for every $\mu < 1$. 
\end{lemma}

On the other hand, the next lemma allows to control the error terms involving $q$. 

\begin{lemma}
\label{lemma q bdry integral}
$\|\frac{\partial q}{\partial n}\|_{L^2(\partial \Omega)}  \lesssim   \eps \lambda^{-1/2}$.
\end{lemma}

Before proving these two lemmas, let us use them to give the proof of Proposition \ref{prop-grad-phi}. In that proof, and later in this subsection, we will use the inequality
\begin{equation}
	\label{eq:qbound2}
	\|q\|_2 \lesssim \eps \lambda^{-1/2} \,.
\end{equation}
This follows from the bound \eqref{bounds s} on $s$ and the bounds in Corollary \ref{cor-lambda-eps} on $\lambda^{-1}$ and $r$. Note that \eqref{eq:qbound2} is better than the bound \eqref{eq:qbound} in the $L^6$ norm.

\begin{proof}[Proof of Proposition \ref{prop-grad-phi}]
	We shall make use of the bounds
	\begin{equation}
		\label{eq:psibounds2}
		\|\psi_\xl\|_2 +  \| \frac{\partial \psi_\xl}{\partial n}\|_{L^2(\partial \Omega)} \lesssim \lambda^{-1/2} \,.
	\end{equation}
	The first bound follows by writing $\psi_\xl = U_\xl -\lambda^{-1/2}H_a(x,\cdot) + f_\xl$ and using the bounds in Lemmas \ref{lemma Lq norm of U} and \ref{lemma PU} and in \eqref{Ha-bound}. For the second bound we write $\psi_\xl = PU_\xl - \lambda^{-1/2}(H_a(x,\cdot)-H_0(x,\cdot))$ and use the bounds in Lemmas \ref{lemma PU bdry integral} and \ref{lemma nabla Hb bounds}.
 
 Combining the bounds \eqref{eq:psibounds2} with the corresponding bounds for $q$ from Lemma \ref{lemma q bdry integral} and \eqref{eq:qbound2} we obtain
\[ 
\left| I[\psi_\xl, q] \right| \lesssim \eps \lambda^{-1}, \qquad I[q] \lesssim \eps^2 \lambda^{-1} \,.
\]
Moreover, by \eqref{eq:qbound2} and \eqref{eq:psibounds2}, 
$$
\eps \left| \int_\Omega (\nabla V) (\psi_\xl+q)^2 \right| \lesssim \eps \lambda^{-1}. 
$$

In view of these bounds, Lemma \ref{lemma I psi} and equation \eqref{pohoz-nabla-u} imply $|\nabla \phi_a(x)| \lesssim \eps + \lambda^{-\mu}$. Because of \eqref{lambda-1}, this implies \eqref{bound-grad-phi}. 
\end{proof}

It remains to prove Lemmas \ref{lemma I psi} and \ref{lemma q bdry integral}. 

\begin{proof}
[Proof of Lemma \ref{lemma I psi}]
We integrate equation \eqref{eq:eqpsi} for $\psi_\xl$ against $\nabla \psi_\xl$ and obtain
\begin{equation}
	\label{eq:lemmaipsiproof}
	-\frac12 I[\psi_\xl] = 3 \int_\Omega U_\xl^5 \nabla \psi_\xl - \int_\Omega a (f_\xl + g_\xl) \nabla \psi_\xl \,.
\end{equation}
For the first integral on the right side we write $\psi_\xl = U_\xl - \lambda^{-1/2} H_a(x, \cdot) + f_\xl$ and integrate by parts to obtain
\begin{align*}
	3 \int_\Omega U_\xl^5 \nabla \psi_\xl & = 3 \int_{\partial\Omega} U_\xl^5 \left( \frac16\, U_\xl - \lambda^{-1/2}H_a(x,\cdot) + f_\xl \right) n \\
	& \quad + 15 \int_\Omega U_\xl^4 (\nabla U_\xl)  \left(\lambda^{-1/2}H_a(x,\cdot) -f_\xl \right).
\end{align*}
By Lemma \ref{lemma U Ha}, see also Remark \ref{rem:U^4 pxi Hb}, we have
$$
\int_\Omega U_\xl^4 (\nabla U_\xl)  H_a(x,\cdot) = -\int_\Omega U_\xl^4 (\nabla_x U_\xl)  H_a(x,\cdot) = -\frac{2\pi}{15} \nabla\phi_a(x) \lambda^{-1/2} + \mathcal O(\lambda^{-1/2-\mu}) \,.
$$
Finally, since $U_\xl \lesssim \lambda^{-1/2}$ on $\partial\Omega$ and by the bounds on $U_\xl$, $f_\xl$ and $H_a(x,\cdot)$ from Lemmas \ref{lemma Lq norm of U} and \ref{lemma PU} and from \eqref{Ha-bound}, we have
$$
3 \int_{\partial\Omega} U_\xl^5 \left( \frac16 U_\xl - \lambda^{-1/2}H_a(x,\cdot) + f_\xl \right) n + 15 \int_\Omega U_\xl^4 (\nabla U_\xl)  f_\xl = \mathcal O(\lambda^{-2})  \,.
$$
This shows that the first term on the right side of \eqref{eq:lemmaipsiproof} gives the claimed contribution.

On the other hand, for the second term on the right side of \eqref{eq:lemmaipsiproof} we have
\begin{align*}
	& \int_\Omega a(f_\xl+g_\xl)\nabla\psi_\xl = \int_\Omega a (f_\xl +g_\xl) \nabla (U_\xl -\lambda^{-1/2} H_a(x,\cdot)) \\
	& \quad - \frac12 \int_\Omega (\nabla a) f_\xl^2 - \int_\Omega (a\nabla g_\xl + g_\xl \nabla a) f_\xl + \frac12 \int_{\partial\Omega} a f_\xl^2 + \int_{\partial \Omega} a f_\xl g_\xl \\
	& = \int_\Omega a g_\xl \nabla U_\xl + \mathcal O(\lambda^{-3}) \,.
\end{align*}
Here we used bounds from Lemmas \ref{lemma PU} and \ref{lem-g} and from the proof of the latter. Finally, we write $a(y) = a(x) + \mathcal O(|x-y|)$ and using oddness of $g_\xl \nabla U_\xl$ to obtain
$$
\int_\Omega a g_\xl \nabla U_\xl = \mathcal O\left( \int_\Omega |x-y| g_\xl |\nabla U_\xl| \right) = \mathcal O(\lambda^{-2}) \,.
$$
This proves the claimed bound on the second term on the right side of \eqref{eq:lemmaipsiproof}. 
\end{proof}

\begin{proof}
[Proof of Lemma \ref{lemma q bdry integral}]
The proof is analogous to that of Lemma \ref{lemma w bdry integral}. By combining equation \eqref{equation w} for $w$ with $\Delta(H_a(x,\cdot) - H_0(x,\cdot)) = - aG_a(x,\cdot)$, we obtain $-\Delta q = F$ with
$$
F:=  -3U_\xl^5 + 3 \alpha^4 (\psi_\xl + q)^5 - aq + a (f_\xl + g_\xl) - \eps V(\psi_\xl+q) \,.
$$
(We use the same notation as in the proof of Lemma \ref{lemma w bdry integral} for analogous, but different objects.)

We define the cut-off function $\zeta$ as before, but now in our bounds we do not make the dependence on $d$ explicit, since we know already $d^{-1}=\mathcal O(1)$ by Proposition \ref{prop bdry concentration}. Then $\zeta q\in H^2(\Omega)\cap H^1_0(\Omega)$ and
$$
-\Delta(\zeta q) = \zeta F - 2\nabla\zeta \cdot\nabla q - (\Delta\zeta)q  \,.
$$

We claim that
\begin{equation}
	\label{eq:ptwf}
	\zeta |F| \lesssim \zeta |q|^5 + \eps \zeta U_\xl + |q| + \eps \lambda^{-1/2} \,.
\end{equation}
Indeed, on $\Omega \setminus B_{d/2}(x)$, we have $U_\xl \lesssim \lambda^{-1/2}$ and $g_\xl\lesssim \lambda^{-5/2}$. By Corollary \ref{cor-lambda-eps}, we have $\lambda^{-5/2} = \mathcal O(\eps \lambda^{-1/2})$. Moreover, we write $\psi_\xl = U_\xl - \lambda^{-1/2} H_a(x, \cdot) + f_\xl$ and use the bounds on $f_\xl$ and $H_a(x,\cdot)$ from Lemma \ref{lemma PU} and \eqref{Ha-bound}.

Combining \eqref{eq:ptwf} with inequality \eqref{trace estimate w}, we obtain
\begin{align*}
	\left\| \frac{\partial q}{\partial n} \right\|_{L^2(\partial \Omega)} & = \left\| \frac{\partial (\zeta q)}{\partial n} \right\|_{L^2(\partial \Omega)} \lesssim \|\Delta(\zeta q) \|_{3/2} = 	
	 \| \zeta F - 2\nabla\zeta \cdot\nabla q - (\Delta\zeta)q \|_{3/2} \\
	& \lesssim \|\zeta q^5\|_{3/2} + \eps \|\zeta U_\xl\|_{3/2} + \|q\|_{3/2} + \eps \lambda^{-1/2} + \| |\nabla \zeta| |\nabla q|\|_{3/2} + \|(\Delta \zeta) q\|_{3/2} \,.
\end{align*}

It remains to bound the norms on the right side. All terms, except for the first one, are easily bounded. Indeed, by \eqref{eq:qbound2},
\[ \|q\|_{3/2} + \| (\Delta\zeta) q \|_{3/2} \lesssim \|q\|_2 \lesssim \eps \lambda^{-1/2} \]
and 
\[ \| |\nabla \zeta| |\nabla q| \|_{3/2} \lesssim \|\nabla q\|_{L^2(\Omega \setminus B_{d/2}(x))} \leq \|\nabla s\|_{L^2(\Omega \setminus B_{d/2}(x))} + \|\nabla r\|_2 \lesssim \eps \lambda^{-1/2}, \]
where we used $\|\nabla s\|_{L^2(\Omega \setminus B_{d/2}(x))} \lesssim \lambda^{-3/2}$ by Lemma \ref{bounds s} and $\|\nabla r\|_2\lesssim \eps \lambda^{-1/2}$ by Corollary \ref{cor-lambda-eps}. (Notice that for the estimate on $s$ it is crucial that the integral avoids $B_{d/2}(x)$.) Moreover, by Lemma \ref{lemma Lq norm of U},
$$
\|\zeta U_\xl\|_{3/2} \lesssim \|U_\xl\|_{L^{3/2}(\Omega\setminus B_{d/2}(x))} \lesssim \lambda^{-1/2} \,. 
$$

To bound the remaining term $\|\zeta q^5\|_{3/2}$ we argue as in Lemma \ref{lemma w bdry integral} above and get
\begin{align*} \| \zeta q^5\|_{3/2} &= \left(\int_\Omega |\zeta^{1/4} |q|^{1/4} q|^6  \right)^{2/3} \lesssim \left(\int_\Omega |\nabla (\zeta^{1/4} |q|^{1/4} q)|^2  \right)^2	\\
&\lesssim \left( \int_\Omega |q|^{5/2} |\nabla (\zeta^{1/4})|^2  \right) ^2 + \left( \int_\Omega |F|\,  \zeta^{1/2} |q|^{3/2}  \right)^2 \nonumber  \\
& \lesssim \|q\|_6^5 + \left( \int_\Omega |F|\, \zeta^{1/2} |q|^{3/2}  \right)^2. 
\end{align*}
We use the pointwise estimate \eqref{eq:ptwf} on $\zeta F$, which is equally valid for $\zeta^{1/2} F$. The term coming from $|q|^5$ is bounded by
\[ \left( \int_\Omega |q|^{5 + \frac 32}  \zeta^{1/2}   \right)^2 = \left( \int_\Omega ( \zeta|q|^{5})^{1/2} q^4   \right)^2 \leq \|\zeta q^5\|_{3/2} \|q\|_6^{8} = o(\|\zeta q^5\|_{3/2}), \]
which can be absorbed into the left side. The contributions from the remaining terms in the pointwise bound on $\zeta^{1/2} |F|$ can by easily controlled and we obtain
\[ \| \zeta q^5\|_{3/2} \lesssim  \|q\|_6^5 + \lambda^{-5} + (\eps \lambda^{-1/2})^5 \lesssim \epsilon \lambda^{-1/2}. \]
Collecting all the estimates, we obtain the claimed bound.
\end{proof}



\section{Proof of Theorems \ref{thm expansion} and \ref{thm BP}}\label{sec:proofsadd}

\subsection{The behavior of $\phi_a$ near $x_0$}

We are now in a position to complete the proof of Theorem \ref{thm expansion}. Our main remaining goal is to prove 
\begin{equation}
\label{varphi_a final }
\phi_a(x) = o(\epsilon).
\end{equation} 
Once this is shown, we will be able to find a relation between $\lambda$ and $\eps$. The proof of \eqref{varphi_a final } (and only this proof) relies on the nondegeneracy of critical points of $\phi_a$.

We already know that $\phi_a(x_0) = 0$ and that $\phi_a(y)\geq 0$ for all $y\in\Omega$, hence $x_0$ is a critical point of $\phi_a$.  In this subsection we collect the necessary ingredients which exploit this fact. 

\begin{lemma}
\label{lemma C2}
The function $\phi_a$ is of class $C^2$ on $\Omega$. 
\end{lemma} 

Since we were unable to find a proof for this fact in the literature, we provide one in Appendix \ref{section c2}. 

Thus, the following general lemma applies to $\phi_a$. 

\begin{lemma}
\label{lemma hessian}
	Let $u$ be $C^2$ near the origin and suppose that $u(0)=0$, $\nabla u(0)=0$ and that $\Hess u(0)$ is invertible. Then, as $x \to 0$,
	\begin{equation}
	\label{hessian general}
	u(x) = \frac12 \nabla u(x)\cdot \left( \Hess u(0) \right)^{-1} \nabla u(x) + o(|x|^2) \,.
	\end{equation}
Suppose additionally that $\Hess u(0) \geq c$ for some $c > 0$ in the sense of quadratic forms, i.e. the origin is a nondegenerate minimum of $u$. Then, as $x \to 0$, 
	\begin{equation}
	\label{hessian minimum}
	u(x) \lesssim  |\nabla u(x)|^2. 
\end{equation}	 
\end{lemma}

\begin{proof}
	We abbreviate $H(x) = \Hess u(x)$ and make a Taylor expansion around $x$ to get
	\begin{equation}
	\label{eq:u}
	0 = u(0) = u(x) - \nabla u(x)\cdot x + \frac12 x\cdot H(x) x + o(|x|^2)
	\end{equation}
	and
	\begin{eqnarray}
	\label{eq:nablau}
	0 = \nabla u(0) = \nabla u(x) - H(x) x + o(|x|^2) \,.
	\end{eqnarray}
	
	We infer from \eqref{eq:nablau} and the invertibility of $H(0)$ that
	$$
	x = H(x)^{-1} \nabla u(x) + o(|x|^2) \,.
	$$
	Inserting this into \eqref{eq:u} gives
	$$
	0 = u(x) - \frac12 \nabla u(x) \cdot H(x)^{-1} \nabla u(x) + o(|x|^2) \,,
	$$
Since $H(x)^{-1} = H(0)^{-1} + o(|x|)$, this yields \eqref{hessian general}. 

To prove \eqref{hessian minimum}, if 0 is a nondegenerate minimum, then a Taylor expansion around 0 shows
\begin{equation} \label{phi-lowerb} 
u(x) = \frac{1}{2} x \cdot H(0) x + o(|x|^2) \geq \frac{c}{4}|x|^2 
\end{equation}
for small enough $|x|$. Thus, the $o(|x|^2)$ in \eqref{hessian general} can be absorbed in the left side, thus \eqref{hessian minimum}.
\end{proof}


\subsection{Proof of Theorem \ref{thm expansion}}

Equation \eqref{u-eps-final} follows from Proposition \ref{prop first expansion}, together with \eqref{definition q}, \eqref{q-split} and \eqref{eq:sproj}. The facts that $x_0\in\mathcal N_a$ and that $Q_V(x_0)\leq 0$ follow from Corollary \ref{cor-lambda-eps}.

By Lemma \ref{lemma C2} and the assumption that $x_0$ is a nondegenerate minimum of $\phi_a$, we can apply Lemma \ref{lemma hessian} to the function $u(x) := \phi_a(x + x_0)$ to get  
\[ \phi_a(x) \lesssim |\nabla \phi_a(x)|^2 \,. \]
Therefore, by the bound on $\nabla\phi_a(x)$ in Proposition \ref{prop second expansion} with some fixed $\mu\in(1/2,1)$, we get
\begin{equation} \label{phi-grad-phi}
\phi_a(x) \lesssim  |\nabla \phi_a(x)|^2  = o(\eps) \,.
\end{equation}
This proves \eqref{phi-asymp} and, by non-degeneracy of $x_0$, also \eqref{x-x}. Moreover, inserting \eqref{phi-grad-phi} into the expansion of $\phi_a(x)$ from Proposition \ref{prop second expansion}, we find 
\[ 0=   a(x) \pi \lambda^{-1} - \frac{\eps}{4\pi}\, Q_V(x) + o(\lambda^{-1}) +o(\eps),  \]
that is, 
\[ \eps \lambda = 4 \pi^2 \frac{|a(x_0)|+ o(1)}{|Q_V(x_0)|+o(1)} \]
with the understanding that this means $\eps\lambda\to \infty$ if $Q_V(x_0)=0$. This proves \eqref{lim eps lambda}. 

 The remaining claims in Theorem \ref{thm expansion} follow from Proposition \ref{prop second expansion}.
 

\subsection{A bound on $\|w\|_\infty$}
\label{subsection infty bound w}

In this subsection, we prove a crude bound on the $L^\infty$ norm of the first-order remainder $w$ appearing in the decomposition $u = \alpha( PU_\xl + w)$, and also on some of its $L^p$ norms which cannot be controlled through Sobolev, i.e. $p > 6$. This bound was not needed in the proof of Theorem \ref{thm expansion}, but will be in that of Theorem \ref{thm BP}.

\begin{proposition}
\label{proposition infty bound w}
As $\eps \to 0$, 
\begin{equation}
\label{higher int w}
\|w\|_{p} \lesssim \lambda^{-\frac{3}{p}} \qquad \text{ for all } \, p \in (6, \infty).
\end{equation}
Moreover, for every $\mu > 0$, 
\begin{equation}
\label{infty bound w}
\|w\|_\infty = o(\lambda^\mu) \,. 
\end{equation}
\end{proposition}

Our proof follows \cite[proof of (25)]{Re1}, which concerns the case $N \geq 4$ and $a=0$. Since some of the required modifications are rather complicated to state, we give details for the convenience of the reader. 

\begin{proof}
	We begin by proving the first bound in the proposition, which we write as
	$$
	\|w\|^{r+1}_{3(r+1)} \lesssim \lambda^{-1} 
	\qquad\text{for all}\ r\in (1, \infty) \,.
	$$
	To prove this, we define $F$ by \eqref{eq:eqwrhs}, multiply \eqref{equation w} with $|w|^{r-1} w$ and integrate by parts to obtain 
	\[ \frac{4r}{(r+1)^2} \int_\Omega |\nabla |w|^{\frac{r+1}{2}}|^2  = \int_\Omega F |w|^{r-1} w . \]
	Thus, by Sobolev's inequality applied to $v = |w|^{\frac{r+1}{2}}$, 
	\begin{equation}
		\label{estimate int fw} 
		\|w\|^{r+1}_{3(r+1)} \lesssim \int_\Omega |F| |w|^r. 
	\end{equation}
	In order to estimate the right side of \eqref{estimate int fw}, we make use of the bound
	\begin{equation}
		\label{f eq 2}
		|F| \lesssim |\alpha^4 - 1| U_\xl^5 + U_\xl^4 |w| + |w|^5 + U_\xl^4 \varphi_\xl +  U_\xl + \varphi_\xl + |w| \,.
	\end{equation}
	This is a refinement of \eqref{eq:ptwf}, which is obtained by writing $PU_\xl = U_\xl - \varphi_\xl$ and using Lemma \ref{lemma PU} to bound $\varphi_\xl^5\lesssim\varphi_\xl$.
	
	We estimate the resulting terms separately. Using H\"older's inequality, Lemma \ref{lemma Lq norm of U}, Proposition \ref{prop-alpha4} and the fact that for any $\eta, p,q > 0$ with $p^{-1} + q^{-1} = 1$ there is $C_\eta > 0$ such that for any $a, b > 0$ one has $ab \leq \eta a^p + C_\eta b^q$, we obtain
	\begin{align*}
		|\alpha^4-1| \int_\Omega U_\xl^5 |w|^r  & \leq \lambda^{-1} \|w\|_{3(r+1)}^r \|U\|^5_{5 \cdot \frac{3r+3}{2r+3}} \lesssim \lambda^{-1} \|w\|_{3(r+1)}^r \lambda^{\frac{1}{2} \cdot \frac{r-1}{r+1}} = \|w\|_{3(r+1)}^r  \lambda^{-\frac{r+3}{2(r+1)}} \\
		& \leq \eta \|w\|_{3(r+1)}^{r+1} + C_\eta \lambda^{-\frac{r+3}{2}} \, ; 
	\end{align*}
	\begin{align*}
		\int_\Omega U_\xl^4 |w|^{r+1}  \leq 	\left(	\int_\Omega U_\xl^5 |w|^r 	\right)^{4/5}	\left(	\int_\Omega |w|^{r+5}	  \right)^{1/5} \leq \|w\|_{3(r+1)}^{r + \frac{1}{5}} \lambda^{- \frac{4}{5(r+1)}} \leq \eta \|w\|_{3(r+1)}^{r+1} + C_\eta \lambda^{-1}; 
	\end{align*}
		\begin{align*}
		\int_\Omega |w|^{5+r}   \leq \|w\|_{3(r+1)}^{r+1} \|w\|_6^4 \lesssim \|w\|_{3(r+1)}^{r+1}   \lambda^{-2} \,; 
	\end{align*}
	\begin{align*}
		\int_\Omega U_\xl^4 |w|^{r} \varphi_\xl & \leq \lambda^{-1/2} \|w\|_{3(r+1)}^r  \|U_\xl\|^4_{4 \cdot \frac{3r+3}{2r+3}} = \lambda^{-\frac{1}{2} - \frac{1}{r+1}} \|w\|_{3(r+1)}^r = \lambda^{- \frac{r+3}{2(r+1)}}  \|w\|_{3(r+1)}^r \\
		& \leq \eta \|w\|_{3(r+1)}^{r+1} + C_\eta \lambda^{-\frac{r+3}{2}} \,;
	\end{align*}
	\begin{align*}
		\int_\Omega U_\xl |w|^{r}   \leq \|w\|_{3(r+1)}^r \|U_\xl\|_\frac{3r+3}{2r+3} \lesssim \|w\|_{3(r+1)}^r  \lambda^{-\frac{1}{2}} \leq \eta \|w\|_{3(r+1)}^{r+1} + C_\eta \lambda^{-\frac{r+1}2}\, ; 
	\end{align*}
		\begin{align*}
		\int_\Omega \varphi_\xl |w|^{r}   \lesssim \lambda^{-\frac12} \|w\|_{3(r+1)}^r 
		\leq \eta \|w\|_{3(r+1)}^{r+1} + C_\eta \lambda^{-\frac{r+1}2} \,;
	\end{align*}
	\begin{align*}
		\int_\Omega |w|^{r+1} \lesssim \left( \int_\Omega |w|^{5+r}  \right)^\frac{r+1}{r+5} \lesssim  \| w \|_{3(r+1)}^{\frac{(r+1)^2}{r+5}}	\lambda^{-\frac{2(r+1)}{r+5}} \leq \eta \|w\|_{3(r+1)}^{r+1} + C_\eta \lambda^{-\frac{r+1}{2}} \,.
	\end{align*}
	By choosing $\eta$ small enough (but independent of $\lambda$), we can absorb the term $\eta \|w\|_{3(r+1)}^{r+1}$, as well as the term $\lambda^{-2} \|w\|_{3(r+1)}^{r+1}$, into the left hand side of inequality \eqref{estimate int fw} to get 
	\begin{align*}
		\|w\|^{r+1}_{3(r+1)} &\lesssim  \lambda^{-\frac{r+3}{2}} + \lambda^{-1} + \lambda^{-\frac{r+1}{2}} \lesssim \lambda^{-1} \,.
	\end{align*}
	This is the claimed bound.
	
	We now turn to the bound of the $L^\infty$ norm of $w$. We write equation \eqref{equation w} for $w$ as
\begin{equation}
\label{equation w with greensfct}
w (x) = \frac{1}{4 \pi} \int_\Omega G_0(x,y) F(y) .
\end{equation}
By H\"older's inequality and the fact that $0\leq G_0(x,y)\leq |x-y|^{-1}$, we have for every $\delta \in (0,2)$
\begin{equation}
\label{w infty est} 
\|w\|_\infty \leq \sup_{x \in \Omega} \|G_0(x, \cdot)\|_{3 - \delta} \|F\|_{\frac{3 - \delta}{2 - \delta}} \lesssim \|F\|_{\frac{3 - \delta}{2 - \delta}}. 
\end{equation}
Hence it suffices to estimate $\|F\|_q$ with some $q:= \frac{3-\delta}{2-\delta} > 3/2$.  
	
	We use again the bound \eqref{f eq 2}. The $L^q$-norms of the resulting terms are easy to estimate. Indeed, since $|\alpha^4 - 1| \lesssim \lambda^{-1}$ by Proposition \ref{prop-alpha4}, we have by Lemma \ref{lemma Lq norm of U}
\[ |\alpha^4 - 1| \|U_\xl^5\|_{q} \lesssim \lambda^{-1} \|U\|_{5q}^5 \lesssim \lambda^{\frac{3}{2} - \frac{3}{q}}. \]
Next, by Lemma \ref{lemma Lq norm of U} and \ref{lemma PU},
\begin{equation*}
	\| U_\xl^4 \varphi_\xl\|_q \lesssim  \lambda^{-1/2} \|U_\xl\|_{4q}^4 = \lesssim \lambda^{\frac{3}{2} - \frac 3q}. 
\end{equation*}
Using additionally the bound on $\|\nabla w\|$ from Proposition \ref{prop first expansion}, we can estimate, for every $q < 3$, 
\begin{align*}
	\| U_\xl +\varphi_\xl + |w| \|_q \leq \|U_\xl\|_q + \|\varphi_\xl\|_\infty + \|\nabla w\|_6 \lesssim \lambda^{-1/2}  \, .
\end{align*}
Finally, using the bound \eqref{higher int w},
\begin{equation*}
\|U_\xl^4 w\|_q \leq \|U_\xl\|_{5q}^4 \|w\|_{5q} \lesssim \lambda^{2 - \frac{12}{5q}} \|w\|_{5q} \lesssim \lambda^{2 - \frac{3}{q}}
\end{equation*}
and 
\[ \|w^5 \|_q = \|w\|_{5q}^5 \lesssim \lambda^{-\frac{3}{q}}.  \]
Inserting these estimates into \eqref{w infty est} yields
\[ \|w\|_\infty \lesssim  \lambda^{2 - \frac{3}{q}} \qquad \text{ for every } \quad q \in (3/2, 3). \]
As $\delta \searrow 0$ in \eqref{w infty est}, we have $q \searrow 3/2$ and hence $2 - \frac{3}{q} \searrow 0$. Thus \eqref{infty bound w} is proved. 
\end{proof}


\subsection{Proof of Theorem \ref{thm BP}}

By Proposition \ref{prop first expansion}, we have $u = \alpha(PU_\xl + w)$ with $\alpha = 1 + o(1)$. Moreover, by Proposition \ref{proposition infty bound w}, $\|w\|_\infty = o(\lambda^{1/2})$. On the other hand, by Lemma \ref{lemma PU} we have 
\[ \|PU_\xl\|_\infty = \|U_\xl\|_\infty + \mathcal O(\|\varphi_\xl\|_\infty) = \lambda^{1/2} + \mathcal O(\lambda^{-1/2}). \]
Putting these estimates together, we obtain
\[ \eps \|u_\eps \|_\infty^2 = \eps (\lambda^{1/2} + o(\lambda^{1/2}))^2 = \eps \lambda (1+o(1)) =  4 \pi^2 \frac{|a(x_0)|}{|Q_V(x_0)|}(1+ o(1)) \]
by the relationship between $\eps$ and $\lambda$ proved in Theorem \ref{thm expansion}. Moreover, $U_\xl(x)=\lambda^{1/2} = \|U_\xl\|_\infty$. This finishes the proof of  part (a) in Theorem \ref{thm BP}.

The proof of part (b) necessitates much fewer prerequisites. It only relies on the crude expansion of $u$ given in Proposition \ref{prop first expansion} and the rough bounds on $w$ from Proposition \ref{proposition infty bound w}. 

By applying $(-\Delta + a)^{-1}$, we write \eqref{equation u} as
\begin{equation}
\label{u greens asympt} u(z) = \frac{3}{4 \pi} \int_\Omega G_a(z,y) u(y)^5  - \frac{\epsilon}{4 \pi } \int_\Omega G_a(z,y) V(y) u(y)  \,. 
\end{equation}
We fix a sequence $\delta = \delta_\eps = o(1)$ with $\lambda^{-1}  = o(\delta_\eps)$. This condition, together with the bounds from Proposition \ref{prop first expansion} easily implies $\frac{3}{4 \pi} \int_{B_{\delta}(x)} u(y)^5 = \lambda^{-1/2} + o(\lambda^{-1/2})$. Hence
\[ \frac{3}{4 \pi} \int_{B_{\delta}(x)} G_a(z,y) u(y)^5 =\frac{3}{4 \pi} \int_{B_{\delta}(x)} (G_a(z,x_0)+o(1)) u(y)^5 = \lambda^{-1/2} G_a(z, x_0) + o(\lambda^{-1/2}). \]
On the complement of $B_\delta(x)$, using Proposition \ref{proposition infty bound w} and Lemma \ref{lemma Lq norm of U} we bound
\[ \left|\int_{\Omega \setminus B_{\delta}(x)} G_a(z,y) u(y)^5 \right| \lesssim \|G_a(z, \cdot)\|_2 (\|U_\xl\|^5_{L^{10}(\Omega \setminus B_\delta(x))} + \|w\|_{10}^5) \lesssim \lambda^{-5/2} \delta^{-7/2} + \lambda^{-3/2}. \]
Choosing e.g. $\delta = \lambda^{-2/7}$, the last bound is $o(\lambda^{-1/2})$.  

The second term on the right side of \eqref{u greens asympt} is easily bounded by
\[ \eps \left|\int_\Omega G_a(z,y) V(y) u(y) \right| \lesssim \eps \|G_a(z, \cdot)\|_2 (\|U\|_2 + \|w\|_2) \lesssim \eps \lambda^{-1/2} \]
by the bounds from Proposition \ref{prop first expansion} and from Lemma \ref{lemma Lq norm of U}. Collecting the above estimates, part (b) of Theorem \ref{thm BP} follows.


\section{Subcritical case: A first expansion} 
\label{sec-Bp-prelim}

In the remaining part of the paper we will deal with the proof of Theorems \ref{thm expansionconj} and \ref{thm BPconj}. The structure of our argument is very similar to that leading to Theorems \ref{thm expansion} and \ref{thm BP}. Namely, in the present section we derive a preliminary asymptotic expansion of $u_\eps$ and the involved parameters, which is refined subsequently in Section \ref{section refining conj} below. Because of the similarities to the above argument, we will not always give full details. 

The following proposition summarizes the results of this section. 

\begin{proposition}
\label{prop first expansion conj}
Let $(u_\epsilon)$ be a family of solutions to \eqref{BP-problem} satisfying \eqref{eq:sobmineps}. Then, up to the extraction of a subsequence, there are sequences $(x_\epsilon)\subset\Omega$, $(\lambda_\epsilon)\subset(0,\infty)$, $(\alpha_\epsilon)\subset\R$ and $(w_\eps) \subset T_{x_\eps, \lambda_\eps}^\bot$ such that
\begin{equation}
\label{expansion PU + w conj}
u_\epsilon = \alpha_\epsilon(PU_{x_\eps, \lambda_\eps} + w_\eps)
\end{equation}
and a point $x_0 \in \Omega$ such that
\begin{equation}
\label{parameters PU + w conj}
|x_\eps - x_0| = o(1), \quad \alpha_\eps = 1 + o(1), \quad \lambda_\eps \to \infty, \quad \|\nabla w_\eps \|_2= \mathcal O(\lambda_\eps^{-1/2}), \quad \eps = \mathcal O(\lambda_\eps^{-1}). 
\end{equation}
\end{proposition}


\subsection{A qualitative initial expansion}\label{sec:qualexpconj}

As a first step towards Proposition \ref{prop first expansion conj}, we observe that the qualitative expansion from Proposition \ref{lemma PU + w} still holds true, that is, there are sequences $(x_\epsilon)\subset\Omega$, $(\lambda_\epsilon)\subset(0,\infty)$, $(\alpha_\epsilon)\subset\R$ and $(w_\eps) \subset T_{x_\eps, \lambda_\eps}^\bot$ such that \eqref{expansion PU + w conj} holds and a point $x_0\in\overline\Omega$ such that, along a subsequence,
$$
|x_\eps - x_0| = o(1), \quad \alpha_\eps = 1 + o(1), \quad d_\eps \lambda_\eps \to \infty, \quad \|\nabla w_\eps \|_2= o(1),
$$
where, as before, $d_\eps := d(x_\eps, \partial \Omega)$.

Indeed, as explained in the proof of Proposition \ref{lemma PU + w}, it suffices to prove $u_\eps \rightharpoonup 0$ in $H^1_0(\Omega)$ up to a subsequence. To achieve this, we first integrate \eqref{BP-problem} against $u_\eps$ to obtain
	$$
	3 \left( \int_\Omega u_\eps^{6-\eps} \right)^{\frac{4-\epsilon}{6 - \eps}} = \frac{\int_\Omega |\nabla u_\eps|^2}{\left( \int_\Omega u_\eps^{6-\eps} \right)^{\frac{2}{6-\eps}}} + \frac{\int_\Omega a u_\eps^2}{\left( \int_\Omega u_\eps^{6-\eps} \right)^{\frac{2}{6-\eps}}} \,.
	$$
By \eqref{eq:sobmineps} and Hölder, the right side is bounded, hence $\|u_\eps\|_{6-\eps} \lesssim 1$. By \eqref{eq:sobmineps} again, $\|\nabla u_\eps\|_2\lesssim 1$. On the other hand, the right side is $\gtrsim 1$ by coercivity of $-\Delta +a$, which is a consequence of criticality, and by Hölder. This gives $\|u_\eps\|_{6-\eps} \gtrsim 1$, and hence $\|\nabla u_\eps\|_2 \gtrsim 1$ by Sobolev and Hölder. This completes the analogue of Step 1 in the proof of Proposition \ref{lemma PU + w}.

Let us now turn to Step 2 in that proof. We denote by $u_0$ a weak limit point of $u_\eps$ in $H^1_0(\Omega)$, which exists by Step 1. Still by Step 1, we may assume that the quantities $\|u_\eps\|_{6-\eps}$ and $\|\nabla u_\eps\|_2$ have non-zero limits. The only difference to Proposition \ref{lemma PU + w} is now that we modify the definition of $\mathcal M$ to 
$$
	\mathcal M = \lim_{\eps\to 0} \int_\Omega  (u_\eps-u_0)^{6-\eps},
$$
where the exponent is $6-\eps$ instead of $6$. Thanks to the uniform bound $\|u_\eps\|_{6-\eps} \lesssim 1$ by Step 1, it can be easily checked that the proof of the Br\'ezis-Lieb lemma (see e.g. \cite{LiLo}) still yields
$$
 \lim_{\eps\to 0} \int_\Omega u_\eps^{6-\eps} = \lim_{\eps \to 0} \int_\Omega u_0^{6-\eps} + \mathcal M  = \int_\Omega u_0^{6} + \mathcal M  \,. 
$$
Then the modified assumption \eqref{eq:sobmineps} can be used to conclude
	$$
	S\left( \int_\Omega u_0^6 + \mathcal M \right)^{1/3} = \int_\Omega |\nabla u_0|^2 + \mathcal T \,.
	$$
	The rest of the proof is identical to Proposition \ref{lemma PU + w}.

We again adopt the convention that in the remainder of the proof we only consider the above subsequence and we will drop the subscript $\eps$.

In order to prove Proposition \ref{prop first expansion conj}, we will prove in the following subsections that $x_0\in\Omega$, $\|\nabla w\|_2= \mathcal O(\lambda^{-1/2})$ and $\eps = \mathcal O(\lambda^{-1})$.


\subsection{The bound on $\|\nabla w\|_2$ }

The goal of this subsection is to prove

\begin{proposition}
	\label{lem-w-new-bound}
	As $\eps\to 0$,
	\begin{equation} \label{w-new-bound}
		\|\nabla w \|_2 \,  = \,  \mathcal O(\lambda^{-1/2}) + \mathcal O( (\lambda d )^{-1} ) + \mathcal O( \eps ) \, .
	\end{equation}
\end{proposition}

Note that, in contrast to Proposition \ref{boundw}, there appears an additional error $\mathcal O(\eps)$. We will prove in an extra step (Proposition \ref{boundoneps}) that $\eps = \mathcal O( (\lambda d )^{-1} )$, so this extra term will disappear later.

The proof of Proposition \ref{lem-w-new-bound} is somewhat lengthy and we precede it by an auxiliary result, which is a simple consequence of the fact that $\alpha\to 1$.

\begin{lemma} \label{lem-eps-log}
As $\eps\to 0$, 
$$
\eps\log\lambda = o(1).
$$
\end{lemma}

A useful consequence of this lemma is that
\begin{equation}
	\label{eq:uepsbound}
	U_\xl^{-\eps} \lesssim 1
	\qquad\text{in}\ \Omega \,.
\end{equation}
Indeed, this follows from the lemma together with the fact that $U_\xl \gtrsim \lambda^{-1/2}$ in $\Omega$.

\begin{proof} 
We integrate equation \eqref{BP-problem} against $u$ and use the decomposition \eqref{expansion PU + w conj}. This gives 
\begin{equation} \label{u-u}
\int_\Omega |\nabla (PU_\xl+ w)|^2 + \int_\Omega a (PU_\xl+ w)^2 = 3 \alpha^{4-\eps} \int_\Omega (PU_\xl+ w)^{6-\eps}. 
\end{equation}
By orthogonality 
\begin{align*}
\int_\Omega |\nabla (PU_\xl+ w)|^2 & = \int_\Omega |\nabla PU_\xl|^2 + \int_\Omega |\nabla  w|^2 =  \frac{3\pi^2}{4} +o(1). 
\end{align*}
Moreover, using Lemmas \ref{lemma Lq norm of U} and \ref{lemma PU} we find 
$ \int_\Omega a (PU_\xl+ w)^2 = o(1).$ On the other hand, 
$$
 \int_\Omega (PU_\xl+ w)^{6-\eps} =  \int_\Omega U_\xl^{6-\eps} + o(1). 
$$
Hence equation \eqref{u-u} combined with the fact that $\alpha\to 1$ implies 
 \begin{equation} \label{U-6-eps}
 \int_\Omega U_\xl^{6-\eps} = \frac{\pi^2}{4} +o(1)\, . 
 \end{equation} 
 Since 
 $$
  \int_\Omega U_\xl^{6-\eps} = \lambda^{-\frac \eps 2}\, \lambda^3  \int_\Omega \big(1+\lambda^2 |x-y|^2\big)^{-3+\frac \eps 2}  = \lambda^{-\frac \eps 2}\, \frac{\pi^2}{4} \, (1+o(1)), 
 $$
we have $\lambda^{-\frac \eps 2} \to 1 $ and hence the claim. 
\end{proof}

The next result quantifies the difference between $\int_\Omega U_\xl^{5-\eps}\, v$ and $\int_\Omega U_\xl^{5}\, v=0$ for  $v\in T_{x, \lambda}^\bot$.

\begin{lemma} \label{lem-aux-eps}
For every $v\in T_{x, \lambda}^\bot$,
\begin{align} 
\big | \int_\Omega U_\xl^{5-\eps}\, v\,  \big | &  \ \lesssim\  \eps\, \| v\|_6  \label{u5v-2} .
\end{align}
\end{lemma}

\begin{proof}  
By orthogonality,
\begin{equation*} 
 \int_\Omega U_\xl^{5-\eps}\, v  = \lambda^{-\frac \eps2}  \int_\Omega U_\xl^5\, e^{\eps\log  \sqrt{1+\lambda^2 |x-y|^2}}\, v =  \lambda^{-\frac\eps2}  \int_\Omega U_\xl^5\,\Big( e^{\eps\log  \sqrt{1+\lambda^2 |x-y|^2}}\ -1\Big)\, v \,.
\end{equation*} 
By Lemma \ref{lem-eps-log}, 
\begin{equation} \label{eps-log-lambda}
\eps\log  \sqrt{1+\lambda^2 |x-y|^2}= o(1)
\end{equation}
uniformly in $x$ and $y$. Hence 
\begin{equation} \label{exp-log}
0< e^{\eps\log  \sqrt{1+\lambda^2 |x-y|^2}}\ -1 \ \lesssim \ \eps\log  \sqrt{1+\lambda^2 |x-y|^2} \, \leq \, \eps\lambda\, |x-y|,  \\[3pt]
\end{equation}
where we have used the inequality  $\log\sqrt{1+t^2}\, \leq |t|$. Since $\| |x-y| \, U_\xl^5\|_{6/5} = \mathcal{O}(\lambda^{-1})$, the result follows from the H\"older inequality. 
\end{proof}

We are now in position to give the

\begin{proof}[Proof of Proposition \ref{lem-w-new-bound}]
	From equation \eqref{BP-problem} for $u$ we obtain the following equation for $w$,
\begin{equation} \label{eq-w-bp}
-\Delta w + a w = -3 U_\xl^5 -a PU_\xl +3\alpha^{4-\eps} (PU_\xl+w)^{5-\eps} \,.
\end{equation}
Integrating this equation against $w$ gives
\begin{equation}  \label{w2-eq}
\int_\Omega (|\nabla w|^2 + a w^2) = -\int_\Omega a PU_\xl w + 3\alpha^{4-\eps} \int_\Omega w (PU_\xl +w)^{5-\eps} \,.
\end{equation} 
As before, the first term on the right hand side is controlled easily by H\"older,
$$
\left | \int_\Omega a PU_\xl w\, \right | \ \lesssim\ \|PU_\xl\|_2\, \|w\|_2 \ \lesssim \ \lambda^{-1/2}\, \|\nabla w\|_2\, .
$$
In order to control the second term we use the fact that $PU_\xl= U_\xl -\varphi_\xl$.  Moreover, by Taylor and \eqref{eq:uepsbound},
\begin{align}\label{eq:decomptaylor}
(PU_\xl +w)^{5-\eps} & = (U_\xl -\varphi_\xl+w)^{5-\eps}   = U_\xl^{5-\eps} + (5-\eps) U_\xl^{4-\eps}w  \notag  \\ 
& \quad + \mathcal{O}\left( U_\xl^{4} \varphi_\xl + U_\xl^{3} w^2 + |w|^{5-\eps} + \varphi_\xl^{5-\eps} \right)  .
\end{align}
Hence,
\begin{align*}
\big | \int_\Omega  (PU_\xl +w)^{5-\eps}\, w \, - (5-\eps) \alpha^{4-\eps} \int_\Omega U_\xl^{4-\eps} w^2 \big |  & \leq \big | \int_\Omega U_\xl^{5-\eps}\, w\,  \big |  + \mathcal{O}\left( \int_\Omega U_\xl^{4} \varphi_\xl |w| \right)  \\
& \quad + \mathcal{O}\left( \|\nabla w\|_2^3 +  \|\nabla w\|_2 \|\varphi_\xl\|_6^{5-\eps} \right) .
\end{align*}
We estimate the first term on the right side using Lemma \ref{lem-aux-eps}. For the second term on the right side we argue as in the proof of Proposition \ref{boundw} and obtain
$$
\int_\Omega U_\xl^{4} \varphi_\xl |w| =  \mathcal{O}\left( (\lambda d)^{-1} \, \|\nabla w\|_2 \right).
$$
For the last term on the right side we use $\|\varphi_\xl\|^2_6 = \mathcal{O}((\lambda d)^{-1})$. Moreover, in view of  \eqref{exp-log},
\begin{align} \label{U4eps-w}
\int_\Omega U_\xl^{4-\eps} w^2 & \leq\,  \lambda^{-\eps/2} \int_\Omega  U_\xl^4 w^2\, + C \eps\lambda \int_\Omega  U_\xl^4 |x-y|\, w^2 \notag \\
& \leq\, \left( 1+ o(1) \right)  \int_\Omega  U_\xl^4 w^2\, + \mathcal{O} ( \eps\lambda^{-1/2} \, \|\nabla w\|_2^2). 
\end{align}
Altogether we obtain from \eqref{w2-eq},
\begin{align*} 
\int_\Omega (|\nabla w|^2 + a w^2 -15 \alpha^{4-\eps} U_\xl^4\, w^2)\, & \lesssim\,   \big( (\lambda d )^{-1} + \lambda^{-1/2} + \eps\big ) \|\nabla w\|_2 + o(\|\nabla w\|_2^2)
\end{align*}
An application of the coercivity inequality of Lemma \ref{lemma coercivity} now implies \eqref{w-new-bound}.
\end{proof}


\subsection{The bound on $\eps$}

The goal of this subsection is to prove

\begin{proposition}
	\label{boundoneps}
	As $\eps\to 0$,
	\begin{equation} 
		\label{eps leq lambda d -1}
		\eps = \mathcal{O}((\lambda d)^{-1})\, .
	\end{equation} 
\end{proposition}

We note that the analogue of this proposition is not needed in Section \ref{sec:firstexpansion} when studying \eqref{equation u}.

The proof of Proposition \ref{boundoneps} is based on the Pohozaev-type identity
\begin{equation}
\label{pohozaev_plPU}
\int_\Omega \nabla PU_\xl \cdot \nabla \pl PU_\xl  + \int_\Omega a (PU_\xl+w) \pl PU_\xl = \alpha^{4 -\eps} 3 \int_\Omega (PU_\xl+w)^{5-\eps} \pl PU_\xl \,,
\end{equation}
which arises from integrating equation \eqref{eq:u} against $\pl PU_\xl$ and inserting the following bounds.

\begin{lemma}\label{lem-rey1}
As $\eps \to 0$, we have 
\begin{equation}
\label{plPU_LHS} 
\int_\Omega \nabla PU_\xl \cdot \nabla \pl PU_\xl + \int_\Omega a (PU_\xl+w) \pl PU_\xl = \mathcal O(\lambda^{-2} d^{-1} + \lambda^{-1} \|\nabla w\|^2_2)
\end{equation}
and 
\begin{equation}
\label{plPU_RHS} 
3 \int_\Omega (PU_\xl+w)^{5-\eps} \pl PU_\xl = -\frac{1}{16} (1+o(1))\, \eps \lambda^{-1} + \mathcal O(\lambda^{-2} d^{-1} + \lambda^{-1} \|\nabla w\|^2_2) \,. 
\end{equation}
\end{lemma}

Before proving Lemma \ref{lem-rey1}, let us use it to deduce the main result of this subsection.

\begin{proof}[Proof of Proposition \ref{boundoneps}]
	Inserting \eqref{plPU_LHS} and \eqref{plPU_RHS} into \eqref{pohozaev_plPU} and applying the bound \eqref{w-new-bound} on $\|\nabla w\|$ we obtain
	$$
	(1 + o(1)) \eps \lesssim (\lambda d)^{-1} + \|\nabla w\|^2_2 \lesssim (\lambda d)^{-1} + \eps^2 \,.
	$$
	Since $\eps = o(1)$, \eqref{eps leq lambda d -1} follows. 
\end{proof}

In the proof of Lemma \ref{lem-rey1} we need the following auxiliary bound.

\begin{lemma} \label{lem-aux-eps2}
	For every $v\in T_{x, \lambda}^\bot$,
	\begin{align} 
		\big | \int_\Omega U_\xl^{4-\eps}\,\partial_\lambda U_\xl\, v\,  \big | &  \ \lesssim\  \eps\,\lambda^{-1}\, \|\nabla v\|_2  \label{u5v-3} .
	\end{align}
\end{lemma}

The proof of this lemma is analogous to that of Lemma \ref{lem-aux-eps} and is omitted.

\begin{proof}[Proof of Lemma \ref{lem-rey1}] 
We begin with proving \eqref{plPU_LHS}. First, by \cite[(B.5)]{Re2},
\[  \int_\Omega \nabla PU_\xl \cdot \nabla \pl PU_\xl  = \mathcal O(\lambda^{-2} d^{-1}). \]
Writing $PU_\xl = U_\xl - \varphi_\xl$, the second term in \eqref{plPU_LHS} is bounded by 
\begin{align*}
\left| \int_\Omega a (PU_\xl+w) \pl PU_\xl \right| &\lesssim (\|U_\xl\|_2 + \|w\|_2) (\|\pl U_\xl\|_2 + \|\pl \varphi_\xl\|_2) \\
& \lesssim \lambda^{-2} d^{-1/2} + \lambda^{-3/2} d^{-1/2} \|\nabla w\|_2\\
&\lesssim \lambda^{-2} d^{-1} + \lambda^{-1} \|\nabla w\|^2_2,
\end{align*} 
by Lemma \ref{lemma Lq norm of U} and \eqref{phi-upperbs}, followed by Young's inequality.  

Next, we prove \eqref{plPU_RHS}. Using \eqref{eq:decomptaylor} and \eqref{eq:uepsbound} we bound pointwise
\begin{align}\label{eq:decomptaylor2}
	(PU_\xl+w)^{5-\eps} \pl PU_\xl & = U_\xl^{5-\eps} \pl U_\xl + (5-\eps) U_\xl^{4-\eps} \pl U_\xl w \notag \\
	& \quad + \mathcal O\left( \left( U_\xl^{4} \varphi_\xl + U_\xl^{3} w^2 + |w|^{5-\eps} + \varphi_\xl^{5-\eps}\right) |\pl U_\xl| \right) \notag \\
	& \quad + \mathcal O \left( \left( U_\xl^{5} + |w|^{5-\eps} + \varphi_\xl^{5-\eps} \right) |\pl\varphi_\xl| \right).
\end{align}
The integral over $\Omega$ of the two remainder terms is bounded by a constant times
\begin{align*}
	& \|\varphi_\xl\|_\infty \| U_\xl \|_5^{4} \|\pl U_\xl\|_5 + \left( \| U_\xl\|_6^{3} \|w\|_6^2 + \|w\|_6^{5-\eps} + \|\varphi_\xl\|_6^{5-\eps} \right) \|\pl U_\xl\|_6 \\
	& \qquad +  \| U_\xl \|_5^{5} \|\pl\varphi_\xl\|_\infty + \left( \|w\|_6^{5-\eps} + \|\varphi_\xl\|_6^{5-\eps} \right) \|\pl \varphi_\xl\|_6 \\
	& \lesssim \lambda^{-2} d^{-1} + \lambda^{-1} \| w\|_6^2 \,,
\end{align*}
where in the last inequality we used the bounds from Lemmas \ref{lemma Lq norm of U} and \ref{lemma PU}.

By Lemma \ref{lem-aux-eps2}, the integral over $\Omega$ of the second term on the right side of \eqref{eq:decomptaylor2} is bounded by a constant times $\eps \lambda^{-1} \|\nabla w\|_2= o(\eps \lambda^{-1})$. 

Finally, by an explicit calculation,
\begin{align}\label{u5-eps}
 \int_\Omega U_\xl^{5-\eps} \pl U_\xl &=  \int_\Omega U_\xl^{5-\eps} \Big(\frac{U_\xl}{2\lambda}- \frac{\lambda^{3/2}\, |x-y|^2}{(1+\lambda^2\, |x-y|^2)^{3/2}} \Big) \notag  \\[3pt]
 &= \pi \lambda^{-1-\frac\eps2} 
 \left[ \frac{\Gamma(\frac 32) \Gamma(\frac{3-\eps}{2})}{\Gamma(3-\frac \eps 2)} - \frac{2\, \Gamma(\frac 52) \Gamma(\frac{3-\eps}{2})}{\Gamma(4-\frac \eps 2)} \right] +  \mathcal O(\lambda^{-4} d^{-3}) \notag \\
 & = -\frac{\pi^{3/2}}{4} \,  \eps \lambda^{-1-\frac\eps2}  \frac{ \Gamma(\frac{3-\eps}{2})}{\Gamma(4-\frac \eps 2)} +  \mathcal O(\lambda^{-4} d^{-3}) \notag \\
& = -\frac{\pi^2}{48} \, \eps \lambda^{-1} (1+o(1) ) +   \mathcal O(\lambda^{-4} d^{-3}),
\end{align}
where, in the last step, we used Lemma \ref{lem-eps-log}. This completes the proof of \eqref{plPU_RHS}. 
\end{proof}


\subsection{Excluding boundary concentration} 

The goal of this subsection is to prove

\begin{proposition} \label{prop-ebc2}
	$d^{-1}=\mathcal O(1)$. 
\end{proposition}

The proof is very similar to that of Proposition \ref{prop bdry concentration} and we will be brief. Integrating the first equation in \eqref{BP-problem} against $\nabla u$ implies the Pohozaev-type identity
\begin{equation}  \label{pohoz-1}
	- \int_\Omega (\nabla a) \, u^2 = \int_{\partial\Omega} n \Big(\frac{\partial u}{\partial n}\Big)^2 \,. 
\end{equation}
The volume integral on the left side can be estimated as before, since by Propositions \ref{lem-w-new-bound} and \ref{boundoneps} we have the same bound
$$
\|\nabla w\|_2^2 \lesssim \lambda^{-1} + (\lambda d)^{-2}
$$
as before. To bound the surface integral, we use the fact that
$$
\int_{\partial\Omega} \left(\frac{\partial w}{\partial n}\right)^2   =  \mathcal O(\lambda^{-1} d^{-1}) + o \big( \lambda^{-1} d^{-2}) \,.
$$
This is the analogue of Lemma \ref{lemma w bdry integral}. We only note that by \eqref{eq-w-bp} we have 
\begin{equation}
\label{F conj}
 F:= -\Delta w = 3 \alpha^{4-\eps} (PU_\xl+w)^{5-\eps} - 3 U_\xl^5 - a(PU_\xl+w)
\end{equation}
and that this function satisfies \eqref{pointwise bound f}. Therefore, using the above bound on $\|\nabla w\|_2$ we can proceed exactly in the same was as in the proof of Lemma \ref{lemma w bdry integral}.

Thus, as before, we obtain
$$
C\lambda^{-1} \nabla\phi_0(x) = \mathcal O(\lambda^{-1} d^{-3/2}) + o(\lambda^{-1} d^{-2})
$$
and then from $|\nabla\phi_0(x)|\gtrsim d^{-2}$ we conclude that $d^{-1}=\mathcal O(1)$, as claimed.


\subsection{Proof of Proposition \ref{prop first expansion conj}}

The existence of the expansion is discussed in Subsection \ref{sec:qualexpconj}. Proposition \ref{prop-ebc2} implies that $d^{-1} = \mathcal O(1)$, which implies that $x_0\in\Omega$. Moreover, inserting the bound $d^{-1} = \mathcal O(1)$ into Propositions \ref{lem-w-new-bound} and \ref{boundoneps}, we obtain $\eps = \mathcal{O} (\lambda^{-1})$ and $\|\nabla w\|_2 =   \mathcal{O} ( \lambda^{-1/2} )$, as claimed in Proposition \ref{prop first expansion conj}. This completes the proof of the proposition.


\section{Subcritical case: Refining the expansion} 
\label{section refining conj}

As in the additive case, we refine the analysis of the remainder term $w_\eps$ in Proposition \ref{prop first expansion conj}, which we write as $w_\eps = \lambda_\eps^{-1/2} (H_0(x_\eps, \cdot) - H_a(x_\eps, \cdot)) + s_\eps + r_\eps$ with $s_\eps$ and $r_\eps$ as in \eqref{definition r}. 

The following proposition summarizes the main results of this section. 

\begin{proposition}
\label{prop second expansion conj}
Let $(u_\epsilon)$ be a family of solutions to \eqref{BP-problem} satisfying \eqref{eq:sobmineps}. Then, up to the extraction of a subsequence, there are sequences $(x_\epsilon)\subset\Omega$, $(\lambda_\epsilon)\subset(0,\infty)$, $(\alpha_\epsilon)\subset\R$, $(s_\eps) \subset T_{x_\eps, \lambda_\eps}$ and $(r_\eps) \subset T_{x_\eps, \lambda_\eps}^\bot$ such that
\begin{equation}
\label{expansion psi + q conj}
u_\epsilon = \alpha_\epsilon(\psi_{x_\eps, \lambda_\eps} + s_\eps + r_\eps)
\end{equation}
and a point $x_0 \in \Omega$ such that, in addition to Proposition \ref{prop first expansion conj},
\begin{align}
\|\nabla r_\eps\|_2 &= \mathcal{O}\big(\eps +\lambda_\eps^{-3/2} +\phi_a(x_\eps)\, \lambda_\eps^{-1} \big) \,, \label{r-eps-bound conj}\\
\phi_a(x_\eps) & = \pi\,a(x_\eps)\,\lambda_\eps^{-1} + \frac{\pi}{32}\,\eps\lambda_\eps \left( 1 + o(1) \right) + o(\lambda_\eps^{-1}) \,, \label{eq-ph-tris-prop} \\
\nabla \phi_a(x) &= \mathcal O\left( \eps\lambda_\eps^{1/2} + \lambda_\eps^{-\mu} + \phi_a(x_\eps) \, \lambda_\eps^{-1/2} \right) \qquad\text{for any}\ \mu<1 \,, \label{eq:nablaphaconj} \\
\alpha_\eps^{4-\eps} &= 1 +\frac\eps 2\log\lambda_\eps -4\beta\lambda_\eps^{-1}
+ \mathcal O \big( \eps+ \phi_a(x_\eps) \lambda_\eps^{-1} \big) + o(\lambda_\eps^{-1})\, \label{exp-alpha-2-prop} . 
\end{align}
\end{proposition}

We will prove Proposition \ref{prop second expansion conj} through a series of propositions in the following subsections. 

\subsection{The bound on $\|\nabla r\|_2$}

The following proposition contains the bound on $\|\nabla r\|_2$ claimed in Proposition \ref{prop second expansion conj}. 

\begin{proposition} \label{prop-r-bound2}
As $\eps\to 0$,
\begin{equation}  \label{nabla-r-b}
\|\nabla r\|_2 = \mathcal{O}\big(\eps +\lambda^{-3/2} +\phi_a(x)\, \lambda^{-1} \big).
\end{equation}
\end{proposition}

\begin{proof}
Notice that 
\begin{equation*}
-\Delta r = -3 U_\xl^5 +3\alpha^{4-\eps} (\psi_\xl +s+r)^{5-\eps} +a \big( g_\xl + f_\xl \big)  -a (s+r) +\Delta s, 
\end{equation*}
with $g_\xl$ as in \eqref{eq:defg}. Hence
\begin{align}  \label{nabla-r}
\int_\Omega\big (|\nabla r|^2 +a r^2  \big) & = 3\alpha^{4-\eps} \int_\Omega (\psi_\xl+s+r)^{5-\eps} r  - \int_\Omega a\big(U_\xl -\frac{\lambda^{-1/2}}{|x-y|} +s-f_\xl \big) \, r  \,.
\end{align}
By Lemma \ref{lemma expansion r}(b)  
$$
\Big | \int_\Omega a\big(g_\xl + f_\xl - s \big) \, r\, \Big|  \, \lesssim \, \lambda^{-3/2}\, \| r\|_6 \, .
$$ 
Now, 
\begin{align} \label{psi-5eps}
\int_\Omega (\psi_\xl + s + r)^{5-\eps} r &= \int_\Omega U_\xl^{5-\eps} r + (5-\eps) \int_\Omega U_\xl^{4-\eps}  r^2 + (5-\eps) \int_\Omega U_\xl^{4-\eps}  rs \notag \\
& \quad - (5-\eps) \int_\Omega U_\xl^{4-\eps} ( \lambda^{-1/2} H_a(x, \cdot) + f_\xl) r + T_{3,\eps} \, ,
\end{align}
where similarly as in the proof Lemma \ref{lemma expansion r} we find that 
$$
 | T_{3,\eps} |   \, \lesssim \,  \lambda^{-2}\, \| r\|_6  + \| r\|_6^3.
$$
Moreover, similarly as in \eqref{U4eps-w} we obtain
$$
3 \alpha^{4-\eps}\, (5-\eps) \int_\Omega U_\xl^{4-\eps}\,  r^2  \, \leq\,  15  \int_\Omega U_\xl^4 r^2\, +o ( \| r\|_6^2) .  
$$

Next, we write
$$
\int_\Omega U_\xl^{4-\eps}  rs = \lambda^{-\eps/2} \left( \int_\Omega U_\xl^{4}  rs + \int_\Omega U_\xl^4 \left( e^{\eps \log\sqrt{1+ \lambda^2|x-y|^2}} - 1 \right)rs \right). 
$$
The prefactor $\lambda^{-\eps/2}$ on the right side tends to 1 by Lemma \ref{lem-eps-log}. The first integral in the parentheses is bounded in \eqref{u4-rs}. For the second integral we proceed again as in \eqref{U4eps-w} and obtain
$$
\left| \int_\Omega U_\xl^4 \left( e^{\eps \log\sqrt{1+ \lambda^2|x-y|^2}} - 1 \right)rs \right|
\lesssim \lambda \eps \left\| U^4 |x-y| \right\|_{3/2} \|r\|_6 \|s\|_6 \lesssim \eps \lambda^{-1} \|r\|_6 \,,
$$
where we used \eqref{bounds s} in the last inequality. Thus, recalling the bound on $\eps$ in \eqref{parameters PU + w conj},
$$
\big |  \int_\Omega U_\xl^{4-\eps}  rs \, \big | \, \lesssim\, \lambda^{-3/2}\, \| r\|_6\, .
$$

The fourth term on the right side of \eqref{psi-5eps} is bounded, in absolute value, by a constant times
$$
\int_\Omega U_\xl^4 \left(\lambda^{-1/2} |H_a(x,\cdot)| + |f_\xl| \right)|r| \lesssim \left( \lambda^{-1} \phi_a(x) + \lambda^{-2} \right) \|r\|_6 \,,
$$
where we used \eqref{u4r6}.

Using Lemma \ref{lem-aux-eps} to control the first term on the right hand side of \eqref{psi-5eps} and putting all the estimates into \eqref{nabla-r} we finally get
$$
\int_\Omega\big (|\nabla r|^2 +a r^2  -15\,  U_\xl^4 r^2  \big) \, \lesssim \,  \left(  \eps+ \lambda^{-1} \phi_a(x) + \lambda^{-3/2} \right) \| r\|_6   + o(\| r\|^2_6)\, .
$$
This, in combination with the coercivity inequality of Lemma \ref{lemma coercivity}, implies the claim.
\end{proof}


\subsection{Expanding $\alpha^{4-\eps}$} 

In this subsection, we prove the expansion of $\alpha^{4-\eps}$ in Proposition \ref{prop second expansion conj}. 

\begin{proposition}\label{prop-alpha4conj}
	As $\eps \to 0$,
	\begin{equation} \label{exp-alpha-2}
		\alpha^{4-\eps} = 1 +\frac\eps 2\log\lambda -4\beta\lambda^{-1} + \mathcal O \big( \eps+ \phi_a(x) \lambda^{-1} \big) + o(\lambda^{-1}) \,.
	\end{equation} 
\end{proposition}

\begin{proof}
	As in the proof of Lemma \ref{lem-eps-log} we integrate equation \eqref{BP-problem} against $u$. However, this time we write $u=\alpha(\psi_\xl+ q)$ and obtain
	$$
	\int_\Omega |\nabla (\psi_\xl+ q)|^2 + \int_\Omega a (\psi_\xl+ q)^2 = 3 \alpha^{4-\eps} \int_\Omega (\psi_\xl+ q)^{6-\eps} \,,
	$$
	which we write as
	\begin{align}
		\label{energy-alpha4-eps}
		& \int_\Omega \left( |\nabla\psi_\xl|^2 + a\psi_\xl^2 - 3\alpha^{4-\eps} |\psi_\xl|^{6-\eps} \right) \notag \\
		& + 2 \int_\Omega \left( \nabla q \cdot\nabla\psi_\xl + aq\psi_\xl - \frac{3(6-\eps)}2 \alpha^{4-\eps} q |\psi_\xl|^{4-\eps}\psi_\xl \right) = \mathcal R_0
	\end{align}
	with
	$$
	\mathcal R_0 := - \int_\Omega \left(|\nabla q|^2 + aq^2 \right) + 3 \alpha^{4-\eps} \int_\Omega \left( (\psi_\xl + q)^{6-\eps} - |\psi_\xl|^{6-\eps} - (6-\eps) |\psi_\xl|^{4-\eps}\psi_\xl q \right).
	$$
	We discuss separately the three terms that are involved in the identity \eqref{energy-alpha4-eps}.
	
	First, we claim that
	\begin{align*}
		& \int_\Omega \left( |\nabla\psi_\xl|^2 + a\psi_\xl^2 - 3\alpha^{4-\eps}|\psi_\xl|^{6-\eps} \right) \\
		& = (1-\alpha^{4-\eps}) \frac{3 \pi^2}{4}  + \frac{3\pi^2}8 \alpha^{4-\eps} \eps \log\lambda
		+ \mathcal O( \eps + \phi_a(x) \lambda^{-1} + \lambda^{-2} ) \,.
	\end{align*}
	Indeed, this follows in the same way as in the proof of Lemma \ref{lemma alpha^4} (a), together with the fact that
	$$
	\int_\Omega \left( |\psi_\xl|^{6-\eps} - \psi_\xl^{6} \right) = - \frac{\pi^{2}}{8} \eps \log\lambda 
	+ \mathcal O( \eps +\phi_a(x) \lambda^{-1} +\lambda^{-5/2}) \,.
	$$
	To prove the latter expansion, we write $\psi_\xl = U_\xl - \lambda^{-1/2} H_a(x,\cdot)-f_\xl$ and expand, recalling \eqref{eq:uepsbound},
	\begin{align*}
		|\psi_\xl|^{6-\eps} - \psi_\xl^{6} & = U_\xl^{6-\eps} - U_\xl^{6} \\
		& \quad + \mathcal O \left( U_\xl^5 \left( \lambda^{-1/2} |H_a(x,\cdot)| +|f_\xl| \right) + \lambda^{-5/2} |H_a(x,\cdot)|^5 + |f_\xl |^5 \right).
	\end{align*}
	Using the bounds from Lemma \ref{lemma PU}, \eqref{Ha-bound} and proceeding as in the proof of Lemma \ref{lemma U Ha}, we obtain
	$$
	\int_\Omega\! \left( U_\xl^5 \left( \lambda^{-1/2} |H_a(x,\cdot)| +|f_\xl| \right) + \lambda^{-5/2} |H_a(x,\cdot)|^5 + |f_\xl |^5 \right) \!=\! \mathcal O( \phi_a(x) \lambda^{-1} + \lambda^{-5/2}).
	$$
	On the other hand, by an explicit computation,
	\begin{align*}
		\int_\Omega \left( U_\xl^{6-\eps} - U_\xl^{6} \right) = \int_{\R^3} \left( U_\xl^{6-\eps} - U_\xl^{6} \right) + \mathcal O(\lambda^{-3}) & = \pi^{3/2} \left( \lambda^{-\eps/2} \frac{\Gamma(\frac{3-\eps}2)}{\Gamma(3-\frac\eps 2)} - \frac{\Gamma(\frac32)}{\Gamma(3)} \right) + \mathcal O(\lambda^{-3}) \\
		& = - \frac{\pi^{2}}{8} \eps \log\lambda + \mathcal O( \eps +\lambda^{-3}) \,, 
	\end{align*}
	proving the claimed expansion of the first term on the left side of \eqref{energy-alpha4-eps}
	
	We turn now to the second term on the left side of \eqref{energy-alpha4-eps} and claim that
	$$
	\int_\Omega \left( \nabla q \cdot\nabla\psi_\xl + aq\psi_\xl - \frac{3(6-\eps)}2 \alpha^{4-\eps} q |\psi_\xl|^{4-\eps} \psi_\xl \right) = \left( 1-3 \alpha^{4-\eps} \right) \frac{3\pi^2}{4}\, \beta \lambda^{-1} + \mathcal O(\lambda^{-2}) \,.
	$$
	To show this, we proceed as in the proof of Lemma \ref{lemma alpha^4} (b) and use the equation for $\psi_\xl$ to write
	\begin{align*}
		& \int_\Omega \left( \nabla q \cdot\nabla\psi_\xl + aq\psi_\xl - \frac{3(6-\eps)}2 \alpha^{4-\eps} q |\psi_\xl|^{4-\eps}\psi_\xl \right) = 3 \left( 1 - \frac{6-\eps}{2} \alpha^{4-\eps}\right) \int_\Omega q U_\xl^5 \\
		& \quad - \frac{3(6-\eps)}{2} \int_\Omega q \left( U_\xl^{5-\eps} - U_\xl^5\right)
		- \int_\Omega q \left( \frac{3(6-\eps)}{2} (|\psi_\xl|^{4-\eps}\psi_\xl - U_\xl^{5-\eps}) + a (f_\xl + g_\xl) \right).
	\end{align*}
	The first term on the right side was already computed in the proof of Lemma \ref{lemma alpha^4} (b) and the last term on the right side can be bounded in the same way as there, except that now, instead of \eqref{eq:qbound}, we use the bound
	\begin{equation}
		\label{eq:qboundconj}
		\|\nabla q\|_2 \lesssim \lambda^{-1} \,,
	\end{equation}
	which follows from the bounds on $s$ and $r$ in Propositions \ref{proposition s} and \ref{nabla-r-b}. For the second term on the right side we proceed as in the proof of Lemma \ref{lem-aux-eps} and obtain
	$$
	\left| \int_\Omega q \left( U_\xl^{5-\eps} - U_\xl^5 \right) \right| \lesssim \eps \lambda^{1-\eps/2} \int_\Omega |q| U_\xl^5 |x-y| \leq \eps \lambda^{1-\eps/2} \left\| U^5 |x-y| \right\|_{6/5} \|q\|_6 \lesssim \eps \|q\|_6 \lesssim \eps \lambda^{-1} \,.
	$$
	By Proposition \ref{boundoneps}, this is $\mathcal O(\lambda^{-2})$.	
	
	Finally, we bound $\mathcal R_0$, the term on the right side of \eqref{energy-alpha4-eps}. Because of \eqref{eq:qboundconj}, the first integral in the definition of $\mathcal R_0$ is $\mathcal O(\lambda^{-2})$. The second integral is bounded, in absolute value, by a constant times
	$$
	\int_\Omega \left( |\psi_\xl|^{4-\eps} q^2 + |q|^{6-\eps} \right) \lesssim \left\| \psi_\xl \right\|_6^{4-\eps} \|q\|_6^2 + \|q\|_6^{6-\eps}  \lesssim \lambda^{-2} \,.
	$$
	
	Inserting all the bounds in \eqref{energy-alpha4-eps}, we obtain the claimed bound.
\end{proof}


\subsection{Expanding $\phi_a(x)$}

In this subsection we prove the following important expansion.

\begin{proposition} \label{eps-lambda-phi-BP}
	As $\eps \to 0$, 
	\begin{align}   \label{eq-ph-tris}
		\phi_a(x) = \pi\,a(x)\,\lambda^{-1} + \frac{\pi}{32}\, \eps \lambda \left(1+o(1) \right) + o(\lambda^{-1}) \,. 
	\end{align}
\end{proposition}

The proof of this proposition, which is the analogue of Proposition \ref{phiaexp}, is a refined version of the proof of Proposition \ref{boundoneps}. We integrate equation \eqref{BP-problem} for $u$ against $\partial_\lambda \psi_{x,\lambda}$ and we write the resulting equality in the form
\begin{align}\label{phi-a-0} 
	& \int_\Omega \left( \nabla \psi_\xl \cdot \nabla \pl\psi_\xl + a \psi_\xl \pl\psi_\xl -3 \alpha^{4-\eps} |\psi_\xl|^{4-\eps}\psi_\xl \pl\psi_\xl \right) \notag \\
	& = - \int_\Omega \left( \nabla q \cdot \nabla \pl\psi_\xl + a q \pl\psi_\xl - 3(5-\eps) \alpha^{4-\eps} q |\psi_\xl|^{4-\eps} \pl\psi_\xl \right) \notag \\[4pt]
	& \quad + \frac{3(5-\eps)(4-\eps)}2 \alpha^{4-\eps} \int_\Omega q^2 |\psi_\xl|^{2-\eps} \psi_\xl \pl\psi_\xl + \mathcal R
\end{align}
with
$$
\mathcal R \!=\! 3\alpha^{4-\eps} \!\int_\Omega \!\! \big( (\psi_\xl+q)^{5-\eps} \!- |\psi_\xl|^{4-\eps} \psi_\xl - \!(5-\eps) |\psi_\xl|^{4-\eps} q - \frac{(5-\eps)(4-\eps)}2 |\psi_\xl|^{2-\eps}\psi_\xl q^2 \big) \pl\psi_\xl.
$$

\begin{lemma}\label{lemma pohozaev bisconj}
	As $\eps\to 0$, the following holds.
	\begin{enumerate}
		\item[(a)]
		$ \displaystyle 
		\begin{aligned}[t]
			& \int_\Omega \left( \nabla \psi_\xl \cdot \nabla \pl\psi_\xl + a \psi_\xl \pl\psi_\xl -3 \alpha^{4-\eps}|\psi_\xl|^{4-\eps}\psi_\xl \pl\psi_\xl \right) \\
			& = -  2\pi\, \phi_a(x)\, \lambda^{-2} \left(1+o(1)\right) + \frac{\pi^2}{16}\, \eps\lambda^{-1}  \left(1+o(1)\right)  +2\pi^2 a(x) \lambda^{-3}  + o(\lambda^{-3})  \,.
		\end{aligned}$
		\item[(b)]
		$ \displaystyle
		\begin{aligned}[t]
			& \int_\Omega \left( \nabla q \cdot \nabla \pl\psi_\xl + a q \pl\psi_\xl - 3(5-\eps) \alpha^{4-\eps} q |\psi_\xl|^{4-\eps} \pl\psi_\xl \right) \\
			& = - \big( 1- \alpha^{4-\eps} \big) 2\pi \left( \phi_a(x) - \phi_0(x) \right) \lambda^{-2} + \mathcal O(\eps \lambda^{-2} \log\lambda + \phi_a(x)\, \lambda^{-3}) + o(\lambda^{-3}) \,.
		\end{aligned}$
		\item[(c)]
		$\displaystyle \int_\Omega q^2 |\psi_\xl|^{2-\eps} \psi_\xl \pl\psi_\xl = \frac{\pi^2}{32}\,\beta\gamma\, \lambda^{-3} + \mathcal O(\eps\lambda^{-2} + \phi_a(x)\,\lambda^{-3}) + o(\lambda^{-3}) \,.$
		\item[(d)]
		$\displaystyle \mathcal R = o(\lambda^{-3})$		
	\end{enumerate}
\end{lemma}

The proof of Lemma \ref{lemma pohozaev bisconj} is independent of the expansion of $\alpha^{4-\eps}$ in Proposition \ref{prop-alpha4conj}. We only use the fact that $\alpha= 1+ o(1)$.

\begin{proof}
	(a) As in the proof of Lemma \ref{lemma pohozaev bis} (a), see equation \eqref{eq:expnp1}, we have
	\begin{align*}
		& \int_\Omega \left( \nabla \psi_\xl \cdot \nabla \pl\psi_\xl + a \psi_\xl \pl\psi_\xl -3 	\alpha^{4-\eps} |\psi_\xl|^{4-\eps}\psi_\xl \pl\psi_\xl \right) \\
		& = 3 \int_\Omega \left( U_\xl^5 -  \alpha^{4-\eps} |\psi_\xl|^{4-\eps}\psi_\xl \right) \pl\psi_\xl
		- \int_\Omega a (f_\xl + g_\xl) \pl\psi_\xl \,.
	\end{align*}
	The second integral on the right side was shown in the proof of Lemma \ref{lemma pohozaev bis} (a) to satisfy
	$$
	\int_\Omega a (f_\xl + g_\xl) \pl\psi_\xl = 2\pi\left( 3- \pi \right) a(x) \lambda^{-3} + o(\lambda^{-3}) \,.
	$$
	We write the first integral on the right side as
	\begin{align}
		& \int_\Omega \left( U_\xl^5 -  \alpha^{4-\eps} |\psi_\xl|^{4-\eps}\psi_\xl \right) \pl\psi_\xl
		= \left( 1 - \alpha^{4-\eps} \right) \int_\Omega U_\xl^5 \pl\psi_\xl \nonumber  \\
		& \qquad\qquad - \alpha^{4-\eps} \int_\Omega \left( U_\xl^{5-\eps} - U_\xl^5 \right)\pl\psi_\xl 
		- \alpha^{4-\eps} \int_\Omega \left(|\psi_\xl|^{4-\eps}\psi_\xl - U_\xl^{5-\eps}\right)\pl\psi_\xl.      \label{psi-u-5-eps}
	\end{align}
	As shown in the proof of Lemma \ref{lemma pohozaev bis} (a),
	$$
	\int_\Omega U_\xl^5 \pl\psi_\xl = \frac{2\pi}3\phi_a(x) \lambda^{-2} + \mathcal O(\lambda^{-3}) \,.
	$$
	Next, by Lemma \ref{lemma PU},
	$$
	\int_\Omega \!\big( U_\xl^{5-\eps} - U_\xl^5 \big)\pl\psi_\xl = \int_\Omega \!\big( U_\xl^{5-\eps} - U_\xl^5 \big)\pl U_\xl + \frac12 \lambda^{-3/2} \int_\Omega \!\big( U_\xl^{5-\eps} - U_\xl^5 \big) H_a(x,\cdot) + o(\lambda^{-3}) \,.
	$$
	For the first term, we use \eqref{u5-eps} and the bounds from the proof of Lemma \ref{lemma pohozaev bis} (a) to get
	$$
	\int_\Omega \left( U_\xl^{5-\eps} - U_\xl^5 \right)\pl U_\xl = - \frac{\pi^2}{48} \eps \lambda^{-1} (1+o(1)) + \mathcal O(\lambda^{-4}) \,.
	$$
	For the second term, we use the bound $\|U_\xl^{-\eps}-1\|_\infty = \mathcal{O}(\eps\log\lambda)$ and compute
	$$
	\lambda^{-3/2} \left| \int_\Omega \left( U_\xl^{5-\eps} - U_\xl^5 \right) H_a(x,\cdot) \right| \lesssim \eps\lambda^{-3/2} \log\lambda \int_\Omega U_\xl^5 H_a(x,\cdot) \lesssim \eps\lambda^{-2} \log\lambda = o(\eps\lambda^{-1}) \,.
	$$
	
	Concerning the last term on the right hand side of \eqref{psi-u-5-eps}, we will prove	
	\begin{equation}
		\label{eq:psiueps}
		\int_\Omega \! \big(|\psi_\xl|^{4-\eps}\psi_\xl - U_\xl^{5-\eps}\big)\pl\psi_\xl \!=\! \frac{2\pi}3\,\phi_a(x) \lambda^{-2} \left( 1 + o(1) \right) - 2\pi\,a(x)\,\lambda^{-3} + \mathcal O(\phi_a(x)^2\,\lambda^{-3}) + o(\lambda^{-3}).
	\end{equation}
	This will complete our discussion of the right hand side of \eqref{psi-u-5-eps} and hence the proof of (a). 
	
	The proof of \eqref{eq:psiueps} is similar to the corresponding argument in the proof of Lemma \ref{lemma pohozaev bis} (a), but we include some details. We bound pointwise
	\begin{align*}
		|\psi_\xl|^{4-\eps}\psi_\xl - U_\xl^{5-\eps} & = -(5-\eps) \lambda^{-\frac 12}\, U_\xl^{4-\eps} \, H_a(x,\cdot) + \tfrac12(5-\eps)(4-\eps) \lambda^{-1} U_\xl^{3-\eps} \, H_a(x,\cdot)^2   \\
		& \quad +
		\mathcal{O} \Big ( \lambda^{-3/2} U_\xl^2 |H_a(x,\cdot)|^3 + \lambda^{-5/2} |H_a(x,\cdot)|^5 + U_\xl^4 |f_\xl| + |f_\xl|^5 \Big).
	\end{align*}
	Using the bounds from Lemmas \ref{lemma Lq norm of U} and \ref{lemma PU}, we easily find that the remainder term, when integrated against $|\pl\psi_\xl|$ is $o(\lambda^{-3})$. Using  expansion \eqref{expansion Ha new} we obtain, by an explicit calculation similar to  \eqref{U^4 pl U Hb} and \eqref{U^4 Hb^2},
	\begin{align*}
		&\int_\Omega  U_\xl^{4-\eps} H_a(x,\cdot)\pl\psi_\xl    =  \int_\Omega  U_\xl^{4-\eps} \pl U_\xl  H_a(x,\cdot) + \mathcal{O}(\lambda^{-5/2}\, \phi_a(x)^2) +o(\lambda^{-5/2}) \\
		& \qquad =- \left( \frac{2\pi}{15} + \mathcal O(\eps) \right) \phi_a(x)\, \lambda^{-\frac{3+\eps}{2}} + \frac{2\pi}{5}\, a(x)\, \lambda^{-\frac{5}{2}}  + \mathcal{O}(\lambda^{-5/2}\,  \phi_a(x)^2)+ o(\lambda^{-5/2}) \\
		& \qquad =- \frac{2\pi}{15} \, \phi_a(x)\, \lambda^{-\frac{3}{2}} \left( 1+ o(1) \right) + \frac{2\pi}{5}\, a(x)\, \lambda^{-\frac{5}{2}}  + \mathcal{O}(\lambda^{-5/2}\,  \phi_a(x)^2)+ o(\lambda^{-5/2}) \,,
	\end{align*} 
	where we used Lemma \ref{lem-eps-log}. In the same way, we get
	$$
	\int_\Omega   U_\xl^{3-\eps}  H_a(x,\cdot)^2 \pl\psi_\xl   = \mathcal{O}(\lambda^{-2}\, \phi^2_a(x)) +o(\lambda^{-2}) .
	$$
	This proves \eqref{eq:psiueps}.

(b) As in the proof of Lemma \ref{lemma pohozaev bis} (b) we have
	\begin{align*}
		& \int_\Omega \left( \nabla q \cdot \nabla \pl\psi_\xl + a q \pl\psi_\xl -3  (5-\eps)	\alpha^{4-\eps} |\psi_\xl|^{4-\eps} q \pl\psi_\xl \right) \\
		& = 3 \int_\Omega q \left( 5 U_\xl^4 \pl U_\xl - (5-\eps) \alpha^{4-\eps} |\psi_\xl|^{4-\eps} \pl\psi_\xl \right)	- \int_\Omega a q \left(\pl f_\xl + \pl g_\xl \right).
	\end{align*}
	According to \eqref{negligible}, the second term on the right side is $o(\lambda^{-3})$. (Note that we now use the bound \eqref{eq:qboundconj} instead of \eqref{eq:qbound}.) We write the first integral as
	\begin{align*}
		& \int_\Omega q \left( 5 U_\xl^4 \pl U_\xl - (5-\eps) \alpha^{4-\eps} |\psi_\xl|^{4-\eps} \pl\psi_\xl \right)	 = \left( 5\left( 1 - \alpha^{4-\eps}\right) + \eps \alpha^{4-\eps} \right) \int_\Omega q U_\xl^4 \pl U_\xl \\
		& \quad + (5-\eps) \alpha^{4-\eps} \int_\Omega q \left( U_\xl^4 \pl U_\xl - \psi_\xl^{4} \pl\psi_\xl \right) + (5-\eps) \alpha^{4-\eps} \int_\Omega q \left( \psi_\xl^4  - |\psi_\xl|^{4-\eps}\right) \pl\psi_\xl .
	\end{align*}
	According to \eqref{qU^4-U},
	\begin{align*}
		& \left( 5\left( 1 - \alpha^{4-\eps}\right) + \eps \alpha^{4-\eps} \right) \int_\Omega q U_\xl^4 \pl 	U_\xl \\
		& =  \left( 5\left( 1 - \alpha^{4-\eps}\right) + \eps \alpha^{4-\eps} \right) \left( - \frac{2\pi}{15} \left( \phi_a(x) - \phi_0(x) \right) \lambda^{-2} + \mathcal O(\lambda^{-3}) \right) \\
		& = - \frac{2\pi}{3} \big( 1 - \alpha^{4-\eps}\big) \left( \phi_a(x) - \phi_0(x) \right) \lambda^{-2} + \mathcal O(\eps \lambda^{-2}) + o(\lambda^{-3})
	\end{align*}
	and according to \eqref{u^4-psi^4},	using \eqref{eq:qboundconj} instead of \eqref{eq:qbound},
	\begin{align*}
		\int_\Omega q \left( U_\xl^4 \pl U_\xl - \psi_\xl^{4} \pl\psi_\xl \right) = \mathcal O (\phi_a(x)\lambda^{-3}) + o(\lambda^{-3})
	\end{align*}
	Finally, for any fixed $\delta \in (0,d(x))$ and for any $p>1$ we have, by Lemma \ref{lemma PU},
	\begin{equation} \label{ball-cp} 
		\| \psi_\xl^p\,  \pl \psi_\xl \|_{L^\infty(B_\delta(x)^c\cap \Omega)} =  \mathcal{O}\big(\lambda^{-\frac{3+p}{2}} \big) \,.
	\end{equation} 
	On the other hand, taking $\delta$ sufficienctly small (but independent of $\eps$) we obtain $U_\xl \lesssim \psi_\xl \lesssim U_\xl$ on $B_\delta(x)$. The latter implies $\psi_\xl^{-\eps} =U_\xl^{-\eps}(1+ \mathcal{O}\big(\eps))$ on $B_\delta(x)$, and therefore 
	$$
	\|  1- \psi_\xl^{-\eps} \|_{L^\infty(B_\delta(x))} = \mathcal{O}\big(\eps \log \lambda \big) .
	$$
	Consequently, using \eqref{eq:qboundconj} and \eqref{ball-cp}, 
	\begin{align*}
		\left| \int_\Omega q \left( \psi_\xl^4  - |\psi_\xl|^{4-\eps}\right)  \pl\psi_\xl  \right| \lesssim \|q\|_6 \left( \eps \log\lambda\ \| \psi_\xl^4 \pl\psi_\xl \|_{6/5} + \lambda^{-\frac{7}{2}} \right) \lesssim \eps \lambda^{-2} \log\lambda + \lambda^{-\frac{9}{2}} \,.
	\end{align*}
	Collecting all the bounds, we arrive at the claimed expansion in (b).
	
	(c) The relevant term with exponent $2-\eps$ replaced by $2$ was computed in Lemma \ref{lemma pohozaev bis} (c). The same computation, but with Proposition \ref{prop-r-bound2} instead of Proposition \ref{prop-r-bound}, gives
	$$
	\int_\Omega q^2 \psi_\xl^3 \pl\psi_\xl = \frac{\pi^2}{32}\,\beta\gamma\, \lambda^{-3} + \mathcal O(\eps \lambda^{-2} + \phi_a(x)\,\lambda^{-3}) + o(\lambda^{-3}) \,.
	$$
	(The $\mathcal O(\eps\lambda^{-2})$ term comes from bounding $\int_\Omega rs\psi_\xl^3\pl\psi_\xl$.)
	
	 We bound the difference similarly as at the end of the previous part (b), namely,
	\begin{align*}
		\left| \int_\Omega q^2 \left( |\psi_\xl|^{2-\eps}\psi_\xl - \psi_\xl^3\right) \pl\psi_\xl \right| & \lesssim \|q\|_6^2 \left( \eps \log\lambda\, \|\psi_\xl^3 \pl\psi_\xl \|_{3/2} + \lambda^{-3}\right) \\
		& \lesssim \eps \lambda^{-3}\log\lambda + \lambda^{-5} = o(\lambda^{-3}) \,.
	\end{align*} 
	
	The proof of (d) uses similar bounds as in the rest of the proof and is omitted.
\end{proof}

\begin{proof}[Proof of Proposition \ref{eps-lambda-phi-BP}] 
	Inserting the bounds from Lemma \ref{lemma pohozaev bisconj} into \eqref{phi-a-0}, we obtain
	\begin{align*}
		\phi_a(x) \left(1+o(1)\right) - \frac{\pi}{32}\, \eps\lambda  \left(1+o(1)\right)  -\pi a(x) \lambda^{-1} - \left( 1- \alpha^{4-\eps} \right) \phi_0(x)
		+ \frac{15\,\pi}{32}\,\beta\gamma\, \lambda^{-1} = o(\lambda^{-1}) \,.
	\end{align*}
	Inserting the expansion of $\alpha^{4-\eps}$ from Proposition \ref{prop-alpha4conj}, this becomes
	$$
	\phi_a(x) \left(1+o(1)\right) - \frac{\pi}{32}\, \eps\lambda  \left(1+o(1)\right)  -\pi a(x) \lambda^{-1} - 4\beta \, \phi_0(x)\, \lambda^{-1}
	+ \frac{15\,\pi}{32}\,\beta\gamma\, \lambda^{-1} = o(\lambda^{-1}) \,.
	$$
	Using the expansions \eqref{eq beta gamma subsec} of $\beta$ and $\gamma$, this can be simplified to
	$$
	\phi_a(x) \left(1+o(1)\right) - \frac{\pi}{32}\, \eps\lambda  \left(1+o(1)\right)  -\pi a(x) \lambda^{-1} = o(\lambda^{-1}) \,,
	$$
	which is the assertion.  
\end{proof}


\subsection{Bounding $\nabla\phi_a$} In this subsection we prove the bound on $\nabla \phi_a(x)$ in Proposition \ref{prop second expansion conj}. 

\begin{proposition}  \label{prop-nabla-phi-a}
For every $\mu<1$, as $\eps\to 0$,  
\begin{equation}\label{bound-grad-phi2}
|\nabla \phi_a(x)| \lesssim \eps \lambda^{1/2} + \lambda^{-\mu} + \phi_a(x) \, \lambda^{-1/2}\, .
\end{equation}
\end{proposition}

Note that together with \eqref{parameters PU + w conj} it follows from Proposition \ref{prop-nabla-phi-a} that $x_0$ is a critical point of $\phi_a$. 

The proof of Proposition \ref{prop-nabla-phi-a} is a refined version of the proof of Proposition \ref{prop-ebc2} and is again based on the Pohozaev identity  \eqref{pohoz-1}. The latter reads, in the notation of \eqref{I-def},
\begin{align}\label{pohoz-3}
	0 = I[\psi_\xl] + 2\, I[\psi_\xl, q] + I[q] \, .
\end{align}
To control the boundary integrals involving $q$ in this identity, we need the following lemma, which is the analogue of Lemma \ref{lemma q bdry integral}.

\begin{lemma}
	\label{lem-q-bdry}
	$\|\frac{\partial q}{\partial n}\|_{L^2(\partial \Omega)}  \lesssim \eps +  \lambda^{-3/2} +\phi_a (x)\, \lambda^{-1}$.
\end{lemma}

Before proving this lemma, let us use it to complete the proof of Proposition \ref{prop-nabla-phi-a}. In that proof, and later in this subsection, we will use the inequality
\begin{equation}\label{eq:qbound2conj}
	\|q\|_2 \,\lesssim\, \eps + \lambda^{-3/2} + \phi_a(x) \, \lambda^{-1}\, ,
\end{equation}
This follows from the bound \eqref{bounds s} on $s$ and the bound in Proposition \ref{prop-r-bound2} on $r$.

\begin{proof}[Proof of Proposition \ref{prop-nabla-phi-a}]
	It follows from Lemma \ref{lem-q-bdry} and the bounds \eqref{eq:qbound2conj} and \eqref{eq:psibounds2} that
	$$
	\left| I[\psi_\xl, q] \right| \lesssim \eps \lambda^{-1/2} +  \lambda^{-2} + \phi_a(x)\,\lambda^{-3/2} \,,
	\qquad
	\left| I[q] \right| \lesssim \eps^2 + \lambda^{-3} + \phi_a(x)^2\,\lambda^{-2} \,.
	$$
	The claim thus follows from Lemma \ref{lemma I psi} and \eqref{pohoz-3}.
\end{proof}

\begin{proof}[Proof of Lemma \ref{lem-q-bdry}] 
Note that $-\Delta q = F$ with
$$
F:=  -3U_\xl^5 + 3 \alpha^{4-\eps} (\psi_\xl + q)^{5-\eps} - aq + a (f_\xl + g_\xl)  \,.
$$
With the cut-off function $\zeta$ defined as in the proof of Lemma \ref{lemma w bdry integral}, we have
$$
-\Delta(\zeta q) = \zeta F - 2\nabla\zeta \cdot\nabla q - (\Delta\zeta)q  \,.
$$
Arguing as in \eqref{eq:ptwf} we deduce that 
\begin{equation}
	\label{eq:ptwf2}
	\zeta | F| \lesssim \zeta |q|^{5-\eps} +  |q| + \lambda^{-5/2} \,.
\end{equation}
Now we follow the line of arguments in the proof of Lemma \ref{lemma q bdry integral}. The only difference is that instead of \eqref{eq:qbound2} we have the bound 
\begin{equation} 
\|q\|_2 \,\lesssim\, \eps +  \lambda^{-3/2} + \phi_a(x) \, \lambda^{-1}\, ,
\end{equation}
which follows from \eqref{bounds s} and Proposition \ref{prop-r-bound2}. Using this estimate we find 
$$
\| \Delta (\zeta q )\|_{3/2} \, \lesssim\, \eps +   \lambda^{-3/2} + \phi_a(x) \, \lambda^{-1}\, .
$$
In combination with  \eqref{trace estimate w}, this proves the claim. 
\end{proof}


\section{Proof of Theorems \ref{thm expansionconj} and \ref{thm BPconj}}
\label{sec:proofssubcrit}


\subsection{Proof of Theorem \ref{thm expansionconj}}

Equation \eqref{u-eps-finalconj} follows from Proposition \ref{prop first expansion conj}, together with \eqref{definition q}, \eqref{q-split} and \eqref{eq:sproj}. Proposition \ref{prop first expansion conj} gives also $|x_\eps - x_0| = o(1)$. Moreover, the bound on $\lambda$ in \eqref{parameters PU + w conj} together with \eqref{eq:nablaphaconj} gives $\nabla \phi_a (x_0) = 0$, and \eqref{r-eps-bound conj} gives $\|\nabla r\|_2= \mathcal O(\eps + \lambda^{-3/2}+ \phi_a(x)\lambda^{-1})$. By the bound on $\lambda$ in \eqref{parameters PU + w conj}, this proves the claimed bound on $\|\nabla r\|_2$ if $\phi_a(x_0)\neq 0$. In case $\phi_a(x_0)=0$ we will see below that $\phi_a(x)=o(\lambda^{-1})$ and $\eps= \mathcal O(\lambda^{-2})$, so we again obtain the claimed bound.

Next, equation \eqref{eq-ph-tris-prop} shows that
\begin{equation} \label{exp-phi-bp}
\lim_{\eps\to 0}\eps\lambda = \frac{32}{\pi}\, \phi_a(x_0)    \, ,
\end{equation} 
which is \eqref{lim eps lambdaconj}. 

Equation \eqref{alpha-asympconj} follows from \eqref{exp-alpha-2-prop}. In case $\phi_a(x_0)\neq 0$ this is immediate, and in case $\phi_a(x_0) = 0$ we use, in addition, the expansion of $\beta$ from Proposition \ref{prop-beta-gamma} and the fact that $\eps=o(\lambda^{-1})$ by \eqref{exp-phi-bp}.

Finally, let us assume $\phi_a(x_0) = 0$ and prove \eqref{lim-eps-lambda^2}. We apply Lemma  \ref{lemma hessian} to the function $u(x) := \phi_a(x + x_0)$ and get $\phi_a(x) \lesssim  |\nabla \phi_a(x)|^2$. From \eqref{eq:nablaphaconj}, together with the fact that $\eps=o(\lambda^{-1})$ by \eqref{exp-phi-bp}, we then get 
\begin{equation} \label{phi-final}
	\phi_a(x) = o(\lambda^{-1}) \,.
\end{equation}
Inserting this into \eqref{eq-ph-tris-prop}, we obtain
$$
\pi\,a(x)\,\lambda^{-1} + \frac{\pi}{32}\,\eps\lambda \left( 1+ o(1) \right) = o(\lambda^{-1}) \,,
$$
which is \eqref{lim-eps-lambda^2}. This completes the proof of Theorem \ref{thm expansionconj}.


\subsection{A bound on $\|w\|_\infty$ }\label{subsection infty bound wconj}

To complete the proof of Theorem \ref{thm BPconj} it remains to establish a suitable bound on $\|w\|_\infty$, as well as on $\|w\|_p$ for $p > 6$. This is provided by the following modification of Proposition \ref{proposition infty bound w}.

\begin{proposition} \label{lem-w-infty}
As $\eps \to 0$, 
\begin{equation}
\label{wp bound conj}
\|w\|_{p} \lesssim \lambda^{-\frac{3}{p}} \qquad \text{ for every } \quad p \in (6, \infty).
\end{equation}
Moreover, for every $\mu > 0$, 
\begin{equation}
\label{w infty bound conj}
\|w\|_\infty = o(\lambda^\mu).
\end{equation} 
\end{proposition}

\begin{proof}
To prove the bound \eqref{wp bound conj}, let $r > 1$ and $F$ given by \eqref{F conj}. As in the proof of Proposition \ref{proposition infty bound w}, we obtain the same bound \eqref{estimate int fw}, where similarly to \eqref{f eq 2}, $F$ satisfies 
\begin{equation}
\label{f-eps-upperb}
|F| \, \lesssim\,  U_\xl^{5-\eps}\, |\alpha^{4-\eps} - 1| + |U_\xl^{5-\eps}-U_\xl^5|+  U_\xl^4 (|w| + \varphi_\xl) + |w|^5 + \varphi_\xl +U_\xl+|w| \, .
\end{equation}
Using the bounds $\eps \lesssim \lambda^{-1}$ from Proposition \ref{prop first expansion conj} and $|\alpha^{4-\eps} - 1| \lesssim \eps \log \lambda$ by Proposition \ref{prop-alpha4conj}, we can estimate, for every $r > 1$, 
\begin{align*}
& \int_\Omega \left(U_\xl^{5-\eps}\, |\alpha^{4-\eps} - 1| + |U_\xl^5- U_\xl^{5-\eps}| \right) |w|^r \\
&  \lesssim \|w\|_{3(r+1)}^r \left( \|U_\xl^{5-\eps}\|_\frac{3r+3}{2r+3} |\alpha^{4-\eps} - 1| + \|U_\xl^5 - U_\xl^{5-\eps}\|_\frac{3r+3}{2r+3} \right) \\
& \lesssim \|w\|_{3(r+1)}^r \eps \log \lambda \|U_\xl\|_{5 \cdot \frac{3r+3}{2r+3}}^5 \lesssim  \|w\|_{3(r+1)}^r \eps \log \lambda \,  \lambda^{\frac12 \frac{r-1}{r+1}} \\
&  \leq \eta \|w\|_{3(r+1)}^{r+1} + C_\eta (\log \lambda)^{r+1} \lambda^{-\frac{r+3}{2}} \leq \eta \|w\|_{3(r+1)}^{r+1} + C_\eta \lambda^{-1}, 
\end{align*}
Hence the right side of \eqref{estimate int fw} fulfills the same estimate as in the proof of Proposition \ref{proposition infty bound w}, and we conclude \eqref{wp bound conj} as there.

We now turn to the bound \eqref{w infty bound conj}. 
From \eqref{eq-w-bp} we deduce that
\begin{equation}
\label{equation-w-eps}
w (x) = \frac{1}{4 \pi} \int_\Omega G_0(x,y) F(y) , 
\end{equation}
As in Proposition \ref{proposition infty bound w}, we need to estimate $\|F\|_q$ for some $q > 3/2$ using \eqref{f-eps-upperb}. We bound
\[ 
\|U_\xl^{5-\eps}|\alpha^{4-\eps} - 1|\|_{q} \ \lesssim  \ (\eps \log \lambda+\lambda^{-1}) \|U_\xl\|_{5q}^5 \lesssim    \lambda^{\frac{3}{2} - \frac{3}{q}} \log \lambda
 \]
for every $q > 3/2$. Similarly, 
$$
\|U_\xl^{5-\eps}- U_\xl^5\|_q \, \lesssim\, \eps \log \lambda \|U_\xl\|_{5q}^5 \lesssim \lambda^{\frac{3}{2} - \frac{3}{q}} \log \lambda
$$
for every $q > 3/2$. The other terms resulting from \eqref{f-eps-upperb} are identical to those already estimated in Proposition \ref{proposition infty bound w}. As there, we thus obtain $\|F\|_q \lesssim \lambda^{2 - \frac{3}{q}} \log \lambda$. Letting $q \searrow 3/2$ yields \eqref{w infty bound conj}.
\end{proof}


\subsection{Proof of Theorem \ref{thm BPconj}}

At this point, the proof of Theorem \ref{thm BPconj} is almost identical to the proof of Theorem \ref{thm BP}. We provide some details nevertheless. 

By the bound $\|w\|_\infty = o(\lambda^{1/2})$ from Proposition \ref{lem-w-infty} and Proposition \ref{prop first expansion}, we have $\|u_\eps\|_\infty = \lambda^{1/2} + o(\lambda^{1/2})$. Thus part (a) of Theorem \ref{thm BPconj} follows from \eqref{lim eps lambdaconj} and \eqref{lim-eps-lambda^2}, respectively. 

To prove part (b), we rewrite equation \eqref{equation u} as 
\[ u(z) = \frac{3}{4\pi} \int_\Omega G_a(z,y) u(y)^{5-\eps}. \]
Fix again  $\delta = \delta_\eps = o(1)$ with $\lambda^{-1}  = o(\delta_\eps)$, so that $\frac{3}{4 \pi} \int_{B_{\delta_\eps(x)}} u(y)^5 = 1 + o(1)$. Then 
\[ \frac{3}{4 \pi} \int_{B_{\delta}(x)} G_a(z,y) u(y)^5 =\frac{3}{4 \pi} \int_{B_{\delta}(x)} (G_a(z,x_0)+o(1)) u(y)^5 = \lambda^{-\frac{1}{2} - \frac{\epsilon}{2}} G_a(z, x_0) + o(\lambda^{-\frac{1}{2} - \frac{\epsilon}{2}}). \]
On the other hand, by Lemmas \ref{lem-w-infty} and \ref{lemma Lq norm of U},
\[ |\int_{\Omega \setminus B_{\delta}(x)} G_a(z,y) u(y)^{5-\eps}| \lesssim \|G_a(z, \cdot)\|_2 (\|U_\xl\|^{5-\eps}_{L^{10}(\Omega \setminus B_\delta(x))} + \|w\|_{10}^{5-\eps}) \lesssim \lambda^{-5/2} \delta^{-7/2} + \lambda^{-3/2}. \]
Choosing $\delta = \lambda^{-c}$ with $c>0$ small enough and  observing that $\lambda^{-\eps/2} = 1+o(1)$ by Lemma \ref{lem-eps-log}, the proof of part (b) of Theorem \ref{thm BPconj} is complete. 


\appendix
 
\section{Some useful bounds}

In this section, we collect some bounds which will be of frequent use in our estimates. 

\begin{lemma}
\label{lemma Lq norm of U}
Let $x \in \Omega$ and let $1 \leq q < \infty$. As $\lambda \to \infty$, we have 
\begin{equation}
\label{U qnorm bound}
\|U_\xl\|_{L^q(\Omega)} \lesssim 
\begin{cases}
\lambda^{-1/2}, & 1 \leq q < 3, \\
\lambda^{-1/2}\, (\log\lambda) ^{\frac 13} & q = 3, \\
\lambda^{\frac{1}{2} - \frac{3}{q}}, & q > 3. 
\end{cases}
\end{equation} 

Moreover, we have $\pxi U_\xl (y) = \lambda^{5/2} \frac{y_i - x_i}{(1 + \lambda^2 |x-y|^2)^{3/2}}$ with 
\[ \|\pxi U_\xl\|_{L^q(\Omega)} \lesssim 
\begin{cases}
\lambda^{-1/2}, & 1 \leq q < 3/2, \\
\lambda^{-1/2}\, (\log\lambda)^{\frac 23}, & q = 3/2,  \\
\lambda^{\frac{3}{2} - \frac{3}{q}}, & q > 3/2. 
\end{cases}
\]
and $\pl U_\xl(y) = \frac 12 \lambda^{-1/2} \frac{1 - \lambda^2|x-y|^2}{(1 + \lambda^2 |x-y|^2)^{3/2}}$ with
\[ \|\pl U\|_q \leq \lambda^{-1} \|U\|_q \quad \text{ for any } \quad 1 \leq q \leq \infty. \]
Moreover, for any $\rho = \rho_\lambda$ with $\rho \lambda \to \infty$, 
\[ \| U \|_{L^q(\Omega \setminus B_\rho(x))} \lesssim 
\begin{cases}
\lambda^{-1/2}, & 1 \leq q < 3, \\
\lambda^{-1/2}\, (\log\lambda)^{\frac 13}, & q = 3, \\
\lambda^{-\frac{1}{2}} \rho^{\frac{3-q}{q}}, & q > 3, 
\end{cases} \]
and 
\[ \| \pl U \|_{L^q(\Omega \setminus B_\rho(x))} \lesssim 
\begin{cases}
\lambda^{-3/2}, & 1 \leq q < 3, \\
\lambda^{-3/2}\, (\log\lambda)^{\frac 13}, & q = 3, \\
\lambda^{-\frac{3}{2}} \rho^{\frac{3-q}{q}}, & q > 3,
\end{cases} \]
and 
\[ \|\pxi U \|_{L^q(\Omega \setminus B_\rho(x))} \lesssim 
\begin{cases}
\lambda^{-1/2}, & 1 \leq q < 3/2, \\
\lambda^{-1/2}\, (\log\lambda)^{\frac 23}, & q = 3/2, \\
\lambda^{-\frac{1}{2}} \rho^{\frac{3-2q}{q}}, & q > 3/2. 
\end{cases} \]
\end{lemma}

\begin{proof}
Taking $R > 0$ such that $\Omega \subset B_R(x)$, we have 
\begin{align*}
\int_\Omega U_\xl^q \lesssim \lambda^{-3 + \frac{q}{2}} \int_0^{\lambda R} \frac{r^{2}}{(1 + r^2)^{q/2}} \lesssim \lambda^{-3 + \frac{q}{2}} \int_1^{\lambda R} r^{2-q} \lesssim 
\begin{cases}
\lambda^{-q/2}, & 1 \leq q < 3, \\
\lambda^{-q/2}\, (\log\lambda) ^{\frac 13} & q = 3, \\
\lambda^{\frac{q}{2} - 3}, & q > 3. 
\end{cases}
\end{align*}
This proves \eqref{U qnorm bound}. The remaining bounds follow by analogous explicit computations, which we omit. 
\end{proof}

\begin{lemma}
\label{lemma PU}
We have 
\[ U_\xl = PU_\xl + \lambda^{-1/2} H_0(x, \cdot) + f_\xl, \]
with 
\begin{equation} \label{f-upperbs}
\|f_\xl\|_\infty \lesssim \lambda^{-5/2} d^{-3}, \quad 
\|\pl f_\xl\|_\infty \lesssim \lambda^{-7/2} d^{-3}, \quad
\|\pxi f_\xl\|_\infty \lesssim \lambda^{-5/2} d^{-4}.
\end{equation}
The function $\varphi_\xl := \lambda^{-1/2} H_0(x, \cdot) + f_\xl$ satisfies $0 \leq \varphi_\xl \leq U_\xl$ as well as 
\begin{equation} \label{phi-upperbs}
	\| \varphi_\xl\|_6 \lesssim \lambda^{-1/2} d^{-1/2}, \quad
	\|\varphi_\xl\|_\infty \lesssim \lambda^{-1/2} d^{-1} \, .
\end{equation} 
Moreover,
$$
\|\pl \varphi_\xl \|_6 \lesssim \lambda^{-3/2}d^{-1/2}, \quad \|\pl \varphi_\xl \|_\infty \lesssim \lambda^{-3/2} d^{-1}
$$
and
$$
\quad  \|\pxi \varphi_\xl\|_6 \lesssim \lambda^{-1/2} d^{-1/2}, \quad \|\pxi \varphi_\xl\|_\infty \lesssim \lambda^{-1/2} d^{-2}. 
$$
\end{lemma}

\begin{proof} 
	Everything, except for the $L^\infty$ bounds on $\varphi_\xl$, $\pxi\varphi_\xl$ and $\pl\varphi_\xl$, is taken from \cite[Prop.~1]{Re2}. Since these functions are harmonic, the remaining bounds follow from the maximum principle.
\end{proof} 

\begin{lemma}
	\label{lemma PU bdry integral}
	We have
	\begin{enumerate}
		\item[(a)]
		$\int_{\partial \Omega} n \left( \frac{\partial PU_\xl}{\partial n} \right)^2  = C \lambda^{-1} \nabla \phi_0(x) + o(\lambda^{-1} d^{-2})$ for some constant $C > 0$,
		\item[(b)] $\int_{\partial \Omega} y \cdot n \left( \frac{\partial PU_\xl}{\partial n} \right)^2  = \mathcal O(\lambda^{-1} d^{-2})$, 
		\item[(c)] $\int_{\partial \Omega} \left( \frac{\partial PU_\xl}{\partial n} \right)^2  = \mathcal O(\lambda^{-1} d^{-2})$. 
	\end{enumerate}
\end{lemma}

For the proof of Lemma \ref{lemma PU bdry integral} we refer to \cite[Eq.(2.7)]{Re2},   \cite[Eq.(2.10)]{Re2}, and \cite[Eq.(B.25)]{Re2} respectively.

We define the function
\begin{equation}
	\label{eq:defg}
	g_\xl(y) := \frac{\lambda^{-1/2}}{|x-y|} - U_\xl(y) \,,
\end{equation}

\begin{lemma}
\label{lem-g}
As $\lambda \to \infty,$
$$
\| g_\xl \|_p \lesssim \lambda^{1/2-3/p}
\qquad
\|\pl g_\xl \|_p \lesssim \lambda^{-1/2-3/p}
$$
hold if $1\leq p<3$. Moreover, $\nabla g_\xl \in L^p(\R^3)$ for all $1 \leq p < 3/2$.
\end{lemma}

\begin{proof}
We have $g_\xl(y) = \lambda^{1/2} g_{0,1}(\lambda(x-y))$ with $g_{0,1}(z) = |z|^{-1} - (1+|z|^2)^{-1/2}$. As $|z| \to \infty$, 
\[  g_{0,1}(z) = |z|^{-1} \left(1 - (1 + |z|^{-2})^{-1/2} \right) \lesssim |z|^{-3}. \]
Hence $g_{0,1} \in L^p(\R^3)$ for all $1 \leq p < 3$, which yields $\|g_\xl\|_p \leq \lambda^{1/2-3/p} \|g_{0,1}\|_{L^p(\R^3)}$. 

Next, by direct calculation, 
\[ \nabla g_{0,1}(z) = -\frac{z}{|z|^3} + \frac{z}{(1+ |z|^2)^{3/2}} \lesssim |z|^{-4} \qquad \text{ as } |z| \to \infty. \]
Hence $\nabla g_{0,1} \in L^p(\R^3)$ for all $1 \leq p < 3/2$ and so is $\nabla g_\xl = \lambda^{3/2} (\nabla g_{0,1})(\lambda(x-y))$. 

Finally, we observe
\[ \pl g_\xl(y) = \lambda^{-1} g_\xl + \lambda^{1/2} (x-y) \cdot (\nabla g_{0,1})(\lambda(x-y)) . \]
By the above, we have $z \cdot \nabla g_{0,1} \in L^p(\R^3)$ for all $1 \leq p < 3$ and thus 
\[ \|\pl g_\xl \|_p \leq \lambda^{-1} \|g_\xl\|_p + \lambda^{-\frac12 - \frac{3}{p}} \|z \cdot \nabla g_{0,1}\|_{L^p(\R^3)} \]
for all $1 \leq p < 3$. 
\end{proof}

\section{Properties of the functions $H_a(x, y)$}
\label{section properties Ha}

In this appendix, we prove some properties of  $H_a(x,y)$ needed in the proofs of the main results. Since these properties hold  independently of the criticality of $a$, we state them for a generic function $b$ which satisfies the same regularity conditions as $a$, namely,
$$
b \in C(\overline{\Omega})\cap C_{\rm loc}^{2,\sigma}(\Omega) \qquad \text{for some} \quad 0<\sigma<1\, .
$$
(In fact, in Subsection \ref{subsection basic Ha} we only use $b \in C(\overline{\Omega})\cap C_{\rm loc}^{1,\sigma}(\Omega)$ for some $0<\sigma<1$.) In addition, we assume  that  $-\Delta+b$ is coercive in $\Omega$ with Dirichlet boundary conditions. Note that the choice $b=0$ is allowed. 


\subsection{Estimates on $H_b(x, \cdot)$}
\label{subsection basic Ha}

We start by recalling the bound 
\begin{equation} \label{Ha-bound}
\|H_b(x, \cdot) \|_\infty\, \lesssim\, d(x)^{-1} \qquad \forall\, x\in\Omega,
\end{equation} 
see \cite[Eq. (2.6)]{FrKoKo1}. We next prove a similar bound for the derivatives of $H_b(x, \cdot)$. 

\begin{lemma}
\label{lemma nabla Hb bounds}
Let $x, y\in \Omega$ with $x\neq y$. Then $\nabla_x H_b(x,y)$ and $\nabla_y H_b(x,y)$ exist and satisfy 
\begin{align}
\label{nabla x Hb bound}
\sup_{y \in \Omega \setminus \{x\}} |\nabla_x H_b(x,y)| &\leq C, \\
\label{nabla y Hb bound}
\sup_{y \in \Omega \setminus \{x\}} |\nabla_y H_b(x,y)| &\leq C
\end{align}
with $C$ uniform for $x$ in compact subsets of $\Omega$. 
\end{lemma}

\begin{proof}
\textit{Step 1. } We first prove the bounds for the special case $b=0$, which we shall need as an ingredient for the general proof. Since $H_0(x,\cdot)$ is harmonic, we have $\Delta_y \nabla_y H_0(x,y) =0$. Moreover, we have the bound $\nabla_y G_0(x, y) \lesssim |x-y|^{-2}$ uniformly for $x,y \in \Omega$ \cite[Theorem 2.3]{Wi}. This implies that for $x$ in a compact subset of $\Omega$ and for $y \in \partial \Omega$,
\[ |\nabla_y H_0(x,y)| = |\nabla_y (|x-y|^{-1}) - \nabla_y G_0(x,y)| \leq C \,. \]
We now conclude by the maximum principle. 

The proof for the bound on $\nabla_x H_0(x,y)$ is analogous, but simpler, because $\nabla_x G_0(x,y) = 0$ for $y \in \partial \Omega$. 

\textit{Step 2.  } For general $b$, we first prove the bounds for both $x$ and $y$ lying in a compact subset of $\Omega$. By \cite[proof of Lemma 2.5]{FrKoKo1} we have 
\[ H_b(x,y) = \phi_b(x) + \Psi_x(y) - \frac{b(x)}{2}|y-x| \]
with $\|\Psi_x\|_{C^{1,\mu}(K)} \leq C$ for every $0 < \mu < 1$ and every compact subset $K$ of $\Omega$, and with $C$ uniform for $x$ in compact subsets. This shows that $|\nabla_y H_b(x,y)| \leq C$ uniformly for $x,y$ in compact subsets of $\Omega$. By symmetry of $H_b$, this also implies $|\nabla_x H_b(x,y)| \leq C$ uniformly for $x,y$ in compact subsets of $\Omega$.
 
\textit{Step 3.  } We complete the proof of the lemma by treating the case when $x$ remains in a compact subset, but $y$ is close to the boundary. In particular, we may assume for what follows
 \begin{equation}
 \label{x and y far}  |x-y|^{-1} \lesssim 1. 
\end{equation} 

By the resolvent formula, we write
\[ H_b(x,y) = H_0(x,y) + \frac{1}{4 \pi} \int_\Omega G_0(x,z) b(z) G_b(z,y) \, dz. \]
By Step 1, the derivatives of $H_0(x,y)$ are uniformly bounded.

We thus only need to consider the integral term. Its $\pxi$-derivative equals
\begin{align*} & \int_\Omega \pxi (\frac{1}{|x-z|}) b(z) G_b(z,y) \, dz - \int_\Omega \pxi H_0(x,z) b(z) G_b(z,y) \, dz \\
&\lesssim \int_\Omega \frac{1}{|x-z|^2} \frac{1}{|z-y|} \,dz + 1 \lesssim \frac{1}{|x-y|^2} + 1 \lesssim 1
\end{align*}
where we again used the fact that \eqref{nabla x Hb bound} holds for $b =0$, together with \eqref{x and y far}. This completes the proof of \eqref{nabla x Hb bound}. 

The proof of \eqref{nabla y Hb bound} can be completed analogously. It suffices to write the resolvent formula as 
\[ H_b(x,y) = H_0(x,y) + \frac{1}{4 \pi} \int_\Omega G_b(x,z) b(z) G_0(z,y) \, dz \]
in order to ensure that the $\pyi$-derivative falls on $G_0$  and we can use \eqref{nabla y Hb bound} for $b=0$. 
\end{proof}

We now prove an expansion of $H_b(x,y)$ on the diagonal which improves upon \cite[Lemma 2.5]{FrKoKo1}.

\begin{lemma}
\label{lemma Ha expansion new}
Let $0<\mu<1$. If $y \to x$,  then uniformly for $x$ in compact subsets of $\Omega$, 
\begin{equation}
\label{expansion Ha new} H_b(x,y) = \phi_b(x) + \frac 12 \nabla \phi_b(x) \cdot (y-x) -\frac{b(x)}{2} |y-x| + \mathcal O(|y-x|^{1+\mu}) \,.
\end{equation}
\end{lemma}

\begin{proof}
In \cite[Lemma 2.5]{FrKoKo1}, it is proved that 
\begin{equation}
\label{definition psi2} \Psi_x(y) := H_b(x,y) - \phi_b(x) + \frac{b(x)}{2} |y-x| 
\end{equation}
is in $C^{1,\mu}_\text{loc}(\Omega)$ (as a function of $y$) for any $\mu<1$. Thus, by expanding $\Psi_x(y)$ in near $y = x$, 
\begin{equation}
\label{Ha prelim} H_b(x,y) = \phi_b(x) + \nabla \Psi_x(x) \cdot (y-x) - \frac{b(x)}{2} |y-x| + \mathcal O(|y-x|^{1+\mu}) \,.
 \end{equation}
 This gives \eqref{expansion Ha new} provided we can show that for each fixed $x\in\Omega$,
\begin{equation}
\label{nabla phi x}
\nabla \Psi_x(x) = \frac 12 \nabla \phi_b(x) \,. 
\end{equation} 
Indeed, by using \eqref{Ha prelim} twice with the roles of $x$ and $y$ exchanged, subtracting and recalling $H_b(x,y) = H_b(y,x)$, we get
\begin{align} \label{diffquot phi a} 
	\phi_b(y) - \phi_b(x) 
	&= (\nabla \Psi_y(y) + \nabla \Psi_x(x)) (y-x) + \frac{b(y) - b(x)}{2} |x-y| + \mathcal O(|x-y|^{1 + \mu}) \notag \\ 
	&= (\nabla \Psi_y(y) + \nabla \Psi_x(x)) (y-x) + \mathcal O(|x-y|^{1 + \mu}) \,,
\end{align}
because $b \in C^{0, \mu}_\text{loc}(\Omega)$. We now argue that $\Psi_y \to \Psi_x$ in $C_{\rm loc}^1(\Omega)$, which implies $\nabla \Psi_y(y) \to \nabla \Psi_x(x)$. Together with this, \eqref{nabla phi x} follows from \eqref{diffquot phi a}.

To justify the convergence of $\Psi_y$ we argue similarly as in \cite[Lemma 2.5]{FrKoKo1}. We note that $-\Delta_z \Psi_y = F_y(z)$ with
$$
F_y(z) := \frac{b(z)-b(y)}{|z-y|} - b(z) H_b(y,z) \,.
$$
We claim that $F_y \to F_x$ in $L^p_{\rm loc}(\Omega)$ for any $p<\infty$. Indeed, the first term in the definition of $F_y$ converges pointwise to $F_x$ in $\Omega\setminus\{x\}$ and is locally bounded, independently of $y$, since $b\in C_{\rm loc}^{0,1}(\Omega)$. Thus, by dominated convergence it converges in $L^p_{\rm loc}(\Omega)$ for any $p<\infty$. Convergence in $L^\infty_{\rm loc}(\Omega)$ of the second term in the definition of $F_y$ follows from the bound on the gradient of $H_b$ in Lemma \ref{lemma nabla Hb bounds}
. This proves the claim.

By elliptic regularity, the convergence $F_y\to F_x$ in $L^p_{\rm loc}(\Omega)$ implies the convergence $\Psi_y \to \Psi_x$ in $C^{1,1-3/p}_{\rm loc}(\Omega)$. This completes the proof.
\end{proof}

\begin{lemma}
\label{lemma U Ha}
For any  $x\in\Omega$ we have, as $\lambda\to\infty$,
\begin{align}
\int_\Omega U_\xl^5 H_b(x,\cdot)  &= \frac{4 \pi }{3} \phi_b(x) \lambda^{-1/2} - \frac{4 \pi }{3} b(x) \lambda^{-3/2} + o(\lambda^{-3/2}), \label{U^5 Hb}\\
\int_\Omega U_\xl^4 \pl U_\xl H_b(x,\cdot)  &= - \frac{2}{15} \pi \phi_b(x) \lambda^{-3/2} + \frac{2}{5} \pi b(x) \lambda^{-5/2} + o(\lambda^{-5/2}), \label{U^4 pl U Hb} \\
\int_\Omega U_\xl^4 \pxi U_\xl H_b(x,\cdot) &= \frac{2\pi}{15} \nabla \phi_b(x) \lambda^{-1/2} + o(\lambda^{-1/2}), \label{U^4 pxi Hb} \\
\int_\Omega U_\xl^4 H_b(x,\cdot)^2  &= \pi^2 \phi_b(x)^2 \lambda^{-1} + o(\lambda^{-1}), \label{U^4 Hb^2} \\
\int_\Omega U_\xl^3 \pl U_\xl H_b(x,\cdot)^2  &= - \frac{\pi^2}{4}  \phi_b(x)^2 \lambda^{-2}  + o(\lambda^{-2}). \label{U^3 U Hb^2}
\end{align}
The implied constants can be chosen uniformly for $x$ in compact subsets of $\Omega$.
\end{lemma}

\begin{proof}
Equalities \eqref{U^5 Hb} and \eqref{U^4 Hb^2} are proved in \cite[Lemmas 2.5 and 2.6]{FrKoKo1}. To prove \eqref{U^4 pl U Hb}, we write 
\begin{equation}\label{pl-U}
 \pl U_\xl  =  \frac{U_\xl}{2\lambda} - \lambda^{3/2}\, \frac{ |x-y|^2}{(1+\lambda^2\, |x-y|^2)^{3/2}}\, ,
\end{equation}
and therefore, using \eqref{U^5 Hb}, 
\begin{align*}
  \int_\Omega  H_b(x,y) \,  U^4_\xl \, \pl U_\xl & =   \frac 23 \pi  \phi_b(x) \lambda^{-3/2} - \frac 23 \pi b(x) \lambda^{-5/2} \\
&\quad  - \lambda^{7/2} \int_\Omega H_b \frac{ |x-y|^2}{(1+\lambda^2\, |x-y|^2)^{7/2}} +o(\lambda^{-5/2}) .
\end{align*}
With the help of \eqref{expansion Ha new} and the bound \eqref{Ha-bound} we get 
\begin{align*}
 \int_\Omega H_b \frac{ |x-y|^2}{(1+\lambda^2\, |x-y|^2)^{7/2}}   &= 4\pi   \phi_b(x)  \lambda^{-5} \int_0^\infty \frac{t^4\, dt}{(1+t^2)^{7/2}}  - 2\pi b(x) \lambda^{-6} \int_0^\infty \frac{t^5\, dt}{(1+t^2)^{7/2}} +o(\lambda^{-6}) \\
& = \frac{4}{5} \pi  \phi_b(x) \lambda^{-5} - \frac{16}{15}\pi b(x) \lambda^{-6} +o(\lambda^{-6}).
\end{align*}
Combining the last two equations gives \eqref{U^4 pl U Hb}. 

For the proof of \eqref{U^3 U Hb^2} use again \eqref{pl-U}, but now we use \eqref{U^4 Hb^2} instead of \eqref{U^5 Hb}. The constant comes from
$$
\int_0^\infty \frac{t^4\,dt}{(1+t^2)^{3}} = \frac{3\pi}{16} \,.
$$
We omit the details.

For the proof of \eqref{U^4 pxi Hb} we use the explicit formula for $\pxi U_\xl$ in Lemma \ref{lemma Lq norm of U}. We split the integral into $B_d(x)$ and $\Omega\setminus B_d(x)$. In the first one, we used the bound \eqref{Ha-bound} and the expansion \eqref{expansion Ha new}. By oddness, the contribution coming from $\phi_a(x)$ cancels, as does the contribution from $\sum_{k\neq i} \partial_k\phi_b(x) (y_k-x_k)$. For the remaining term we use
$$
\int_{B_d(x)} U_\xl^4(y) \pxi U_\xl(y) (y_i-x_i) = \frac{4\pi}{3} \lambda^{-1/2} \int_0^{\lambda d} \frac{t^4\,dt}{(1+t^2)^{7/2}} = \frac{4\pi}{15} \lambda^{-1/2} + \mathcal O(\lambda^{-5/2}) \,.
$$
As similar computation shows that the contribution from the error $|x-y|^{1+\mu}$ on $B_d(x)$ is $\mathcal O(\lambda^{-1/2-\mu})$. Finally, the bounds from Lemma \ref{lemma Lq norm of U}, show that the contribution from $\Omega\setminus B_d(x)$ is $\mathcal O(\lambda^{-5/2})$. This completes the proof.
\end{proof}

\begin{remark}\label{rem:U^4 pxi Hb}
	The proof just given shows that \eqref{U^4 pxi Hb} holds with the error bound $\mathcal O(\lambda^{-1/2-\mu})$ for any $0<\mu<1$ instead of $o(\lambda^{-1/2})$.
\end{remark}


\subsection{$C^2$ differentiability of $\phi_a$. }
\label{section c2}

In this subsection, we prove Lemma \ref{lemma C2}. The argument is independent of criticality of $a$ and we give the proof for a general function $b \in C^{0,1}(\overline{\Omega}) \cap C^{2, \sigma}_\text{loc}(\Omega)$ for some $0 < \sigma < 1$. The following argument is similar to \cite[Lemma 2.5]{FrKoKo1}, where a first-order differentiability result is proved, and to \cite[Lemma A.1]{dPDoMu}, where it is shown that $\phi_b \in C^\infty(\Omega)$ for constant $b$.

Let  
\begin{equation}
\Psi(x,y) := H_b(x,y) + \frac 14 \left(b(x) + b(y)\right) |x-y|, \qquad (x,y) \in \Omega \times \Omega.  
\end{equation} 
Then $\phi_b(x) = \Psi(x,x)$, so it suffices to show that $\Psi \in C^2(\Omega \times \Omega)$. 

Using $-\Delta_y |x-y| = -2 |x-y|^{-1}$ and $-\Delta_y H_b(x,y) = b(y) G_b(x,y)$, we have
\begin{align*}
 -\Delta_y \Psi(x,y) &= -b(y) H_b(x,y) - \frac{1}{2}  \frac{b(x) - b(y) - \nabla b(y) \cdot (x-y)}{|x-y|} - \frac14 \Delta b(y) |x-y|. 
\end{align*}
Since $b \in C^{2, \sigma}_\text{loc}(\Omega)$ and since $H_b$ is Lipschitz by Lemma \ref{lemma nabla Hb bounds}, the right side is in $C^{0, \sigma}_\text{loc}(\Omega)$ as a function of $y$. By elliptic regularity, $\Psi(x,y)$ is in $C^{2, \sigma}_\text{loc}(\Omega)$ as a function of $y$. Since $\Psi(x,y)$ is symmetric in $x$ and $y$, we infer that $\Psi(x,y)$ is in $C^{2, \sigma}_\text{loc}(\Omega)$ as a function of $x$.

It remains to justify the existence of mixed derivatives $\pyj \pxi \Psi(x,y)$. For this, we carry out a similar elliptic regularity argument for the function $\pxi \Psi(x,y)$. We have 
\begin{align*}
 -\Delta_y \pxi \Psi(x,y) &= -b(y) \pxi H_b(x,y) - \frac 14 \Delta b(y) \frac{x_i-y_i}{|x-y|} - \frac 12 \frac{\partial_i b(x) - \partial_i b(y)}{|x-y|} \\
 & \quad + \frac 12 \frac{x_i - y_i}{|x-y|^3} \left( b(x)-b(y) - \nabla b(y) \cdot (x-y) \right). 
\end{align*} 
Since $b \in C^{1,1}_\text{loc}(\Omega)$, and since $\pxi H_b$ is bounded by Lemma \ref{lemma nabla Hb bounds}, the right side is in $L^\infty_\text{loc}(\Omega)$ as a function of $y$. By elliptic regularity, $\pxi \Psi(x,y) \in C^{1, \mu}(\Omega)$ for every $\mu < 1$, as a function of $y$. In particular, the mixed derivative $\pyj \pxi \Psi(x,y)$ is in $C^{0,\mu}_\text{loc}(\Omega)$ as a function of $y$. By symmetry, the same argument shows that the mixed derivative $\pxj \pyi \Psi(x,y)$ is in $C^{0,\mu}_\text{loc}(\Omega)$ as a function of $x$.

The proof of Lemma \ref{lemma C2} is therefore complete. 


\end{document}